\documentclass[reqno,12p]{amsart}
\usepackage[T1]{fontenc}
\usepackage{geometry}
\usepackage{mathtools,amssymb,amsthm,mathrsfs,color,lineno,paralist,graphicx,float}
\usepackage[colorlinks,
linkcolor=red,
anchorcolor=green,
citecolor=blue,
]{hyperref}
\usepackage{amsmath}
\usepackage{pdfpages}
\usepackage{amstext}
\usepackage{color}
\usepackage{stmaryrd}
\usepackage{bbm}
\usepackage{comment}
\usepackage{faktor,url}
\usepackage{xfrac}
\usepackage{mathabx}
\usepackage{enumitem}
\usepackage{mathtools,amssymb,amsthm,mathrsfs,color,lineno,paralist,graphicx,float}

\def\blue{\color{blue}}

\def\black{\color{black}}

\definecolor{orange}{rgb}{1,0.5,0}

\definecolor{green}{rgb}{0.5,0.7,0.3}

\numberwithin{equation}{section}
\numberwithin{figure}{section}
\theoremstyle{plain}
\newtheorem{thm}
{\protect\theoremname}[section]
  
  \theoremstyle{definition}
  \newtheorem{defn}[thm]{\protect\definitionname}
\newtheoremstyle{boldremark}
    {\dimexpr\topsep/2\relax} % space above
    {\dimexpr\topsep/2\relax} % space below
    {}          % body font
    {}          % indent amount
    {\bfseries} % theorem head font
    {.}         % punctuation after theorem head
    {.5em}      % space after theorem head
    {}          % theorem hed spec. (empty = "normal")
    
\theoremstyle{boldremark}
\newtheorem{rem}{Remark} % remarks are numbered within sections

  \theoremstyle{plain}
  \newtheorem{lem}[thm]{\protect\lemmaname}
  \theoremstyle{plain}
  
  \theoremstyle{plain}
\newtheorem{prop}[thm]{\protect\propositionname}
  \theoremstyle{plain}
  \newtheorem{cor}[thm]{\protect\corollaryname}

\newtheorem{thmx}{Theorem}
\renewcommand{\thethmx}{\Alph{thmx}} % "letter-numbered" theorems

  \providecommand{\definitionname}{Definition}
  \providecommand{\lemmaname}{Lemma}
  \providecommand{\propositionname}{Proposition}
  
\providecommand{\theoremname}{Theorem}
\providecommand{\corollaryname}{Corollary}

\newcommand{\bs}{\boldsymbol}

\renewcommand{\tt}{\tau}
\newcommand{\OLDtau}{g}

\newcommand{\de}{{\delta}}
\newcommand{\eps}{{\varepsilon}}

\renewcommand{\th}{\theta}

\def\CC{{\mathbb C}}

\newcommand{\cI}{{\mathcal I}}

\newcommand{\C}{\mathcal C}
\newcommand{\cU}{{\mathcal U}}
\newcommand{\cV}{{\mathcal V}}

\newcommand{\cA}{{\mathcal A}}
\newcommand{\bA}{{\bar \Lambda}}
\newcommand{\hA}{{\hat \Lambda}}

\def\<{\langle}
\def\>{\rangle}

\title{On the density of Lyapunov unstable elliptic equilibria}

\author[B. Fayad]{B. Fayad}
\address[BF]{Department of Mathematics, University of Maryland, 4176 Campus Drive, 20782, MD, USA}
\email{bassam@umd.edu}

\author[J. Paradela]{J. Paradela}
\address[JP]{ Department of Mathematics, University of Maryland, 4176 Campus Drive, 20782, MD, USA}
\email{paradela@umd.edu}

\author[M. Saprykina]{M. Saprykina}
\address[MS]{Department of Mathematics, KTH Royal Institute of Technology, SE-114 28 Stockholm, Sweden}
\email{masha@kth.se}

\author[T. M. Seara]{T.M. Seara}
\address[TS]{Departament de Matem\`atiques, Universitat Polit\`ecnica de Catalunya, Diagonal 647, 08028 Barcelona, Spain \& Centre de Recerca Matemàtica, Edifici C, Campus Bellaterra, 08193 Bellaterra, Spain}
\email{tere.m-seara@upc.edu }

\begin{document}

\begin{abstract}  We prove that any real-analytic Hamiltonian in five or more degrees-of-freedom, with a locally integrable  non-degenerate elliptic equilibrium with indefinite quadratic part, can be perturbed within the real analytic category, while preserving the Birkhoff normal form at the equilibrium  up to any arbitrary order, so that the equilibrium becomes Lyapunov unstable. 
\end{abstract}

\maketitle

\tableofcontents

\section{Introduction and main results}\label{sec:intro}

The question of Lyapunov stability of elliptic equilibria is one of the oldest and most fundamental problems in Hamiltonian dynamics. Let $H:(\mathbb R^{2d},0)\to\mathbb R$ be a Hamiltonian having an elliptic equilibrium at the origin with frequency vector $\omega\in\mathbb R^d$. The linearized dynamics is completely stable, being symplectically conjugate to a product of rotations. However, describing  the nonlinear dynamics near the equilibrium is considerably more intricate due to \textit{resonances} between linear and nonlinear effects.

When the frequency vector is non-resonant, i.e. rationally independent, Birkhoff normal form theory shows that the Hamiltonian is formally conjugate to an integrable system. As a consequence, nearby orbits remain close to the equilibrium for times that are superpolynomial, and under suitable arithmetic assumptions even doubly exponentially long \cite{MR3606478}. Moreover, if the Birkhoff normal form satisfies a suitable non-degeneracy condition, KAM theory implies that the equilibrium is accumulated by invariant Lagrangian tori whose relative measure tends to one near the origin. Thus, %despite 
 regardless of the possible presence of instability, non-resonant elliptic equilibria exhibit remarkable effective and probabilistic stability properties.

Despite these strong stability properties, Arnold conjectured that they do not, in general, imply Lyapunov stability. More precisely, he conjectured that, apart from the two exceptional situations where the quadratic part of the Hamiltonian is sign-definite (equivalently, all frequencies have the same sign) and the two degrees of freedom case, elliptic equilibria of generic Hamiltonian systems should be Lyapunov unstable. The latter conjecture remains one of the central open problems in Hamiltonian dynamics and motivates the present work.

Arnold's conjecture  %is strongly motivated 
was inspired by the instability mechanism introduced by Arnold in his seminal work on %Arnold 
diffusion \cite{Arnolddiffusion}. The guiding principle is that resonances create channels along which instability may develop. More precisely, near resonant regions, where the frequency vector satisfies an integer relation, one expects to find chains of partially  hyperbolic invariant objects connected by transverse homoclinic or heteroclinic  channels. Orbits shadowing such chains may then travel between distant regions of phase space, producing global instability.

Over the last decades, this mechanism has been extensively developed and refined, leading to remarkable advances in the understanding of diffusion in Hamiltonian systems (see, for instance,  \cite{MR3646879,GelfreichTuraevChaotic,MR4298716,MR4033892,MR4509324,MR4729212,MR4913967,MR3479576}). However, these works are primarily concerned with constructing diffusing orbits  along a single or finite collection of resonances, typically in nearly integrable settings. The problem of proving that Lyapunov instability is a \emph{generic} phenomenon near elliptic equilibria is of a rather different nature. Indeed, one seeks to perturb an arbitrary Hamiltonian while creating a diffusion mechanism through a sequence of resonances accumulating at the (possibly non-resonant)  equilibrium.

Here the distinction between the smooth and the real-analytic categories becomes fundamental. In the smooth category, compactly supported perturbations make it possible to modify the dynamics near one resonance while leaving the others essentially unaffected. In the real-analytic category, such localization is impossible: every perturbation has global consequences, so creating the desired dynamics near one resonance necessarily affects all the others. It is this local-versus-global dichotomy, rather than the construction of diffusion itself, that constitutes the principal obstacle to establishing generic Lyapunov instability in the analytic category.

In the smooth category, Douady \cite{MR944100} proved that Lyapunov unstable elliptic equilibria are dense among smooth Hamiltonian systems. His argument naturally decomposes into two independent ingredients. The first consists in approximating an arbitrary Hamiltonian by one that is locally integrable near the elliptic equilibrium. The second establishes that Lyapunov instability is itself dense within the class of locally integrable Hamiltonians. Together, these two steps yield the density of Lyapunov instability in the smooth category.

A natural approach to the analytic category is therefore to follow the same scheme. However, each of the two steps presents substantial new difficulties. The first one already requires a density result for locally integrable real analytic Hamiltonians. In the smooth category this follows immediately from Birkhoff normal form theory, since the Taylor polynomial of sufficiently high order can be realized by a compactly supported perturbation. In contrast, in the analytic category, Birkhoff normal forms only provide approximate integrability, either on shrinking neighborhoods of the equilibrium or on shrinking complex domains. This issue is closely related to the conjectural density of locally integrable real analytic Hamiltonians, a problem for which a proof has recently been announced by Krikorian.

The main contribution of the present work is to establish the second step in the real-analytic category. More precisely, we prove that Lyapunov instability is dense within the class of locally integrable real-analytic Hamiltonians. Combined with the announced density of local integrability, this would yield the real-analytic analogue of Douady's density theorem. In other words, assuming the announced density result, Arnold's conjecture would hold in the sense of density for real-analytic Hamiltonians with at least five degrees of freedom.

The principal difficulty is that Douady's construction is intrinsically confined to the smooth category. It relies on a sequence of compactly supported perturbations, each creating a diffusion mechanism along a resonant arc while leaving the dynamics near the other resonances essentially unchanged. Since the resonant arcs accumulate at the equilibrium, this localization is an essential feature of the argument.

In the real-analytic category such a strategy is unavailable. Any perturbation introduced to create diffusion along one resonant arc necessarily modifies the Hamiltonian globally, so that subsequent perturbations must be performed on a system that is no longer integrable. Thus, unlike in the smooth category, one cannot treat the resonances independently.

Our approach overcomes this obstruction by introducing additional degrees of freedom. We first construct a  parametric  family of real-analytic three-degrees-of-freedom Hamiltonians which, for a sequence of parameter values, exhibits diffusion along resonant arcs accumulating at the origin. This family is then embedded into a locally integrable Hamiltonian with five degrees of freedom, where the additional action variables play the role of parameters selecting the relevant resonant dynamics. This parameter-embedding mechanism, inspired by the construction of the first real-analytic examples of Lyapunov unstable elliptic equilibria by Fayad \cite{MR4495839}, allows us to recover, in the analytic category, the flexibility that compactly supported perturbations provide in the smooth setting.

We believe that the parameter-embedding approach developed in this paper provides a way of reconciling the global nature of real-analytic perturbations with the local mechanisms underlying diffusion. 
We hope that this point of view will find further applications to the study of diffusion in nearly integrable and quasi-integrable real-analytic Hamiltonian systems, where similar local-versus-global difficulties naturally arise.

\section{Main statements}  \label{sec.mainstatements}

The first main result of the present work, Theorem  \ref{thm:main}, establishes the density of Lyapunov instability within the class of locally integrable real-analytic Hamiltonians and therefore makes partial progress towards Arnold's conjecture.
 We show that for $d\geq 5$, any \textit{integrable} Hamiltonian with a non degenerate elliptic fixed point $m$ (with any frequency vector) with indefinite quadratic part, and which satisfies an explicit non-degeneracy condition, can be \textit{approximated}, in the \textit{real-analytic topology}, by a Hamiltonian with the same normal form up to any prescribed order and for which $m$ is Lyapunov unstable. As a corollary of Theorem \ref{thm:main} and the Birkhoff normal form theorem, we obtain that under the same hypothesis ($d\geq 5$ and indefinite quadratic part), \textit{any} Hamiltonian with an elliptic fixed point $m$ (with any frequency vector) can be \textit{locally-approximated} by a real-analytic Hamiltonian for which $m$ is Lyapunov unstable (see Theorem \ref{thm:main2}). The main idea behind our construction is a successful implementation of Arnold's instability mechanism along a sequence of resonances approaching the (possibly non-resonant) elliptic equilibrium.

\subsection{Functional setting}
Denote by $B^d_\rho\subset \mathbb R^{d}$ the ball of radius $\rho>0$ centered around the origin and let $\mathbb B_\rho^d\subset\mathbb C^{d}$ be the closed \black polydisk of radius $\rho$ centered at the origin. 

\subsection*{Elliptic equilibria}
Fix $\rho>0$. For $d\in\mathbb N$ we define the Banach space
\begin{equation}\label{eq:banachspaceelliptintro}
\mathcal P_\rho^{d}=\{F:B_\rho^{2d}\to \mathbb C\colon F\text{ is real-analytic, extends holomorphically to $\mathbb B_\rho^{2d}$  and }\sup_{(q,p)\in \mathbb B_\rho^{2d}}|F(q,p)|<\infty\}.
\end{equation}
For $(q,p)=(q_1, \dots q_d, p_1, \dots p_d)$, denote 
$$
I = (I_1,\dots , I_{d}) ; \quad  I_j =(q_j^2+p_j^2)/2, \quad j=1,\dots ,d,
$$
and for  $\delta>0$ we define the open ball
\begin{equation}\label{eq:balllagrangian-ell}
    \mathcal B^d_{\rho}(\delta)=\{F\in \mathcal P ^d_{\rho}\colon \sup_{(q,p)\in \mathbb B_\rho^{2d}}|F(q,p)|<\delta\}.
\end{equation}

Within $\mathcal P_\rho^{d}$ we consider the set of functions $F_0\in\mathcal P_\sigma^{d}$ that only depend on $(q,p)$ through the functions 
$I_1,\dots ,  I_{d}$. We denote this closed subspace by
\[
\mathcal M_\rho^{d}=\{F_0\in \mathcal P^{d}_\rho \colon F_0(q,p)=f(I) \}\subset \mathcal P_{\rho}^{d}.
\]
In addition, given any $\omega\in \mathbb R^d$ with $\omega_i\neq 0$ for all $i=1,\dots, d$ we let
\begin{equation}\label{eq:integrableswithfreqomega-ell}
\mathcal M_\rho^d(\omega)=\{F_0\in \mathcal M^d_\rho\colon F_0= \omega\cdot I +O_2(I)\}.
\end{equation}

\subsubsection*{Dynamics}
We endow $B_\rho^{2d}\subset\mathbb R^{2d}$ with the canonical symplectic form $\mathrm{d}p\wedge \mathrm dq$. 
In this setting, the equations of motion induced by a Hamiltonian $F\in\mathcal P^{d}_\rho$ read 
\begin{equation}\label{eq:eqsofmotion}
\dot q_j(t)=\partial_{p_j} F(q(t), p(t)), \qquad\qquad \dot p_j(t)=-\partial_{q_j} F(q(t), p(t)), \quad j=1,\dots ,d.
\end{equation}
We say that this vector field has an elliptic equilibrium point at the origin  if $\nabla F(0)=0$ and the linearization of the vector field at the origin, i.e., $JD^2F(0)$, has purely imaginary  eigenvalues and is diagonalizable. This implies in particular that the eigenvalues are nonzero.

We note that for $F_0\in \mathcal M_\rho^{d}$ all the $d$-dimensional tori $\{I =\operatorname{const}\}$ are left invariant by the flow of \eqref{eq:eqsofmotion}. Hence, we call these Hamiltonians \textit{integrable}. For $F_0\in \mathcal M_\rho^d(\omega)$, the origin is an {\it elliptic equilibrium} with frequency vector $\omega$.

\subsection*{Lagrangian invariant tori}
Fix $\rho,\sigma>0$, $d\in\mathbb N$, let $\mathbb {T}^d= (\mathbb{R}/\mathbb{Z})^d$ and  
$\mathbb{T}_\sigma ^d =\{\theta\in(\mathbb C/2\pi\mathbb Z)^d\colon |\mathrm{Im} (\theta)|<\sigma\}$. Then we define the Banach space 
\begin{equation}\label{spaceQdsigmarho}
\begin{split}
\mathcal Q_{\rho,\sigma}^d=\{&F:\mathbb T^d\times B^d\to \mathbb C\colon F\text{ is real-analytic, extends holomorphically to $\mathbb T_\sigma^d\times\mathbb B_\rho^d$ and }\\
&\sup_{(\theta,I)\in \mathbb T_\sigma^d\times\mathbb B_\rho^d}|F(\theta,I)|<\infty\},
\end{split}
\end{equation}
and for  $\delta>0$ we define the open ball
\begin{equation}\label{eq:balllagrangian}
    \mathcal B^d_{\rho,\sigma}(\delta)=\{F\in \mathcal Q^d_{\rho,\sigma}\colon \sup_{(\theta,I)\in \mathbb T_\sigma^d\times\mathbb B_\rho^d}|F(\theta,I)|<\delta\}.
\end{equation}
In addition, we let ${\mathcal N}^{d}_\rho$  be the closed subset of $\mathcal Q^{d}_{\rho,\sigma}$ composed of functions which do not depend on $\theta$  and,  given any $\omega\in \mathbb R^d$ with $\omega_i\neq 0$ for all $i=1,\dots, d$ we let
\begin{equation}\label{eq:integrableswithfreqomega}
\mathcal N_\rho^d(\omega)=\{F_0\in \mathcal N^d_\rho\colon F_0= \omega\cdot I+O_2(I)\}.
\end{equation}
\subsubsection*{Dynamics} We endow $\mathbb T^d\times B^d$ with the canonical symplectic form $\mathrm dI\wedge\mathrm d\theta$. In this setting, the equations of motion induced by a Hamiltonian $F\in \mathcal Q^d_{\rho,\sigma}$ read
\begin{equation}\label{eq:vfieldlag}
\dot\theta_j(t)=\partial_{I_j} F(\theta(t),I(t)),\qquad\qquad\dot I_j=-\partial_{\theta_j} F(\theta(t),I(t)),\quad j=1,\dots,d.
\end{equation}
Notice that for $F_0\in \mathcal N^d_\rho$ all the $d$-dimensional tori
$\{I=\mathrm{const}\}$ are invariant under the flow \eqref{eq:vfieldlag}. Hence, we also call these Hamiltonians integrable. For $F_0\in \mathcal N^d_\rho(\omega)$ the dynamics on the Lagrangian invariant torus $\mathcal T=\mathbb T^d\times\{0\}$ is driven by a linear translation with frequency vector $\omega$.

\medskip

\subsection{Lyapunov instability for elliptic equilibria}

\begin{defn}
    Consider a Hamiltonian  $F\in\mathcal P_\rho^{d}$  for which the origin is an elliptic fixed point. Let  $U$ be an open neighborhood $U\subset \mathbb R^{2d}$ of the origin.  We say that the origin is $U$-Lyapunov unstable for $F$ if, for any open neighborhood $V\subset U$ of the origin,  we have  
\[
\{\phi^t_F\}_{t\in\mathbb R^+}(V)\cap  (\mathbb R^{2d}\setminus U)\neq \emptyset.
\]
\end{defn}
The following is our main result. Given $\omega\in\mathbb R^d$, recall the definition of the set $\mathcal M^d_\rho(\omega)$ in  \eqref{eq:integrableswithfreqomega-ell}.

\begin{thmx}\label{thm:main}
     Let $d\geq 5$. Fix any $\rho>0$ and let  $\omega\in \mathbb R^d$ such that $\omega_i\neq 0$ for all $i=1,\dots, d$ and  its entries are not all of the same sign. Then, for $F_0$ in an open and dense subset of $\mathcal M^{d}_\rho(\omega)$ the following holds. There exists a neighborhood $U\subset B_\rho^{2d}$ of the origin such that, for any  $\delta>0$ \black and any $l\in\mathbb N$  there exists  a real-analytic Hamiltonian $F\in \mathcal P_{\rho}^{d}$ such that:
    \begin{itemize}
    \item the origin is an elliptic fixed point for $F$,
        \item 
        $F$ is uniformly $\delta$-close \black to $F_0$ in $\mathbb B_\rho^{2d}\subset\mathbb C^{2d}$,
        \item 
        the Birkhoff normal forms of $F_0$ and $F$ at the origin coincide up to order $l$,
        \item 
        the origin is $U$-Lyapunov unstable for $F$.
    \end{itemize}
\end{thmx}
Before delving into the details of the proof, let us make a few comments. First,  we recall that a trivial argument, which makes use of energy preservation, precludes  (locally) convex Hamiltonians from being Lyapunov unstable.  
Hence, the condition on the frequency vector $\omega\in\mathbb R^d$ having not all of its entries of the same sign is a necessary condition for instability.  Second, as we will see below, our construction relies on certain invariant objects which, under suitable non-degeneracy conditions, appear at single resonances of the Hamiltonian $F_0$. These non-degeneracy conditions are explicit  and hold for an open and dense set of real-analytic integrable Hamiltonians, so any integrable $F_0$ can be slightly perturbed within this category to verify these conditions. It is worth highlighting that Theorem \ref{thm:main}, as well as Theorem \ref{thm:maintori} below, hold without any arithmetic condition on the frequency vector. Indeed, arithmetic conditions are not used in any of their proofs.

\subsection*{Density of Lyapunov unstable equilibria}
From Theorem \ref{thm:main} we trivially deduce the following approximation result by making use of the Birkhoff normal form theorem (see \cite{MR1345153}).
\begin{thmx}\label{thm:main2}
   Let $d\geq 5$. Fix any $\rho>0$  and let $F_0\in \mathcal P_\rho^{d}$ have an elliptic fixed point at the origin with rationally independent frequency $\omega\in\mathbb R^d$. Assume that the entries of the frequency vector at the origin are not all  of the same sign. 
    Then, for any $\delta>0$ and any $l\in\mathbb N$ there exists an open neighborhood of the origin $U\subset B_\rho^{2d}$ (depending on $\delta,l$) and a real-analytic Hamiltonian $F\in\mathcal P_\rho^{d}$ such that:
    \begin{itemize}
    \item the origin is an elliptic fixed point for $F$,
    \item $F$ is uniformly $\delta$-close to $F_0$ in some compact complex neighborhood of $U$.
         \item the Birkhoff normal forms of $F_0$ and $F$ at the origin coincide up to order $l$,
        \item the origin is $U$-Lyapunov unstable for $F$.
    \end{itemize} 
\end{thmx}
The proof of Theorem \ref{thm:main2} follows from Theorem \ref{thm:main} and a completely standard argument and is left to the reader.
Although Theorem \ref{thm:main2} applies to any, not necessarily integrable, real-analytic Hamiltonian (with an elliptic fixed point whose quadratic part is indefinite), the conclusion that we obtain is weaker than that in Theorem \ref{thm:main}. 
Indeed, given $\delta>0$ the corresponding Lyapunov unstable Hamiltonian $F$ is $O(\delta)$-close to $F_0$ only on a  $\delta$-dependent domain $U\subset U_0$. In particular, from this result  we cannot conclude density of Lyapunov instability within $\mathcal P_\rho^d$.
\medskip

However,   as mentioned in the introduction, density of Lyapunov instability within $\mathcal P_\rho^d$  follows from Theorem \ref{thm:main} and the conjectural density of locally integrable real analytic Hamiltonians, a problem for which a proof has recently been announced by Krikorian.

\subsection{Lyapunov instability for Lagrangian invariant tori}

\begin{defn}
Let $F\in\mathcal Q_{\rho,\sigma}^d$ and let $T=\mathbb T^d\times\{0\}$ be an invariant Lagrangian torus. Given a neighborhood $U\subset \mathbb T^d\times B_\rho^d$ of $T$, we say that $T$ is \emph{$U$-Lyapunov unstable} for $F$ if, for every neighborhood $V\subset U$ of $T$, one has
\[
\{\phi_F^t\}_{t\in\mathbb R}(V)\cap \bigl((\mathbb T^d\times B_\rho^d)\setminus U\bigr)\neq\emptyset.
\]
\end{defn}

From our construction we also deduce the following result. 

\begin{thmx}\label{thm:maintori}
    Let $d\geq 5$. Fix any $\rho,\sigma>0$ and let  $\omega\in \mathbb R^d$, $\omega_i\neq 0$ for all $i=1,\dots, d$. Then, for $F_0$ in an open and dense subset of $\mathcal N^{d}_\rho(\omega)$ the following holds.  There exists a neighborhood $U\subset \mathbb T^d\times B_\rho^{d}$ of the Lagrangian torus $\mathcal T=\mathbb T^d\times\{0\}$ such that, for any  $\delta>0$ and any $l\in\mathbb N$  there exists  a real-analytic Hamiltonian $F\in \mathcal Q_{\rho,\sigma}^{d}$ such that:
    \begin{itemize}
    \item $\mathcal T$ is invariant for $F$,
        \item $F$ is uniformly $\delta$-close to $F_0$ in $\mathbb T_\sigma^d\times \mathbb B_\rho^{d}\subset\mathbb C^{2d}$,
        \item the Birkhoff normal forms of $F_0$ and $F$ at $\mathcal T$ coincide up to order $l$,
        \item $\mathcal T$ is $U$-Lyapunov unstable for $F$.
    \end{itemize}
\end{thmx}
Notice that in Theorem \ref{thm:maintori} there is no assumption on the signs of the entries of the frequency vector at $\mathcal T$ since, in this setting, there is no topological obstruction for the existence of diffusing orbits. 
\medskip

%\medskip

%\medskip

%\medskip
\subsection{Organization of the paper}
%\noindent\textit{Organization of the paper.}
In Section  \ref{sec:proofthmtori} we prove Theorem  \ref{thm:maintori} by reducing it, via a parameter-embedding argument loosely inspired by \cite{MR4495839}, to the construction of a suitable parametric family of three-degrees-of-freedom Hamiltonians, Theorem  \ref{thm:approxbydiffusion}. In Section  \ref{sec.proof.main} we adapt this embedding argument to elliptic equilibria and reduce the proof of Theorem  \ref{thm:main} to the corresponding parametric statement, Theorem  \ref{thm:approxbydiffusionellipticpoints}.

The remainder of the paper is devoted to the proofs of Theorems  \ref{thm:approxbydiffusion} and  \ref{thm:approxbydiffusionellipticpoints}. In Section  \ref{sec:outlineparametric} we present the overall strategy of the proof of Theorem  \ref{thm:approxbydiffusion} and reduce it to two principal ingredients, namely Theorems  \ref{thm:mainparametricoutline} and  \ref{thm:_apply_GToutline}.

In Section  \ref{sec:cylinders} we use averaging theory to successively activate the Fourier mode associated with each prescribed resonance while averaging out the remaining modes, thereby constructing a sequence of normally hyperbolic cylinders. Assuming the generic transversality result established later in Section  \ref{sec:mainsptlittingsection}, we then show that an additional perturbation produces transverse homoclinic cylinders.

In Section  \ref{sec:diffusion} we apply
 the Gelfreich-Turaev implementation of Moeckel's mechanism (see \cite{GelfreichTuraevChaotic}) to construct boundary-to-boundary drifting orbits inside these homoclinic cylinders.

Section  \ref{sec:mainsptlittingsection} establishes the generic transversality result required for the construction of the homoclinic cylinders. The proof follows the parametric-transversality approach introduced by Zehnder \cite{ZehnderHomoclinicpts} and adapted to normally hyperbolic cylinders by Delshams and Zhang  \cite{MR4913967}.

Finally, in Section  \ref{sec:proofElliptic} we explain the modifications needed to adapt the proof of Theorem  \ref{thm:approxbydiffusion} to obtain Theorem  \ref{thm:approxbydiffusionellipticpoints}, thereby completing the proof of Theorem  \ref{thm:main}. The appendices contain the auxiliary results on averaging and on the straightening of unstable foliations used in Sections  \ref{sec:cylinders} and  \ref{sec:diffusion}.

\section{Proof of Theorem \ref{thm:maintori}}\label{sec:proofthmtori}

%Theorems \ref{thm:main} and \ref{thm:main2} are obtained from a  modification of the proof of Theorem \ref{thm:maintori}, which is done  in Section \ref{sec:proofElliptic}. 
In this section we show how to deduce Theorem \ref{thm:maintori} from  Theorem \ref{thm:approxbydiffusion} below, which is a statement about parametric families of  3 degrees-of-freedom Hamiltonians.

\subsection{Embedding low-dimensional parametric families into higher-dimensional Hamiltonians}\label{sec:outlineproof}

The idea, which already appears in \cite{MR4495839}\black, is based on the embedding of a judiciously designed parametric family of 3 degrees-of-freedom Hamiltonians into a single higher dimensional ($d\geq 5$) $d$ degrees-of-freedom Hamiltonian (see Sections \ref{sec:mainparametricHams} and \ref{sec:embedding}). 
A short discussion of the main ideas involved in the construction of the parametric family  in Theorem \ref{thm:approxbydiffusion} can be found in Section \ref{sec:outlineparametric}.     
We fix, once and for all, any $d\geq 5$.
\medskip

\subsection*{Functional setting and external parameters}

Below, given $0<\tilde d\black<d$ and positive constants $\rho,\sigma$, we consider Hamiltonians in action-angle coordinates $(\tilde \theta,\tilde I)\in\mathbb T_\sigma^{  \tilde d}\times \mathbb B_\rho^{ \tilde d}$, which also \textit{depend on external parameters}. We let  
\[
Q^d_{\rho,\sigma}=\mathbb T^d_\sigma\times\mathbb B^d_\rho,
\]
and introduce the Banach space
\begin{equation}\label{eq:banachspace}
\begin{split}
{\mathcal  Q}^{\tilde d, d }_{\rho,\sigma}=\{& G(\tilde \theta,\tilde I;\eta):(\mathbb T^{\tilde d}\times B^{\tilde d}) \times  B_{\rho}^
{d-\tilde d}\to\mathbb C\colon G\text{ is real-analytic, 
extends holomorphically to  $Q^{\tilde d}_{\rho,\sigma} \times \mathbb B_{\rho}^{d-\tilde d}$} \\
&\qquad\qquad \text{and } \sup_{(\tilde \theta,\tilde I;\eta)\in Q^{\tilde d}_{\rho,\sigma}\times \mathbb B_{\rho}^{d-\tilde d}}|G( \tilde \theta,\tilde I; \eta)|<\infty\}
\end{split}
\end{equation}
of functions in $\mathcal Q^{\tilde d}_{\rho,\sigma}$ depending on an external $(d-\tilde d)$-dimensional parameter $\eta$\black.  Given $\omega\in \mathbb R^{\tilde d}$ with $\omega_i\neq 0$, we let:
\begin{enumerate}
    \item ${\mathcal N}^{\tilde d,d}_\rho$ be the closed subset of $\mathcal Q^{\tilde d,d}_{\rho,\sigma}$ composed of functions that do not depend on $\theta$;
    \item ${\mathcal N}^{\tilde d,d}_\rho(\omega)=\{G_0\in {\mathcal N}^{\tilde d,d}_\rho\colon G_0=\omega\cdot \tilde I+O_2(\tilde I)\}$.
\end{enumerate}

\subsection*{A sequence of 3 degrees-of-freedom Hamiltonians with diffusion}\label{sec:mainparametricHams}

In the following theorem we construct a parametric family of 3 degrees-of-freedom Hamiltonians exhibiting drifting orbits at suitable parameter values. In order to interpret this result in connection with Theorem \ref{thm:maintori}, it is helpful to think of:
\begin{itemize}
    \item  
    $(\tilde \theta, \tilde I)\in Q^3_{\rho,\sigma}$ as the first three action-angle pairs for the Hamiltonian $F_0$ in Theorem \ref{thm:maintori}; that is $(\tilde \theta, \tilde I)=(\theta_1, \theta_2,\theta_3, I_1,I_2,I_3)$.
    \item 
    $\eta\in B_\rho^{d-3}$ as the remaining $d-3$ action variables (which are conserved quantities along the flow of the Hamiltonians that we consider below), that is $\eta=(I_4,\dots, I_d)$.
    \end{itemize}
The following approximation result is the crucial ingredient in the proof of Theorem \ref{thm:maintori}.

\begin{thmx}\label{thm:approxbydiffusion}
Fix any $\rho,\sigma>0$ and any $\omega\in\mathbb R^3$ with $\omega_i\neq 0$ for all $i=1,2,3$. For $G_0$ in an  open and dense subset of ${\mathcal N}^{3,d}_\rho(\omega)$ the following holds. There exists $\rho_0>0$  (depending only on $G_0)$  such that for any   $\delta>0$ and $l\in\mathbb N$ there exist two sequences $\{\varepsilon_n\}_n$ and $\{\mu_n\}_n$ satisfying $0<\varepsilon_n,\mu_n\leq 1/n$, and a parametric family of Hamiltonians $G_{\varepsilon,\mu}\in \mathcal Q^{3,d}_{\rho,\sigma}$ of the form 
    \begin{equation}\label{eq:parametricfamily}
    G_{\varepsilon,\mu}(\tilde \theta, \tilde I;\eta)=G_0(\tilde I;\eta)+\varepsilon g_0(\tilde\theta)+\varepsilon^2 g_1(\tilde \theta, \tilde I)+\mu g_2(\tilde \theta, \tilde I), \quad\quad (\tilde \theta, \tilde I)=(\theta_1, \theta_2,\theta_3, I_1,I_2,I_3),
    \end{equation}
    for which we have the following:
    \begin{enumerate}
   \item 
   The functions $\{g_i\}_i$ satisfy\footnote{Indeed, the function $g_0$ 
   only depends on $(\theta_1,\theta_2)$ and $g_1$ on $(\theta_1,\theta_2,I_1,I_2)$. Although this fact has no relevance at all for the proof of Theorem \ref{thm:maintori}, keeping this information in mind will prove relevant to understand the construction of the parametric family \eqref{eq:parametricfamily} (see Section \ref{sec:outlineparametric}).}
   \begin{equation}\label{eq:kappabounds}
       \sup_{(\tilde \theta, \tilde I)\in  Q^3_{\rho,\sigma}} |g_i(\tilde \theta,\tilde I)|\leq \delta, \ \ i=0,1,2 ;
   \end{equation}
   
       \item 
       If we write $\eta^{(n)}=(\varepsilon_n^{\frac{1}{l+1}}, \mu_n^{\frac{1}{l+1}},0,\dots,0)$, for any $n\in\mathbb N$ there exists an orbit 
       $\{(\tilde \theta(t;\eta^{(n)} ),\tilde I(t;\eta^{(n)}))\}_{t\in\mathbb R}$ of the Hamiltonian $G_{\varepsilon_n,\mu_n}(\tilde \theta,\tilde I;\eta^{(n)})$ and a time $T_n>0$ such that 
\[
\lVert \tilde I(0;\eta^{(n)})\rVert \leq  \frac1n,\qquad\qquad \lVert \tilde  I(T_n;\eta^{(n)})\rVert \geq \rho_0.
\]
   \end{enumerate}
\end{thmx}

\black

Before proceeding,  let us make a few remarks.   
\medskip

\noindent\textit{Non-degeneracy conditions.} Integrable Hamiltonians belonging to the open and dense set in Theorem \ref{thm:approxbydiffusion} enjoy certain non-degeneracy properties which guarantee that:
$i)$  along a certain sequence of resonant surfaces approaching $\{I=0\}$, normally hyperbolic cylinders arise after a perturbation; $ii)$ the dynamics on these normally hyperbolic cylinders is driven by a twist map.  

Both ingredients are crucial for implementing (a modern version of)  Arnold's diffusion mechanism along the aforementioned sequence of resonances. Indeed, the orbits in Theorem \ref{thm:approxbydiffusion} correspond to orbits shadowing a pseudo-orbit made of pieces of heteroclinic and inner orbits on the resonant normally hyperbolic cylinders.

\begin{rem}Although not mentioned explicitly in the statement above, the conditions defining the open and dense subset  in Theorem \ref{thm:approxbydiffusion} are explicit: see \ref{it:assumption1}-\ref{it:assumption2} (in sections \ref{sec:Resonant_surfaces}--\ref{sec:Creation_cylinders}), and Lemma \ref{lemmaH1H2}.
\end{rem}
\medskip

\medskip

\medskip
\noindent\textit{Roadmap of the proof of Theorem \ref{thm:approxbydiffusion}.}
The perturbations appearing in Theorem \ref{thm:approxbydiffusion} play distinct roles. The perturbation
\[
g_0=g_0(\theta_1,\theta_2),
\]
depending only on the angle variables, together with a suitable sequence $\{\varepsilon_n\}$, is used to create a family of normally hyperbolic invariant cylinders associated with a prescribed sequence of resonances. Next, a generic perturbation
\[
g_1=g_1(\theta_1,\theta_2,I_1,I_2),
\]
depending both on the angle and the action variables of the resonant subsystem, is used to create transverse homoclinic cylinders simultaneously for all selected resonances. Finally, a generic perturbation
\[
g_2=g_2(\theta_1,\theta_2,\theta_3,I_1,I_2,I_3),
\]
together with a suitable sequence $\{\mu_n\}$, is used to produce boundary-to-boundary drifting orbits inside the corresponding cylinders. The proof of Theorem \ref{thm:approxbydiffusion} is obtained by combining these three steps.
\medskip

\subsection{Proof of Theorem \ref{thm:maintori} assuming Theorem \ref{thm:approxbydiffusion}}\label{sec:embedding}

Theorem \ref{thm:approxbydiffusion} does not guarantee the existence of any parameter value $(\varepsilon,\mu)$ such that   $G_{\varepsilon,\mu}$ exhibits a  Lyapunov unstable invariant torus. However, we now show how to embed the parametric family $G_{\varepsilon,\mu}$ into a single, higher dimensional, real-analytic Hamiltonian to deduce Theorem \ref{thm:maintori}.

Fix some $d\geq 5$, and denote 
\[
(\theta,I)=(\tilde\theta,\hat\theta,\tilde I,\hat I),\qquad\qquad(\tilde\theta,\tilde I)=(\theta_1,\theta_2,\theta_3,I_1, I_2,I_3),\qquad\qquad (\hat\theta,\hat I)=(\theta_4,\dots,\theta_d,\dots, I_4,\dots,I_{d}).
\]
Let $F_0\in \mathcal N_{\rho}^{d}$ be the integrable Hamiltonian in Theorem \ref{thm:maintori} and  let  $G_0\in \mathcal N^{3,d}_\rho$ be the 3 degrees-of-freedom integrable Hamiltonian depending on $d-3$ parameters $\eta=\hat I$:
\[
G_0 (\tilde I; \hat I)=F_0(\tilde I,\hat I).
\]
Fix any $l\in\mathbb N$. Let $\delta>0$ be given, let 
$\{\varepsilon_n\}_n$ and $\{\mu_n\}_n$ 
be as in Theorem \ref{thm:approxbydiffusion}, and  let $g_1,g_2\in \mathcal Q_{\rho,\sigma}^3$ be the real-analytic functions obtained after applying Theorem \ref{thm:approxbydiffusion} to the integrable Hamiltonian $G_0$  and the sequence 
$\eta^{(n)}=(\varepsilon_n^{\frac{1}{l+1}},\mu_n^{\frac{1}{l+1}},0,\dots,0)$.
Then, we define the  Hamiltonian $F\in \mathcal Q^d_{\rho,\sigma}$ by the formula
\begin{equation}\label{eq:finalHam} 
F(\theta,I)=F_0(I)+I_4^{l+1} g_0(\theta_1,\theta_2)+I^{2(l+1)}_4 g_1(\theta_1,\theta_2,I_1,I_2)+ I_5^{l+1} g_2(\theta_1,\theta_2,\theta_3,I_1,I_2,I_3).
\end{equation}
By construction, this   Hamiltonian satisfies the following:
\begin{enumerate}
\item 
The $d$-dimensional torus $\mathcal T=\{I=0\}$ is invariant for the flow of $F$;
\item 
$\hat  I=(I_4,\dots, I_d)$ is conserved along the flow of $F$;
\item  
By \eqref{eq:kappabounds}, the Hamiltonian $F$ is $O(\delta)$-close to  the integrable Hamiltonian $F_0$ on the complex region $Q^d_{\rho,\sigma}\subset \mathbb (\mathbb C/2\pi\mathbb Z)^d\times\mathbb C^d$;
\item 
The Birkhoff normal forms for $F_0$ and $F$ at the origin coincide up to order $l$;
\item 
For each $n\in\mathbb N$ there exists an orbit $\{(\theta(t),I(t))\}_{t\in\mathbb R}$ of the Hamiltonian $F$ and a time $T_n>0$ such that 
\[
\lVert I(0)\rVert \leq \lVert \tilde I(0)\rVert + \lVert \hat I(0)\rVert \leq  \frac1n +\varepsilon_n^{\frac{1}{l+1}} +\mu_n ^{\frac{1}{l+1}}\leq \frac{3}{n^{\frac{1}{l+1}}},
\]
and (for large enough $n\in\mathbb N$) we have:
\[
 \lVert I(T_n)\rVert \geq  \lVert \tilde I(T_n)\rVert - \lVert \hat I(T_n)\rVert \geq   \rho_0-\varepsilon_n^{\frac{1}{l+1}} -\mu_n ^{\frac{1}{l+1}}\geq \frac{\rho_0}{2}.
\]

\end{enumerate}
The  proof of Theorem \ref{thm:maintori} is complete (up to establishing Theorem \ref{thm:approxbydiffusion}).

%%%%%%%%%%%%%
%%%%%%%%%%%%%
%%%%%%%%%%%%%
\section{Proof of Theorem \ref{thm:main}} \label{sec.proof.main}

In this section we show how to adapt the proof of Theorem \ref{thm:maintori} (Lyapunov instability for tori)  to obtain Theorem \ref{thm:main} (Lyapunov instability for elliptic equilibria). 
%The proof of Theorem \ref{thm:main2} follows from Theorem \ref{thm:main} and a completely standard argument and is left to the reader.
\medskip

There are two main steps needed to complete the proof of Theorem \ref{thm:main}:
\begin{enumerate}
    \item Reduction to a statement about parametric families of 3 degrees-of-freedom Hamiltonians.
    \item Existence of a parametric family of 3 degrees-of-freedom Hamiltonians which exhibit diffusion along a sequence of parameter values.
\end{enumerate}

The first step does not need any significant modification and can be achieved by the same argument as in Section \ref{sec:proofthmtori}. The second step is accomplished in Theorem \ref{thm:approxbydiffusionellipticpoints} below. We start with the necessary notations that are analogous to those in Section \ref{sec:outlineproof}.

{\bf Functional setting and external parameters.}
In this section we consider Hamiltonians in Cartesian coordinates $(\tilde q,\tilde p)\in B_\rho^{6}$  that also \textit{depend on external parameters}. Namely, given $0<\tilde d\black<d$, we  introduce the Banach space
\begin{equation}\label{eq:banachspaceelliptic}
\begin{split}
{\mathcal  P}^{\tilde d, d }_{\rho}=\{& G(\tilde q,\tilde p;\eta):  B^{2\tilde d}_\rho \times  B_{\rho}^
{d-\tilde d}\to\mathbb C\colon G\text{ is real-analytic, 
extends holomorphically to  $\mathbb B^{2\tilde d}_{\rho} \times \mathbb B_{\rho}^{d-\tilde d}$} \\
&\qquad\qquad \text{and } \sup_{(\tilde q,\tilde p;\eta)\in \mathbb B^{2\tilde d}_\rho\times \mathbb B_{\rho}^{d-\tilde d}}|G( \tilde q,\tilde p; \eta)|<\infty\}
\end{split}
\end{equation}
of functions $(\tilde q,\tilde p)\mapsto G(\tilde q,\tilde p;\eta)$  depending on an external $(d-\tilde d)$-dimensional parameter $\eta$. For conciseness, let us keep the notation 
$$
I=( I_1,\dots ,  I_{ d}) , \quad    I_i=( q_i^2+ p_i^2)/2, \quad i=1,\dots ,  d.
$$
Given $\omega\in \mathbb R^{\tilde d}$ with $\omega_i\neq 0$ we let:
\begin{enumerate}
    \item ${\mathcal M}^{\tilde d,d}_\rho$ be the closed subset of $\mathcal P^{\tilde d,d}_{\rho,\sigma}$ composed of functions only depending on $\tilde q,\tilde p$ through $\tilde I_j=\tilde q_j^2+\tilde p_j^2$
    \item ${\mathcal M}^{\tilde d,d}_\rho(\omega)=\{G_0\in {\mathcal M}^{\tilde d,d}_\rho\colon G_0=\frac12\omega\cdot \tilde I+O_2(\tilde I)\}$.
\end{enumerate} 
The following statement is the natural adaptation of Theorem \ref{thm:approxbydiffusion} to the case of elliptic equilibria.

{\renewcommand{\thethmx}{D$'$}
\begin{thmx}\label{thm:approxbydiffusionellipticpoints}
Fix any $\rho>0$ and let  $\omega\in \mathbb R^3$ such that $\omega_i\neq 0$ for all $i=1,2,3$ and  its entries are not all of the same sign. Then, for any $G_0$ in an open and dense subset of $\mathcal M^{3,d}_\rho(\omega)$ there exists $\rho_0>0$  (depending only on $G_0$) such that for any  $\delta>0$  and $l\in\mathbb N$  there exist two sequences $\{\varepsilon_n\}_n$ and $\{\mu_n\}_n$ satisfying $0<\varepsilon_n,\mu_n\leq 1/n$,
%such that for any  $\%{\eta^{(n)}\}_{n\in\mathbb N}\subset %B_\rho^{d-3}$   
and a parametric family of Hamiltonians $G_{\varepsilon,\mu}\in \mathcal P^{3,d}_{\rho}$ of the form 
    \[
    G_{\varepsilon,\mu}(\tilde q,\tilde p;\eta)=G_0(\tilde q,\tilde p;\eta)+\varepsilon g_0(\tilde q,\tilde p)+\varepsilon^2 g_1(\tilde q,\tilde p)+\mu g_2(\tilde q,\tilde p),\]
for which we have the following:
   \begin{enumerate}
   \item 
   The functions $\{g_i\}_i$ satisfy
   \[
       \sup_{(\tilde q,\tilde p)\in  \mathbb B^6_{\rho}} |g_i(\tilde q,\tilde p)|\leq \delta, \ \ i=0,1,2 ;
   \]
       \item 
       If we write $\eta^{(n)}=(\varepsilon_n^{\frac{1}{l+1}}, \mu_n^{\frac{1}{l+1}},0,\dots,0)$,  for any $n\in\mathbb N$ there exists an orbit 
       $\{(\tilde q(t;\eta^{(n)} ),\tilde p(t;\eta^{(n)}))\}_{t\in\mathbb R}$ of the Hamiltonian $G_{\varepsilon_n,\mu_n}(\tilde q,\tilde p;\eta^{(n)})$ and a time $T_n>0$ such that 
\[
%(\tilde q^2+\tilde p^2)
\lVert \tilde I(0;\eta^{(n)})\rVert \leq  \frac1n,\qquad\qquad \lVert \tilde I(T_n;\eta^{(n)})\rVert \geq \rho_0.
\]
   \end{enumerate}
\end{thmx}

The proof of Theorem \ref{thm:approxbydiffusionellipticpoints} is obtained as a modification of the proof of Theorem \ref{thm:approxbydiffusion}. The necessary changes are explained in Section \ref{sec:proofElliptic}.

\subsection{Proof of Theorem \ref{thm:main} assuming Theorem \ref{thm:approxbydiffusionellipticpoints}}

When Theorem \ref{thm:approxbydiffusionellipticpoints} is established, the proof of Theorem \ref{thm:main}  follows by adapting the embedding argument of Section \ref{sec:embedding}. The remainder of the paper is devoted to the proof of Theorem \ref{thm:approxbydiffusion}, while the additional modifications needed to establish Theorem \ref{thm:approxbydiffusionellipticpoints}  are presented in Section \ref{sec.82}.

%%%%%%%%%%%%%%%%%%
%%%%%%%%%%%%%%%%%%
%%%%%%%%%%%%%%%%%%

\section{Proof of Theorem \ref{thm:approxbydiffusion}}\label{sec:outlineparametric}

 We now  explain the main ideas which lead to the proof of Theorem \ref{thm:approxbydiffusion}. These are summarized in Theorems \ref{thm:mainparametricoutline} and \ref{thm:_apply_GToutline} below. At the end of this section (to be precise, in Section \ref{sec:proofD}) we show how to combine these results to obtain a proof of Theorem \ref{thm:approxbydiffusion}.
\medskip

\begin{rem} \label{rem.eta} For simplicity of notation, we suppress the dependence of the objects constructed below on the parameter $\eta$. In particular, the residual sets and the constants obtained in the construction may depend on $\eta$; this causes no difficulty, since in the proof of Theorem \ref{thm:approxbydiffusion}
 we only need the construction for a prescribed countable family of parameter values and therefore take the corresponding countable intersection of residual sets.
\end{rem}

\subsection{Resonant surfaces.}\label{sec:Resonant_surfaces}
As discussed above, the orbits realizing item $(2)$ of Theorem \ref{thm:approxbydiffusion} correspond to orbits that drift along certain resonant planes.  To construct a sequence of resonant planes converging to the origin, we introduce the following definition.

\begin{defn}\label{defn:Dirichletseq}
Given any  $\varpi\in\mathbb R$ we say that a sequence of integer vectors 
   $ s^{(n)}=\{(s_1^{(n)},s_2^{(n)})\}_{n\in\mathbb N}\subset \mathbb Z^2\setminus\{0\}$ is a \textit{Dirichlet sequence} for $\varpi$ if for all $n\in\mathbb N$ we have \begin{enumerate}
       \item 
       \textit{(Normalization).} $\mathrm{gcd}(s_1^{(n)},s_2^{(n)})=1$ for all $n\in\mathbb N$,
       \item 
       \textit{(Approximation).} $|s_1^{(n)}+s_2^{(n)}\varpi|\leq |s_2^{(n)}|^{-1}$.
       \end{enumerate}
\end{defn}

Observe that for $\varpi \in \mathbb R\setminus \mathbb Q$ the existence of a Dirichlet sequence is guaranteed by Dirichlet's theorem on Diophantine approximation. 
%\textcolor{magenta}{Add a reference for %Dirichlet theorem }\textcolor{blue}{In fact, %it can be found on Wikipedia.}
\begin{rem}\label{rem:rational}
For the ``degenerate'' situation in which 
$\varpi=\omega_2/\omega_1=p/q\in\mathbb Q$ we may simply take one integer vector $s=(-p,q)\in\mathbb Z^2$ instead of an infinite  sequence. As will be clear from our discussion below, in that case a Lyapunov unstable perturbation can be constructed much more easily, since a resonant surface traverses $\{I=0\}$. Hence, we do not elaborate on this case and assume from now on that the ratio $\omega_2/\omega_1\in\mathbb R\setminus\mathbb Q$.    
\end{rem}
\medskip 

Fix any $\omega=(\omega_1,\omega_2,\omega_3)\in\mathbb R^3$ with  
\begin{equation}\label{def:omega} 
\omega_i\neq 0 \quad \text{and} \quad \omega_2/\omega_1\in \mathbb R\setminus\mathbb Q,
\end{equation}  
and choose, once and for all, any Dirichlet sequence $\{s^{(n)}\}_n=\{(s^{(n)}_1,s^{(n)}_2)\}$ for $\omega_2/\omega_1$.
For notational purposes, it will also be convenient to introduce the extended Dirichlet sequence 
$$
\{\bs s^{(n)}\}_n=\{(s^{(n)}_1,s^{(n)}_2,0)\}\subset\mathbb Z^3\setminus\{0\}.$$ 
Given $\omega$ as in \eqref{def:omega} with a Dirichlet sequence 
$\{s^{(n)}\}_n\subset\mathbb Z^2\setminus\{0\}$, we say that $G_0\in \mathcal N^{3,d}_\rho(\omega)$ satisfies assumption \ref{it:assumption1} if

\begin{equation}\label{it:assumption1}
 \text{for } \bs v = \left(-\omega_2, \omega_1,0 \right) \text{  we have }\  \partial_{I_1} (\bs v\cdot \nabla G_0)(0)  \neq 0\  \text{ and }\     \bs v^T D^2G_0(0) \bs v\neq 0. \tag*{[H1]} 
\end{equation}
\medskip
%\begin{enumerate}[label=%{[H\arabic*]}]
 %   \item \label{it:assumption1} 
%   if $\partial_{I_1^2}^2 %G_0(0)\neq 0$ and for $\bs v = %\left(-\omega_2, \omega_1 \right)$ we have   $ \bs v^T %D^2G_0(0) \bs v\neq 0$.
%\end{enumerate}
%\medskip
Then, for each $n\in\mathbb N$ we define the energy level 
\begin{equation} \label{eq:energieleveln}
e_n=G_0(\hat a_1^{(n)},0,0),
\end{equation}
where $\hat a_1^{(n)}$ is the unique solution to the equation 
\begin{equation}\label{eq:hata1n}
\bs s^{(n)}\cdot \nabla G_0(\hat a_1^{(n)},0,0)=0.
\end{equation}
Notice that  this equation admits a solution, locally unique in a neighborhood of the origin, provided that the first part of  assumption \ref{it:assumption1}  is satisfied. 
Moreover, the point   $(\hat a_1^{(n)},0,0)\in B_\rho^3$ belongs to the resonance locus associated with the integer vector $\bs s^{(n)}\in\mathbb Z^3$, and $e_n$ in \eqref{eq:energieleveln} is the value of the integrable Hamiltonian $G_0$ at that point. 
We now observe that assumption \ref{it:assumption1}  guarantees the existence of a sequence of resonant curves (containing the point $(\hat a_1^{(n)},0,0)$) approaching the origin. 

\medskip

\begin{lem}\label{lem:resnonatpaths}
  Fix any $\rho>0$ and $\omega$ as in \eqref{def:omega} with a Dirichlet sequence $\{s^{(n)}\}_n\subset\mathbb Z^2\setminus\{0\}$ associated to the ratio $\frac{\omega_2}{\omega_1}$.    
  For any $G_0\in \mathcal N^{3,d}_\rho(\omega)$ satisfying \ref{it:assumption1}, there exists $\rho_0=\rho_0(G_0)>0$, such that for any sufficiently large $n\in\mathbb N$ the system of equations
    \begin{equation}\label{eq:anpaths}
          G_0(a^{(n)}(u))=e_n,\qquad\qquad \bs s^{(n)}\cdot  \nabla G_0(a^{(n)}(u))=0,
    \end{equation}
    where $e_n$ is given in \eqref{eq:energieleveln},
    defines a real-analytic embedded curve $a^{(n)}(u):[-\rho_0,\rho_0] \black \to B^{3}_\rho$ of the form
 \begin{equation}\label{eq:anpaths0}
    a^{(n)}(u)=(a^{(n)}_1(u),a^{(n)}_2(u),u)
    \end{equation}
    satisfying 
    \[
    \lim_{n\to \infty} \lVert a^{(n)}(0)\rVert \to 0. 
    \]
    \end{lem}

  \begin{rem}
    The number $\rho_0=\rho_0(G_0)$ provided by the lemma above is important, as it will give a lower bound for  the size of diffusing trajectories for Hamiltonians close to $G_0$.
    \end{rem}
    
    \begin{proof}

For any $u\in [-\rho,\rho]$ define the map $f_u^{(n)}:B^2_\rho\to \mathbb R^2$ given by 
        \[
        f_u^{(n)}:(a_1,a_2)\mapsto (G_0(a_1,a_2,u)-e_n,\  \bs v^{(n)}\cdot  \nabla G_0(a_1,a_2,u)),
        \]
        where $\bs v^{(n)}=\bs s^{(n)}/\lVert \bs s^{(n)}\rVert$. Let $\hat a_1^{(n)}$ be the solution to \eqref{eq:hata1n} and observe that for $u=0$ the point $(\hat a_1^{(n)},0)\in B^2_\rho$ satisfies 
        \[
        f_0^{(n)}(\hat a_1^{(n)},0)=0.
        \]
        Using the fact that  
        \[
        \bs v^{(n)}\cdot \nabla G_0(\hat a^{(n)}_1,0,0)=0,
        \]
        it is straightforward to check that (the right hand side being evaluated at $(\hat a_1^{(n)},0,0)$)
\[
\mathrm{det}(Df_0^{(n)})(\hat a_1^{(n)},0)=-\frac{\partial_{I_2}G_0}{v_1^{(n)}} (\bs v^{(n)})^T D^2G_0 \bs v^{(n)} = \frac{\partial_{I_1}G_0}{v_2^{(n)}} (\bs v^{(n)})^T D^2G_0 \bs v^{(n)}.
\]
%we denote:
%\[
%D^2G_0=\left( 
%\begin{array}{ccc}
%\partial^2_{I_1^2} G_0   % & \partial^2_{I_1 \, %I_2} G_0 & 0 \\
%\partial^2_{I_1\, I_2} %G_0    & %\partial^2_{I_2^2} G_0 &0 \\
%0&0&0
%\end{array}
%\right)  (\hat a^{(n)}_1,0,0)
%\]
Using \eqref{eq:hata1n} we have that
$s^{(n)}_2= -\frac{\partial_{I_1} G_0  }{\partial_{I_2} G_0} (\hat a^{(n)}_1,0,0) s^{(n)}_1$, 
therefore
\[
\bs v^{(n)} = \frac{1}{\sqrt{(\partial_{I_2} G_0)^2+(\partial_{I_1} G_0)^2}} 
\left( -\partial_{I_2} G_0, \partial_{I_1} G_0 ,0 
\right) (\hat a^{(n)}_1,0,0)
\]
and, consequently, since $G_0$ is uniformly $\C^3$ on $B_\rho^3$ and $|\hat a_1^{(n)}|\to 0$ as $n\to \infty$
\[
\lim_{n\to \infty} \bs v^{(n)}=
\frac{ 1}{\sqrt{(\omega_2)^2+(\omega_1)^2}} 
\left( -\omega _2, \omega_1 ,0 \right),
\]
assumption
%Since $G_0$ is uniformly $\C^3$ on $B_\rho^3$ and $|\hat a_1^{(n)}|\to 0$ as $n\to \infty$,   we can compute the assumption 
\ref{it:assumption1} guarantees that there exists $c>0$, depending only on $G_0$ such that for $n$ large enough,
\[
|\mathrm{det} (Df_0^{(n)})(\hat a_1^{(n)},0)|\geq c.
\]
Therefore, the conclusion follows from a direct application of the implicit function theorem. 
    \end{proof}
    \medskip

From a dynamical point of view, the 4-dimensional submanifold 
\begin{equation}\label{eq:resonancelocus}
\mathcal R^{(n)}=\{(\tilde \theta,\tilde I)\in Q^3_{\rho,\sigma}\colon \tilde I=a^{(n)}(I_3), \ I_3\in[-\rho_0,\rho_0]\}\subset\{G_0=e_n\}
\end{equation}
constitutes a resonant submanifold foliated  by the  $3$-dimensional invariant tori 
\[
\mathcal T^{(n)}_u=\{(\tilde \theta,\tilde I)\in Q^3_{\rho,\sigma}\colon \tilde I=a^{(n)}(u )\},\qquad\qquad u\in [-\rho_0,\rho_0],
\] 
where the integrable dynamics is given by the resonant translation
\[
\phi^t_{G_0}|_{\mathcal T^{(n)}_u}:(\tilde \theta,a^{(n)}(u))\mapsto (\tilde \theta+\omega_n(u ) t,\  a^{(n)}(u))\]
with
\[
\omega_n(u)=(-\frac{s_2^{(n)}}{s_1^{(n)}}\partial_{I_2} G_0(a^{(n)}(u)) , \ \partial_{I_2} G_0(a^{(n)}(u)), \partial_{I_3} G_0(a^{(n)}(u)).
\]
Hence, for any $u\in\mathbb [-\rho_0,\rho_0]$ the torus $\mathcal T_u^{(n)}$ is itself foliated by two-dimensional invariant tori. 
To describe the dynamics in a neighborhood of $\mathcal R^{(n)}$ we  define $k^{(n)}\in\mathbb Z^2\setminus\{0\}$ as an
%any of the unique two 
integer vector
%\footnote{By Bezout's lemma infinitely many vectors $k\in\mathbb Z^2\setminus\{0\}$ exist satisfying $\mathrm{det} (s^{(n)}|k)=1$. Among these, only two satisfy the second requirement in \eqref{eq:Bezout}.} 
for which the pair $s^{(n)},k^{(n)}\in\mathbb Z^2\setminus\{0\}$ satisfies that 
 \begin{equation}\label{eq:Bezout}
 s_1^{(n)}k_2^{(n)}-k_1^{(n)}
 s_2^{(n)}=1 \qquad\qquad \text{and}\qquad\qquad |k_j^{(n)}|<|s_j^{(n)}|,\ j=1,2.
 \end{equation}
Note that, by Bezout's lemma, infinitely many vectors $k\in\mathbb Z^2\setminus\{0\}$ exist satisfying $\mathrm{det} (s^{(n)}|k)=1$. Among these, only two satisfy the second requirement in \eqref{eq:Bezout}.
From now on, one should think of $ s^{(n)}\cdot (\theta_1,\theta_2)$ as a {\it slow angle} in a neighborhood of the resonant submanifold $ \mathcal R^{(n)}$ and of $k^{(n)}\cdot(\theta_1,\theta_2)$ and $\theta_3$ as {\it fast angles} in a neighborhood of the resonant submanifold $\mathcal R^{(n)}$. \black
\medskip

\subsection{Creation of a family of transverse homoclinic cylinders.}\label{sec:Creation_cylinders}
We now show that a Hamiltonian $G_0$ satisfying \ref{it:assumption1} can be perturbed in such a way that, along each of the resonant paths $a^{(n)}$  in Lemma \ref{lem:resnonatpaths}, one can build a normally hyperbolic cylinder with a rich network of homoclinic cylinders. To obtain diffusing orbits along this structure, we introduce now our second non-degeneracy condition, which guarantees a twist property for the induced dynamics. Given $\omega$ as in \eqref{def:omega} with a Dirichlet sequence $\{s^{(n)}\}_n\subset\mathbb Z^2\setminus\{0\}$, we say that $G_0\in \mathcal N^{3,d}_\rho(\omega)$ satisfies  assumption \ref{it:assumption2} if

\begin{equation}
\text{ for } \bs v=\left(-\omega_2, \omega_1,0 \right) \text{ and }\tilde{\bs\omega}=(-\omega_3,0,\omega_1)
  \text{ we have }  
    \cA: = D^2G_0(0)\tilde{\bs\omega}\  \cdot\  \left(D^2 G_0(0) \bs v \wedge \omega \right)    \neq 0 ,
 \tag*{[H2]} \label{it:assumption2}
\end{equation}
where $\wedge $ stands for the usual vector product.
\medskip
%\begin{enumerate}[label={[H\arabic*]},start=2]
%    \item \label{it:assumption2} if for $\bs v=\lim_{n\to\infty}\bs s^{(n)}/\lVert \bs s^{(n)}\rVert$
%    \[
%    T: = D^2G_0(0)\begin{pmatrix} -\omega_3\\0\\ \omega_1 \end{pmatrix}\  \cdot\  \left(D^2 G_0(0) \bs v \wedge \omega \right)    \neq 0, 
%\]
%\end{enumerate}

\begin{lem}\label{lemmaH1H2}
Fix any $\rho>0$ and let $\omega\in\mathbb R^3$ satisfy
$\omega_i\neq 0$ for $i=1,2,3$.
Then the subsets of
\[
\mathcal N^{3,d}_\rho(\omega)
\qquad\text{and}\qquad
\mathcal M^{3,d}_\rho(\omega)
\]
consisting of Hamiltonians satisfying \ref{it:assumption1} and \ref{it:assumption2} are open
and dense.
\end{lem}

\begin{proof}
Conditions \ref{it:assumption1} and \ref{it:assumption2} only involve the Hessian
$D^2 G_0(0)$ with respect to the three-dimensional dynamical action
variable ${\widetilde I}$. They are therefore open conditions on
$\mathcal N^{3,d}_\rho(\omega)$ and
$\mathcal M^{3,d}_\rho(\omega)$.

To prove density, let $G_0$ belong to either of these spaces.
For any sufficiently small quadratic polynomial
$Q(\widetilde I)$ depending only on the three dynamical action variables,
the Hamiltonian
\[
\widetilde G_0(\widetilde I)
=
G_0(\widetilde I)+Q(\widetilde I)
\]
still belongs to the same functional space and has the same linear part.
Since the Hessian of $Q$ can be chosen arbitrarily, the Hessian
$D^2 \widetilde G_0(0)$ can be perturbed arbitrarily
within the space of symmetric $3\times3$ matrices.
The conditions \ref{it:assumption1} and \ref{it:assumption2} are given by the non-vanishing of
explicit polynomial expressions in this Hessian, and these polynomials are
not identically zero. Hence, an arbitrarily small perturbation can always
be chosen so that both \ref{it:assumption1} and \ref{it:assumption2} hold simultaneously.

Therefore, the subsets of
$\mathcal N^{3,d}_\rho(\omega)$ and
$\mathcal M^{3,d}_\rho(\omega)$ satisfying \ref{it:assumption1} and \ref{it:assumption2} are
open and dense.
\end{proof}

 In Section \ref{sec:cylinders} we show the following. 

\begin{thm}\label{thm:mainparametricoutline}
Fix any pair $\rho,\sigma>0$, let  
 $\omega$ be as in \eqref{def:omega} and fix a  Dirichlet sequence $\{s^{(n)}\}_n\subset\mathbb Z^2\setminus\{0\}$ associated to the ratio $\frac{\omega_2}{\omega_1}$.   
 Let  $G_0\in \mathcal N^{3,d}_\rho(\omega)$ satisfy \ref{it:assumption1}-\ref{it:assumption2},  let $\rho_0=\rho_0(G_0)$ be given by Lemma \ref{lem:resnonatpaths},    and consider a parametric Hamiltonian of the form 
\begin{equation}\label{eq:limitHamiltonianoutline} 
    G_{\varepsilon}  (\tilde \theta,\tilde I)=
G_0(\tilde I)+\varepsilon g_0(\tilde \theta),
\qquad\qquad
{where}\qquad\qquad g_0(\tilde \theta)=\sum_{j\in\mathbb N}b_j\cos( \bs s^{(j)}\cdot \tilde  \theta).
\end{equation}
Fix any $\delta>0$. 
Then,  if the sequence $\ \{b_n\}\subset\mathbb R_+$ decays to zero sufficiently fast, it is possible to find a sequence $\ \{\varepsilon_n\}_n$
with $0<\varepsilon_n\leq 1/n$ such that the following holds. 
For any fixed  
$\eta\in B_{\rho_0}^{d-3}$  there exists a residual subset $\mathcal U_{\rho,\sigma}(\delta)\subset \mathcal B^2_{\rho, \sigma} (\delta)\subset \mathcal{Q}_{\rho, \sigma}^2$ (see \eqref{spaceQdsigmarho}-\eqref{eq:balllagrangian}) such that for any $g\in \mathcal{U}_{\rho,\sigma}(\delta)$  the Hamiltonian 
\begin{equation}\label{eq:firstperturb}
\mathcal G_\varepsilon (g)=G_\varepsilon+\varepsilon^2 g 
\end{equation}
satisfies the following conditions:
\begin{enumerate}
    \item[(0)]  
    For any $\varepsilon\in[0,1]$ we have $\mathcal G_{\varepsilon}(g)\in \mathcal {Q}_{\rho,\sigma}^{3,d}$ (see \eqref{eq:banachspace}) and 
    \[
    \sup_{(\tilde \theta,\tilde I)\in  Q^3_{\rho,\sigma}} (|\mathcal G_\varepsilon(g)-G_0|)<2\delta;
    \]
    \item[(1)]
    For any $n\in\mathbb N$,
    the Poincar\'e map $\Phi_n:\Sigma_n\to \Sigma_n$ induced by the Hamiltonian $\mathcal G_{\eps_n}(g)$ on the section 
\begin{equation}\label{eq:4dsection}
\Sigma_n=\{  k_1^{(n)}\theta_1+k_2^{(n)}\theta_2=0, \ \mathcal G_{\varepsilon_n}(g)(\theta,I)=e_n\}
\end{equation}
is (locally) well-defined;

\item[(2)] 
The Poincar\'e map $\Phi_n$  
possesses a  normally hyperbolic invariant manifold (cylinder) which admits a parametrization of the form
\begin{equation}\label{eq:2dcylinderoutline}
{\Lambda}_n(g)= \{  (\theta_1^{(n)}(I_3),\theta_2^{(n)}(I_3), \theta_3,I_1^{(n)}(I_3),I_2^{(n)}(I_3),I_3)\colon \ \theta_3\in\mathbb T,\  I_3\in[-\rho_0,\rho_0]\}
\end{equation}
for some real-analytic functions $\theta_i^{(n)}, I_i^{(n)}$, $i=1,2$,  satisfying 
\[
|I_i^{(n)}-a_{i}^{(n)}|_{\C^1}\leq \frac 1n
\]
for $a^{(n)}=(a_1^{(n)},a_2^{(n)})$ as in \eqref{eq:anpaths}-\eqref{eq:anpaths0};
\item[(3)]
The restriction  $\Phi_n|_{\Lambda_n}$ is an integrable twist map which leaves invariant the foliation by  circles $\{I_3=\mathrm{const}\}$, given by
\begin{equation}\label{eq:twist}
\Phi_{n}|_{\Lambda_n}:(\th_3, I_3)\mapsto (\th_3  + \tilde \nu_n (I_3, \varepsilon_n),\ I_3); 
\end{equation}
\item[(4)]  
There exists $m\in\mathbb N$ (depending a priori on $n$) and a covering of $[-\rho_0,\rho_0]$ by intervals:
\[
[-\rho_0,\rho_0]\subset \bigcup_{0\leq i\leq m} [\cI_i^-,\cI_i^+],
\]
satisfying $\cI_{i+1}^-<\cI_{i}^+$ for all $i=0,\dots,m-1$,  such that  for any $i=0,\dots,m-1$ there exists a homoclinic cylinder $\Gamma_{n,i}=\Gamma_{n,i}(g)$ admitting a parametrization of the form
\[
\Gamma_{n,i}=\{  (\hat \theta_{1,i}^{(n)}(I_3), \hat \theta_{2,i}^{(n)}(I_3), \theta_3,\hat I_{1,i}^{(n)}(I_3),\hat I_{2,i}^{(n)}(I_3),I_3)\colon \ \theta_3\in\mathbb T,\  I_3\in[\cI_i^-,\cI_i^+]\}
\]
for certain real-analytic functions $\hat\theta_{j,i}^{(n)}, \hat I_{j,i}^{(n)}$ with $j=1,2$ and $i=0,\dots,m-1$.
\end{enumerate}

\end{thm}

The proof of Theorem \ref{thm:mainparametricoutline} is given in Section \ref{sec:cylinders}.  
It consists of two steps. 
In the first step we show that the Hamiltonian $G_\varepsilon$ in \eqref{eq:limitHamiltonianoutline} admits a sequence of normally hyperbolic invariant cylinders $\Lambda_n$ of the form \eqref{eq:2dcylinderoutline}. The idea is that the perturbation $g_0$ in \eqref{eq:limitHamiltonianoutline}  consists of a sum of Fourier modes such that, for any fixed  $n\in\mathbb N$,  on the resonant surface $\mathcal R^{(n)}$ defined in \eqref{eq:resonancelocus} we have: 
\begin{itemize}
\item 
the terms with $j\neq n$ are non-resonant and average out up to a $O(\varepsilon^2)$ remainder, 
\item 
the $n$-th term in the sum is resonant and creates the normally hyperbolic invariant cylinder $\Lambda_n$.
\end{itemize}
Moreover, since the perturbation is explicit, we can verify that the length of the cylinders $ \Lambda_n$ is uniform with respect to $n\in\mathbb N$. That is, the Hamiltonian $G_\varepsilon$ already satisfies items $(0)-(3)$ in Theorem \ref{thm:mainparametricoutline}. However, we are not able to show directly that the Hamiltonian $G_\varepsilon$ verifies item $(4)$ in Theorem \ref{thm:mainparametricoutline}.
At the end of this step we verify that, by the same arguments, conclusions (0)--(3) hold for a Hamiltonian $\mathcal G_\varepsilon(g)$ of the form \eqref{eq:firstperturb} for any  $g\in \mathcal B_{\rho,\sigma}^{2}(\delta)$.
\medskip

In the second step we show that, for any $g$ in a  generic subset $\mathcal{U}_{\rho,\sigma}(\delta)\subset \mathcal B_{\rho,\sigma}^{2}(\delta)$, the Hamiltonian $\mathcal G_\varepsilon(g)$ as in \eqref{eq:firstperturb} verifies 
item $(4)$ of Theorem \ref{thm:mainparametricoutline}. The idea of the proof is based on a nonlinear version of a recent parametric transversality argument developed by Delshams and Zhang \cite{MR4913967}.  The key observation that allows us to adapt their ideas to our problem  is that we only consider perturbations which do not depend on the angle $\theta_3$ since, in this case, the dynamics on the corresponding normally hyperbolic cylinders are integrable and admit a foliation by invariant tori. 

\begin{figure}
    \centering
\includegraphics[scale=0.65]{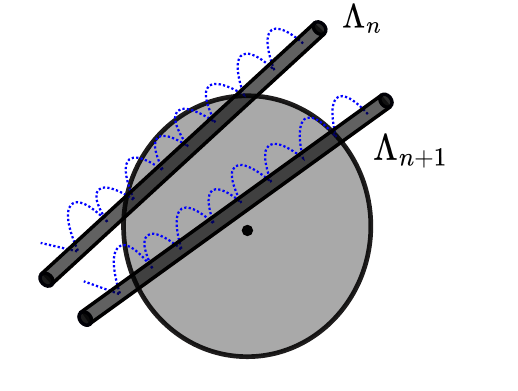}
    \caption{A sequence of resonant cylinders (in action space) which approach the origin. The blue dashed lines represent orbits drifting along them.
    %E. Figure 4.1 is helpful but oversimplifies one point
%The caption says that the dashed lines represent orbits drifting from one end of each normally hyperbolic cylinder to the other.
}
    \label{fig:fig1}
\end{figure}
\medskip

\begin{rem} \label{rem.bn} The sequences ${b_n}$ and ${\varepsilon_n}$ are chosen inductively. Once $b_1,\ldots,b_n$ and $\varepsilon_1,\ldots,\varepsilon_n$ have been fixed so that the required properties hold at the first $n$ resonances, we choose $b_{n+1}>0$ sufficiently small so that the perturbation associated with the additional Fourier mode does not destroy any of the previously constructed objects. We then choose $\varepsilon_{n+1}>0$ sufficiently small so that the estimates required in the construction at the $(n+1)$-st resonance are satisfied. By openness and persistence of the relevant hyperbolic and transverse structures, all properties obtained at the previous stages remain valid. Proceeding inductively yields the required sequences. See Section \ref{sec:cylinders} for details.
\end{rem}

\subsection{Diffusion along the transverse homoclinic cylinders.}

    Finally, we construct orbits drifting along the sequence of cylinders $\Lambda_n $ (see Figure \ref{fig:fig1}). 
   This part is more standard and the main ideas in the instability mechanism go back to the work of Moeckel \cite{MR1884898} in the context of Arnold diffusion (see also \cite{MR2184276,MR2383896}). The celebrated work by Gelfreich and Turaev \cite{GelfreichTuraevChaotic} (see also \cite{MR2086152}) provides a ready-to-use  implementation of Moeckel's mechanism which can be applied directly in our context. Indeed, in very rough terms, the main result in \cite{GelfreichTuraevChaotic} says that among the set of real-analytic Hamiltonian systems exhibiting a normally hyperbolic cylinder $\Lambda$ which displays a homoclinic cylinder to it, those for which there exists an orbit connecting the boundaries of $\Lambda$ form an open and dense subset. In Section \ref{sec:diffusion} we show the following.

\begin{thm}\label{thm:_apply_GToutline} 
 Fix any pair $\rho,\sigma>0$; let $\omega$, $G_0$ and $\rho_0=\rho_0(G_0)$ be as in the assumption of Theorem \ref{thm:mainparametricoutline}.
There exists a sequence $\{ \mu_n\}_n$,   $ 0<\mu_n\leq 1/n$  such that the following holds. Fix any $\delta>0$, any $\eta\in B^{d-3}_{\rho_0}$, let  
$\{\varepsilon_n\}_n$ and $\mathcal{U}_{\rho,\sigma}(\delta)\subset \mathcal B^2_{\rho,\sigma}(\delta)$ be as in the conclusion of Theorem \ref{thm:mainparametricoutline}  and denote
\[
\mathcal G_{\varepsilon,\mu}(g_1,g_2)=\mathcal G_\varepsilon(g_1)+\mu g_2, \qquad\qquad g_1\in \mathcal U_{\rho,\sigma}(\delta), \quad g_2\in \mathcal B^3_{\rho,\sigma}(\delta).
\]

 Then for any $g_1\in \mathcal U_{\rho,\sigma}(\delta)$ and any  $n\in\mathbb N$, there exists an open and dense set $\mathcal V^{(n)}_{\rho,\sigma}(\delta) \subset \mathcal B^3_{\rho,\sigma}(\delta)$ such that for any $g_2\in\mathcal V^{(n)}_{\rho,\sigma}(\delta)$   
the Hamiltonian $\mathcal G_{\varepsilon_n,\mu_n}(g_1,g_2)$  satisfies the following:
\begin{enumerate}
\item it displays a normally hyperbolic invariant manifold $\Lambda_n(g_1,g_2)$ which is $\frac 1n$-close to $\Lambda_n(g_1)$ in Theorem \ref{thm:mainparametricoutline} (see the parametrization  \eqref{eq:2dcylinderoutline}).
    \item If $U^\pm_n$ is any open neighborhood of $\partial^{\pm} \Lambda_n(g_1,g_2)$ then, there exists an orbit of this Hamiltonian  which connects $U^-_n$ to $U^+_n$.
\end{enumerate} 
\end{thm}
We denote by
\[
\partial^\pm\Lambda_n(g_1,g_2)
=
\{(\theta,I)\in\Lambda_n(g_1,g_2): I_3=\pm\rho_0\}
\]
the two boundary components of the cylinder $\Lambda_n(g_1,g_2)$.

A detailed discussion about how to use the main result in \cite{GelfreichTuraevChaotic} to deduce Theorem \ref{thm:_apply_GToutline} is given in Section \ref{sec:diffusion}.
\vspace{0.2cm}

\subsection{Proof of Theorem \ref{thm:approxbydiffusion}}\label{sec:proofD}
We now complete the proof of Theorem \ref{thm:approxbydiffusion}. 

\begin{proof}[Proof of Theorem \ref{thm:approxbydiffusion}]
Fix any $\eta\subset B_{\rho_0}^{d-3}$ and choose any $g_1\in \mathcal U_\sigma(\delta)$.  
Let $\{ \mu_n\}_n$ be the sequence in Theorem \ref{thm:_apply_GToutline} and, for the Hamiltonian $\mathcal G_\varepsilon(g_1)$, let   $\mathcal V^{(n)}_{\rho,\sigma}$ be the sequence of sets provided by Theorem \ref{thm:_apply_GToutline}.  
Choose any $g_2\in\bigcap _n \mathcal V_\sigma^{(n)}(\delta)$.  
Then, for any $n\in\mathbb N$, the Hamiltonian $\mathcal G_{\varepsilon_n,\mu_n} (g_1,g_2)$ satisfies the following.
For any open neighborhood $U^\pm_n$ of $\partial^{\pm} \Lambda_n(g_1,g_2)$ there exists an orbit of this Hamiltonian  which connects $U^-_n$ to $U^+_n$. 
The same conclusion can be guaranteed for an arbitrary sequence 
$\{\eta^{(n)}\}\subset B^{d-3}_{\rho_0}$ by choosing $(g_1,g_2)$ from a countable intersection of residual sets.
The proof of Theorem \ref{thm:approxbydiffusion} is completed.
\end{proof}
\medskip

\section{A sequence of large homoclinic cylinders: proof of Theorem \ref{thm:mainparametricoutline}}\label{sec:cylinders}

In this section we show how to construct a two-parameter family of real-analytic Hamiltonians which include the integrable Hamiltonian $G_0$ in Theorem \ref{thm:approxbydiffusion} and exhibit a family of transverse homoclinic cylinders at suitable parameter values. That is, we give a proof of Theorem \ref{thm:mainparametricoutline}. 

%\begin{rem}\label{rem:I3asparameter}
    The particular class of Hamiltonians which we use in this section  does not depend on the angle $\theta_3\in\mathbb T$, and therefore $I_3$ is a conserved quantity and will be considered as a parameter until Section \ref{sec:Hamtomaps} (in which we complete the proof of Theorem \ref{thm:mainparametricoutline}). 
    Thus, throughout Sections \ref{sec:2dof}-\ref {sec:induction}, we deal with  Hamiltonians with only \textit{two degrees-of-freedom}. 
    
 \subsection{Hamiltonians with 2 degrees-of-freedom}\label{sec:2dof}   
 The following notations will be used in Sections \ref{sec:2dof}-\ref {sec:induction}. (Note that here we use the same names for the two-dimensional objects as the ones we used for the three-dimensional ones in the rest of the paper. Hope this does not lead to confusion.) Fix $\omega$ as in Theorem  \ref{thm:approxbydiffusion}, and let $G_0(I_1,I_2,I_3;\eta) \in \mathcal N^{3,d}_\rho(\omega)$ satisfy  \ref{it:assumption1}.  
 \begin{itemize}

\item Here we denote: 
\[ \theta=(\theta_1,\theta_2)\in \mathbb T^d, \quad I=(I_1,I_2) \in\mathbb B^2_\rho, \quad 
\omega=(\omega_1,\omega_2)\in\mathbb R^2,\qquad\qquad\omega_2/\omega_1\in\mathbb R\setminus\mathbb Q.
\]
 
     \item 
    Let  $\rho_0=\rho_0(G_0)$ be the constant provided by  Lemma \ref{lem:resnonatpaths}.

 \item  We will study
    \[
H_0(I;I_3,\eta)=G_0(I_1,I_2,I_3;\eta ),\qquad\qquad I_3\in[-\rho_0,\rho_0], \quad \eta \in \mathbb{B}_\rho^{d-3}.
    \]
 In other words, we consider a parametric family of real-analytic integrable Hamiltonians $H_0\in \mathcal N_{\rho}^{2,d}(\omega)$ (recall the definition of this Banach space in \eqref{eq:banachspace})
    \begin{equation}\label{eq:unperturbedparam}
   I\mapsto H_0(I;I_3), \qquad\qquad I=(I_1,I_2) \in\mathbb B^2_\rho, \qquad\qquad I_3\in [-\rho_0,\rho_0],
    \end{equation}
    satisfying (compare to \ref{it:assumption1} and see Remark \ref{rem:remarkonnondeg})
   \begin{equation}\label{eq:nondegh}
\partial_{I_1} ( (-\omega _2,\omega_1)\cdot \nabla  H_0)(0;0) > 0.
    \end{equation}

%\magenta should be $\partial_{I_1} ( (-\omega _2,\omega_1)\cdot \nabla  H_0)(0;0) > 0.
%$ \blue OK \black
%We denote the conjugate angle variables by $\theta=(\theta_1,\theta_2)$.

\item 
We will say that $\rho_0$ defined above corresponds to $H_0$ through Lemma \ref{lem:resnonatpaths}, or, in other words, that
$\rho_0=\rho_0(H_0)$ is given by Lemma \ref{lem:resnonatpaths}.

%\begin{rem}\label{rem:Notations_for_H_0}
\item 
Here we denote by 
$$a^{(n)}(I_3)=(a^{(n)}_1(I_3),a^{(n)}_2(I_3))$$  the first two components of the curve in \eqref{eq:anpaths0}.
\end{itemize}
%\end{rem}
\medskip

In this section we think of $I_3\in[-\rho_0,\rho_0]$ and $\eta$ as being fixed. For the construction in Section \ref{sec:diffusion}, where we embed 2 degrees-of-freedom Hamiltonians in 3 degrees-of-freedom ones, it will be relevant to understand the dependence of the objects constructed in this section (hyperbolic periodic orbits and transverse homoclinics) on $I_3$. 
The dependence on $\eta$ will often be omitted in the notation.

\medskip

 In the following, it will be important to keep in mind that 
\[
\varphi:=s^{(n)}\cdot\theta
\]
 is a \textit{slow angle} in a neighborhood of the resonant submanifold $\mathcal R^{(n)}$ (with an abuse of notation, we denote by $\mathcal R^{(n)}$ the section $\mathcal R^{(n)}\cap \{\theta_3=0,\ I_3=\mathrm{const}\}$ with $\mathcal R^{(n)}$ as in \eqref{eq:resonancelocus}).  We define $k^{(n)}\in\mathbb Z^2\setminus\{0\}$ as  in \eqref{eq:Bezout}. One should think of 
\[
\tt:= k^{(n)}\cdot \theta
\]
 as a \textit{fast angle} in a neighborhood of $ \mathcal R^{(n)}$.

\begin{rem}\label{rem:remarkonnondeg}
    Clearly, the non-degeneracy condition in \eqref{eq:nondegh} is not dense within $\mathcal N^{2,d}_\rho(\omega)$. We will see below that having a positive sign in \eqref{eq:nondegh} guarantees that the perturbed Hamiltonian \eqref{eq:limitHamiltonian} bellow  displays a hyperbolic periodic orbit close to $\mathcal R^{(n)}\cap \{\varphi=0\}$. It is straightforward to observe that  for the negative case  we can replace $h_0$ in Theorem \ref{thm:mainparametric}  by $h_0=-\sum_j b_j \cos(s^{(j)}\cdot \theta)$ to obtain the same conclusion. Alternatively, we could leave $h_0$ unchanged and consider the corresponding fixed point close to $\mathcal R^{(n)}\cap \{\varphi=\pi\}$. Hence, from now on we restrict to the case \eqref{eq:nondegh}. 
\end{rem}
    
 The following Theorem \ref{thm:mainparametric}  is just a reformulation of Theorem \ref{thm:mainparametricoutline} in terms of Hamiltonians of 2 degrees-of-freedom.% (see Remark \ref{rem:I3asparameter}). 

\begin{thm}\label{thm:mainparametric}
Fix arbitrary $\delta, \rho,\sigma>0$.  Let $H_0\in\mathcal N_\rho^{2,d}(\omega)$ be as in \eqref{eq:unperturbedparam} satisfying \eqref{eq:nondegh}, and let $\rho_0=\rho_0(H_0)$ be given by Lemma \ref{lem:resnonatpaths} (see the beginning of Section \ref{sec:2dof}). 
Consider a parametric Hamiltonian of the form 
\begin{equation}\label{eq:limitHamiltonian} 
H_{\varepsilon}(\theta, I;I_3,\eta)=
H_0(I;I_3,\eta)+\varepsilon  h_0(\theta),\qquad\text{where}\qquad
h_0(\theta)=\sum_{j\in\mathbb N}b_j\cos( s^{(j)}\cdot \theta).
\end{equation}
Then,  if the sequence $\ \{b_n\}_n\subset\mathbb R_+$ decays to zero sufficiently fast, it is possible to find  a sequence$\ \{\varepsilon_n\}_n$
with $0<\varepsilon_n\leq 1/n$ such that the following holds. 
There exists a residual subset  $\mathcal U_{\rho, \sigma}(\delta)\subset \mathcal B^2_{\rho, \sigma}(\delta)$ such that for any $h\in \mathcal{U}_{\rho, \sigma}(\delta)$  the Hamiltonian 
\begin{equation}\label{eq:firstperturbH}
\mathcal H_\varepsilon (h)=H_\varepsilon+\varepsilon^2 h
\end{equation}
satisfies:
\begin{enumerate}
\item  
For any $\varepsilon\in[0,1]$ we have $\mathcal H_{\varepsilon}(h)\in \mathcal {Q}_{\rho,\sigma}^{2,d}$ and 
\[
\sup_{(\theta,I)\in   Q_{\rho,\sigma}^2} (|\mathcal H_\varepsilon(h)-H_0|)<2\delta;
\]
\item  
For any $n\in\mathbb N$ and 
$I_3\in[-\rho_0,\rho_0]$ the  Hamiltonian $\mathcal H_{\varepsilon_n}(h)(\cdot;I_3)$  possesses a  hyperbolic periodic orbit $\gamma_n(h)(I_3)$ which admits a parameterization of the form
\begin{equation}\label{eq:parametrizationperiodicproposition}
\gamma_n(h)(I_3)= \{  k^{(n)}\cdot \theta=\tt,\ s^{(n)}\cdot \theta=\varphi^{(n)}(h)(\tt;I_3),\ I=\mathtt I^{(n)}(h)(\tt;I_3)\colon \tt\in\mathbb T\}
\end{equation}
for some real-analytic functions $ \varphi^{(n)}(h)(\cdot;I_3)$ and $\mathtt I^{(n)}(h)(\cdot;I_3)$ that satisfy
\[
|\varphi^{(n)}(h)|_{\C^1}\leq \frac 1n,\qquad\qquad |\mathtt I_i^{(n)}(h)-a_{i}^{(n)}|_{\C^1}\leq \frac 1n,
\]
where $a^{(n)}$ is given in \eqref{eq:anpaths0}, and depend on $I_3$ in a real-analytic fashion;
\item 
We have $\gamma_n(h)(I_3)\subset \{\mathcal H_{\varepsilon_n}(h)=e_n\}$ for all $I_3\in[-\rho_0,\rho_0]$;
\item  
There exists $m\in\mathbb N$ (depending a priori on $n$) and a covering 
\[
[-\rho_0,\rho_0]\subset \bigcup_{0\leq i\leq m} [\cI_i^-,\cI_i^+],
\]
satisfying $\cI_{i+1}^-< \cI_{i}^+$ for all $i=0,\dots,m-1$  such that  for any $i=0,\dots,m-1$ and $I_3\in[\cI_i^-,\cI_i^+]$ there exists a transverse homoclinic orbit 
$z_{n,i}(h)(I_3)$ to $\gamma_n(h)(I_3)$ that depends on $I_3$ in a real-analytic fashion (for $I_3\in[\cI_i^-,\cI_i^+]$).
\end{enumerate}
\end{thm}

The remainder of this section is devoted to obtaining a proof of Theorem \ref{thm:mainparametric}, which follows from Proposition \ref{prop:periodicorbits} and \ref{prop:splitting} below, and is completed in Section \ref{sec:induction}.
\medskip

\subsection{Existence of hyperbolic periodic orbits}\label{sec:hypperorbs}  
The following is the first ingredient in the proof of Theorem \ref{thm:mainparametric}.

\begin{prop}\label{prop:periodicorbits}
   Fix arbitrary $\delta, \rho,\sigma>0$.  Let $H_0\in\mathcal N_\rho^{2,d}(\omega)$ satisfy \eqref{eq:nondegh}, and let $\rho_0=\rho_0(H_0)$ be given by Lemma \ref{lem:resnonatpaths}.   Let  
\begin{equation}\label{eq:truncatedHamiltonianper}
H_{\varepsilon}^{(n)}=H_0+\varepsilon h_0^{(n)},\qquad\qquad \text{where}\qquad\qquad h_0^{(n)}(\theta)=\sum_{j\leq n} b_j\cos(s^{(j)}\cdot\theta)
   \end{equation}
   for $\{b_j\}_{j\leq n}\subset \mathbb R_+$ chosen in such a way that   \begin{equation}\label{eq:smallnesstruncated}
     \sup_{\theta\in \mathbb T^2_\sigma} |h_0^{(n)}(\theta)|<\delta.
   \end{equation}
   Then  there exists $\varepsilon_n>0$
such that for all $0<\varepsilon\leq\varepsilon_n$ and all $I_3\in [-\rho_0,\rho_0]$ the Hamiltonian $H_{\varepsilon}^{(n)}(\cdot;I_3)$ in \eqref{eq:truncatedHamiltonianper} possesses a hyperbolic periodic orbit $ \gamma_n(I_3;\varepsilon)=\gamma_n(I_3)$ that admits a parametrization of the form \eqref{eq:parametrizationperiodicproposition} in terms of some real-analytic functions $\varphi^{(n)}(\cdot;I_3)$ and $\mathtt I_i^{(n)}(\cdot;I_3)$,$i=1,2$, from $\mathbb T$ to $\mathbb R$,  
satisfying
\[
|\varphi^{(n)}|_{\C^1}=O(\varepsilon),\qquad\qquad|\mathtt I^{(n)}_i-a^{(n)}_i|_{\C^1}=O(\varepsilon)
\]
with $a_i^{(n)}$, $i=1,2$, as in \eqref{eq:anpaths0}. 
Moreover, $\gamma_n(I_3)\subset\{H^{(n)}_{\varepsilon_n}=e_n\}$,  and the dependence of $ \gamma_{n}(I_3)$ on  $I_3$ is real-analytic. 

Furthermore, the same conclusion holds for any Hamiltonian of the form
\begin{equation}\label{eq:epssquareHamiltonian}
\mathcal H^{(n)}_\varepsilon(h):=H^{(n)}_\varepsilon+\varepsilon^2 h= H_0+\varepsilon h_0^{(n)}+\varepsilon^2 h
\end{equation}
with $h\in \mathcal B^2_{\rho, \sigma}(\delta)$. We denote the corresponding hyperbolic periodic orbits by $\gamma_n(h)(I_3)$.
\end{prop}

\begin{rem}
    The reason for discussing also the more general class of Hamiltonians \eqref{eq:epssquareHamiltonian} will become clear in Section \ref{sec:transvershomorbits} where we study the existence of transverse homoclinic orbits to $\gamma_n(I_3)$.
\end{rem}

\begin {rem} \label{rem.hn} 
The coefficients $b_n$ may be chosen arbitrarily small. Consequently, although the Fourier frequencies of $h_0^{(n)}$ tend to infinity, their analytic norms on the fixed analyticity domain $\mathbb T^2_\sigma$ can be made uniformly as small as desired. This does not affect the  hyperbolic nature  of the associated resonant dynamics, but only the size of the corresponding hyperbolic eigenvalues. 
%\magenta I don't understand the last sentence \blue It says that resonant dynamics remains hyperbolic but with weaker hyperbolicity. \black
\end{rem} 
\medskip

The rest of Section \ref{sec:hypperorbs} will be devoted to the proof of  Proposition \ref{prop:periodicorbits}.

%\blue Jaime, another idea: first introduce %the coordinate changes, and then formulate %Proposition 4.2. It will become long, but %will contain all the nice expressions: %$\gamma, W^s, W^u$... \black

\subsubsection{A good coordinate system at the resonant surface $ \mathcal R^{(n)} $}
Let  $A_n$ be the matrix \begin{equation}\label{eq:anmatrix}
A_n=\begin{pmatrix}    s^{(n)}_1&k^{(n)}_1\\
   s^{(n)}_2&k^{(n)}_2
\end{pmatrix}
\end{equation}
with $k^{(n)}\in\mathbb Z^2$ as in \eqref{eq:Bezout} (which, by construction, satisfies $\mathrm{det}A_n=1$), and consider the symplectic linear change of variables
$\phi_n:(\varphi,\tt,J,E)\in\mathbb T^2\times\mathbb R^2\to (\theta,I)\in\mathbb T^2\times\mathbb R^2$ given by 
\begin{equation}\label{eq:symplecticlatticechange}
\begin{pmatrix}
 \varphi \\ \tt
\end{pmatrix}=A_n^T\begin{pmatrix}
    \theta_1\\ \theta_2
\end{pmatrix},\qquad\qquad \begin{pmatrix}
J\\ E
\end{pmatrix}=A_n^{-1} \begin{pmatrix}
        I_1\\I_2
    \end{pmatrix}.
\end{equation}
For convenience, we  introduce the notation $\bs J=(J,E)$, denote
\begin{equation}\label{eq:potential}
V^{(n)}=b_n\cos  \varphi,\qquad\qquad P^{(n)}= h_0^{(n)}\circ\phi_n- V^{(n)},
\end{equation}
and recast the Hamiltonian $H^{(n)}_{\varepsilon}$ in the form%\footnote{   We will see below that $P^{(n)}$ contains only non-resonant terms, i.e. it has zero average with respect to $\tt$.}

\begin{equation}\label{eq:pendulumplusperturb}
H_\varepsilon ^{(n)}\circ\phi_n=\underbrace{H_{\mathrm{\mathrm{av}},\varepsilon}^{(n)}(\varphi,J,E)}_{\mathrm{resonant}}+\underbrace{\varepsilon P^{(n)}(\varphi,\tt)}_{\text{non-resonant}}
\end{equation}
with $ H_{\mathrm{\mathrm{av}},\varepsilon}^{(n)}$ being an integrable, pendulum-like Hamiltonian 
\begin{equation}\label{eq:integrablen}
H_{\mathrm{\mathrm{av}},\varepsilon}^{(n)}=\underbrace{H_0 (A_n \bs J)}_{N^{(n)}_{\omega_n}(J,E)}+\varepsilon V^{(n)}(\varphi). 
\end{equation}
We will see below that $P^{(n)}$ contains only non-resonant terms, i.e., it has zero average with respect to $\tt$.
We denote 
$$
N_{\omega_n}^{(n)}(\bs J):=H_0 (A_n \bs J).
$$
\begin{rem}\label{rem:fastangle}
In the new coordinate system, the  unperturbed dynamics (i.e., for $\varepsilon=0$) is given by 
\[
\dot \varphi (t)=\nabla N_{\omega_n}^{(n)}(\bs J(t))\cdot s^{(n)},\qquad\qquad \dot \tau (t)=\nabla N_{\omega_n}^{(n)}(\bs J(t))\cdot k^{(n)},\qquad\qquad \dot J(t)=0=\dot E(t)
\]
An observation that will be used in the sequel is that, since $\det A_n=1$, we have
\begin{equation}\label{eq:sk}
\nabla H_0(A_n \bs J)\cdot k^{(n)}=
-\frac{\partial_{I_1}H_0(A_n\bs J)}{s_2^{(n)}}+\frac{k_2^{(n)}}{s_2^{(n)}}\nabla H_0(A_n \bs J)\cdot s^{(n)}=
\frac{\partial_{I_2}H_0(A_n\bs J)}{s_1^{(n)}}+\frac{k_1^{(n)}}{s_1^{(n)}}\nabla H_0(A_n \bs J)\cdot s^{(n)}
\end{equation}
In particular, on the resonance locus $\nabla H_0\cdot s^{(n)}=0$, the $\tau$-component of the integrable vector field
\begin{equation}\label{eq:lowerboundfastfreq}
\dot \tt=\partial_{E}H^{(n)}\circ\phi_n =\nabla H_0(A_n \bs J)\cdot k^{(n)}=\frac{\partial_{I_2}H_0(a^{(n)}) (I_3)} {s_1^{(n)}}\sim \frac{\omega_2}{s_1^{(n)}}= -\frac{\omega_1}{s_2^{(n)}} 
\end{equation}
is bounded away from zero by a constant which depends only on $n\in\mathbb N$ (in particular, does not vanish for $\varepsilon$ in a small neighborhood of zero). 
\end{rem}
\medskip

\subsubsection{The averaged vector field and approximate periodic orbits} 

Before studying the whole Hamiltonian $ H^{(n)}_{\varepsilon}$, let us first  study the integrable part $H_{\mathrm{\mathrm{av}},\varepsilon}^{(n)}$, whose equations of motion are:
\begin{equation}\label{eq:vectorfieldN}
\begin{aligned}
\dot \varphi & =\partial _J  N_{\omega_n}^{(n)}(\bs J)= \nabla H_0(A_n \bs J)\cdot s^{(n)}\qquad\qquad &\dot J& =\eps b_n  \sin \varphi   \\
\dot \tt & =\partial _E  N_{\omega_n}^{(n)}(\bs J)=  \nabla H_0(A_n \bs J)\cdot k^{(n)}\qquad\qquad &\dot E & = 0 .
\end{aligned}
\end{equation}
Take $\rho_0=\rho_0(H_0)$  given by Lemma \ref{lem:resnonatpaths}.
Then, the vector field \eqref{eq:vectorfieldN} displays, for any $I_3 \in [-\rho_0,\rho_0]$, a hyperbolic periodic orbit which, in coordinates $(\varphi,\tt,J,E)$, can be parametrized as
\begin{equation}\label{eq:unperturbedperorbits}
\gamma^{\mathrm{av}}_{n}(I_3)=\{(0,\tt, \tilde a^{(n)}_J(I_3),\tilde a^{(n)}_E(I_3)) \colon \tt\in\mathbb T\}
\end{equation}
with 
\begin{equation}\label{eq:jn*}
\begin{split}
( \tilde a^{(n)}_J(I_3),\tilde a^{(n)}_E(I_3)) =\tilde a^{(n)} (I_3)=A_n^{-1} a^{(n)}(I_3),
\end{split}
 \end{equation}
where $a^{(n)}(I_3)= ( a^{(n)}_1(I_3), a^{(n)}_2(I_3))  $ stands for the first two components of the resonant path given in \eqref{eq:anpaths0}.
The dynamics on the curve $\gamma^{\mathrm{av}}_{n}(I_3)$ is defined by \eqref{eq:vectorfieldN}. 
Therefore, it is a periodic orbit of period
\begin{equation} \label{eq:nun}
\nu_n(I_3) =  \frac{s_1^{(n)}) } {\partial_{I_2}H_0(a^{(n)}(I_3))} \sim \frac {s_1^{(n)}}{\omega_2}\sim -\frac{s_2^{(n)}}{\omega_1}.
\end{equation}

\begin{comment}
Observe that the vector field \eqref{eq:vectorfieldN} has two first integrals: $E$ and $H_{\mathrm{av},\varepsilon}^{(n) }$. 
Therefore, the stable and unstable manifolds of $\gamma^{\mathrm{av}}_{n}(I_3)$ are given by
\[
E = E_n:=  \tilde  a _E^{(n)} (I_3), \quad H_{\mathrm{av},\varepsilon}^{(n)} 
(J,E,\varphi) = H_{\mathrm{av},\varepsilon}^{(n) } (\tilde a^{(n)} (I_3),0)
\]
that is:
\[
E = E_n:=  \tilde  a _E^{(n)} (I_3), \quad 
N^{(n)}_{\omega_n}(J,E)+\varepsilon V^{(n)}(\varphi)= N^{(n)}_{\omega_n}(\tilde a^{(n)} (I_3))+\varepsilon b_n.
\]
%-s_2^{(n)} a_1^{(n)} (I_3)+ s_1^{(n)} a_2^{(n)}(I_3) \\    = 
Taylor expanding the second equation, and using the fact that 
$ \partial_{J} N^{(n)}_{\omega_n}(\tilde a^{(n)} (I_3))=0$,
we obtain a pendulum-like  homoclinic connection to the periodic orbit $\gamma^{\mathrm{av}}_{n}(I_3)$ given by:
\begin{equation} \label{eq:homoclinicpendulum}
E = E_n , \quad C_n (J-\tilde a_J^{(n)})^2+ \varepsilon b_n (\cos \varphi -1) \simeq 0 , \quad \varphi  \in [0, 2\pi],
\end{equation} 
where $C_n= \frac 1 2 \partial_{J^2} N^{(n)}_{\omega_n}(\tilde a^{(n)} (I_3))$.
\end{comment}

\subsubsection{Structure of the perturbation}

In order to construct hyperbolic periodic orbits for the Hamiltonian $ H_{\varepsilon}^{(n)}$ which continue those in \eqref{eq:unperturbedperorbits} we need to study the structure of the ``perturbing term'' $ P^{(n)}$ in \eqref{eq:potential}.

\begin{lem}\label{lem:lt}
The function $P^{(n)}$ in \eqref{eq:potential} has zero average with respect to $\tt$.
\end{lem}

\begin{proof}
The proof boils down to the observation that for $j\neq n$ we have 
\[
s_1^{(n)}s_2^{(j)}-s_2^{(n)}s_1^{(j)}\neq 0.\qedhere
\]
\end{proof}

\subsubsection{Existence of hyperbolic periodic orbits for the Hamiltonians $ H^{(n)}_{\varepsilon}$ and  $ \mathcal{H}^{(n)}_{\varepsilon}$}
Since the perturbing term  
$P^{(n)}$ in \eqref{eq:potential} has zero average with respect to the fast angle $\tt$, it is well known that, provided that $\varepsilon$ is small enough (depending on $n$), the hyperbolic periodic orbit \eqref{eq:unperturbedperorbits} of the averaged Hamiltonian admits a continuation for the full Hamiltonian $H^{(n)}_{\varepsilon}$. 
In particular, we deduce the following. \black 

\medskip 

\begin{rem}One may compare the parameterization in \eqref{eq:parametrizationdistorted} with that in \eqref{eq:unperturbedperorbits} for the hyperbolic periodic orbits of the Hamiltonian $H_{\mathrm{\mathrm{av}},\varepsilon}^{(n)}$.
\end{rem}

\begin{lem}\label{lem:auxiliaryperiodicorbits2}
   Fix any $\delta, \rho,\sigma>0$, and let $H_0\in\mathcal N_\rho^{2,d}(\omega)$ satisfy \eqref{eq:nondegh} with $\rho_0=\rho_0(H_0)$ given by Lemma \ref{lem:resnonatpaths}.  
   Choose any sequence  $\{b_j\}_{j\leq n}\subset\mathbb R_+$ such that $h_0^{(n)}$ defined by \eqref{eq:truncatedHamiltonianper} satisfies \eqref{eq:smallnesstruncated}. 
   Then, there exists $\varepsilon_n>0$ such that, for any $0<\varepsilon\leq \varepsilon_n$ and any $I_3\in[-\rho_0,\rho_0]$, the Hamiltonian $H^{(n)}_{\varepsilon}(\cdot;I_3)$ in \eqref{eq:truncatedHamiltonianper}  possesses a hyperbolic periodic orbit 
   $\gamma_n(I_3)\subset \{H_{\varepsilon}=e_n\}$ 
   \color{black} which, in the coordinate system introduced in \eqref{eq:symplecticlatticechange}, admits a parametrization of the form 
\begin{equation}\label{eq:parametrizationdistorted}
\gamma_n(I_3)=\{(\varphi^{(n)}(\tt;I_3),\tt,J^{(n)}(\tt;I_3),E^{(n)}(\tt;I_3))\colon \tt\in\mathbb T\}
\end{equation}
for some real-analytic functions satisfying
\begin{equation}\label{eq:parametrizationdistortedestimates}
|\varphi^{(n)}|_{\C^1}=O(\varepsilon),\qquad\qquad |J^{(n)}-\tilde a^{(n)}_J|_{\C^1}=O(\varepsilon),\qquad\qquad |E^{(n)}-\tilde a^{(n)}_E|_{\C^1}=O(\varepsilon)
\end{equation}
with $\tilde a_J^{(n)}, \tilde a_E^{(n)}$ as in \eqref{eq:jn*}. 
Moreover, the same conclusion holds for any Hamiltonian $\mathcal{H}_\varepsilon^{(n)}$ of the form \eqref{eq:epssquareHamiltonian} and we denote the corresponding periodic orbits by $\gamma_n(h)(I_3)$.
\end{lem}

The proof of Lemma \ref{lem:auxiliaryperiodicorbits2} is obtained by a standard averaging argument and is  deferred to Appendix \ref{sec:appendixhyperbolic}.\black 

\subsubsection{Proof of Proposition \ref{prop:periodicorbits}}

With Lemma \ref{lem:auxiliaryperiodicorbits2} at hand, we complete the proof of  Proposition \ref{prop:periodicorbits} by setting 
\[
\binom{\mathtt I^{(n)}_1}{\mathtt I^{(n)}_2}=A_n \binom{J^{(n)}}{E^{(n)}}.
\]
\medskip

\subsection{Existence of transverse homoclinic points}\label{sec:transvershomorbits}
  In Proposition \ref{prop:periodicorbits}, for any $n\in\mathbb N$,  we have established the existence of $\varepsilon_n>0$ such that for any $0<\varepsilon\leq \varepsilon_n$ and any $I_3\in[-\rho_0,\rho_0]$,  any Hamiltonian $\mathcal H_{\varepsilon}^{(n)}(h)(\cdot;I_3)$ as in \eqref{eq:epssquareHamiltonian} admits a hyperbolic periodic orbit $\gamma_n(h)(I_3)$. 
    Moreover, in \eqref{eq:parametrizationperiodicproposition} we gave a parametrization of this orbit which depends on   $I_3$ in a real-analytic fashion. 
    %Let us recall here that this periodic was obtained as a perturbation of the periodic orbit $\gamma_n^{\mathrm{av}}(I_3)$ (see \eqref{eq:unperturbedperorbits}) of the averaged system in \eqref{eq:integrablen}. An important observation is that the periodic orbit $\gamma_n^{\mathrm{av}}(I_3)$ had an homoclinic connection (see \eqref{eq:homoclinicpendulum}). It is then  natural to understand the behaviour of the stable and unstable manifolds of the perturbed periodic orbit $\gamma_n(h)(I_3)$.
\medskip
    
    In the following result we show that, for a generic $h\in \mathcal B_{\rho, \sigma}^2(\delta)$, this hyperbolic periodic orbit admits a rich family of transverse homoclinic orbits. 

\begin{prop}\label{prop:splitting}
 Fix any $\delta, \rho,\sigma>0$, let $H_0\in\mathcal N_\rho^{2,d}(\omega)$ satisfy \eqref{eq:nondegh} with $\rho_0=\rho_0(H_0)$ given by Lemma \ref{lem:resnonatpaths}, and choose any sequence  $\{b_j\}_{j\leq n}\subset\mathbb R_+$ such that $h_0^{(n)}$ defined by \eqref{eq:truncatedHamiltonianper} satisfies \eqref{eq:smallnesstruncated}.   There exists $\varepsilon_n>0$  such that for any $0<\varepsilon\leq \varepsilon_n$ there exists an open and dense  subset $\mathcal U^{(n)}_{\rho,\sigma}(\delta)\subset \mathcal B^2_{\rho,\sigma}(\delta)$ such that the following holds.  
 
 For any  $h\in \mathcal U^{(n)}_{\rho,\sigma}(\delta)$ there exists $m\in\mathbb N$ (depending a priori on $n$ and $h$) and a covering 
 \[
 [-\rho_0,\rho_0]\subset \bigcup_{0\leq i\leq m} [\cI_i^-,\cI_i^+]
 \]
satisfying $\cI_{i+1}^-<\cI_{i}^+$ for all $i=0,\dots,m-1$, such that  for any $i=0,\dots,m-1$ and $I_3\in[\cI_i^-,\cI_i^+]$ the Hamiltonian $\mathcal H^{(n)}_{\varepsilon}(h)(\cdot;I_3)$ in \eqref{eq:epssquareHamiltonian} admits a  transverse\footnote{Within the corresponding 3-dimensional energy level.} homoclinic orbit 
 $z_{n,i}(h)(I_3)$ to $\gamma_n(h)(I_3)$ which depends on $I_3$ in a real-analytic fashion for $I_3\in[\cI_i^-,\cI_i^+]$. 
\end{prop}

The proof of Proposition \ref{prop:splitting} is presented in Section \ref{sec:mainsptlittingsection} as it involves a number of steps. We now make use of an inductive argument to deduce Theorem \ref{thm:mainparametric} from Propositions \ref{prop:periodicorbits} and \ref{prop:splitting}.

\medskip

\subsection{Proof of Theorem \ref{thm:mainparametric}}\label{sec:induction}

Given $r>0$ and $H\in\mathcal Q_{\rho,\sigma}^{2,d}$ we let $\mathcal B_{\rho,\sigma}^2(r;H)\subset \mathcal Q_{\rho,\sigma}^{2,d}$ be the ball of radius $r$ centered at $H$. The proof relies on the following definition.

\begin{defn} Let $\rho_0=\rho_0(H_0(I;I_3))$ be given by Lemma \ref{lem:resnonatpaths}.
    Let $n\in\mathbb N$, $\zeta\geq 0$ and $r>0$. We say that a Hamiltonian $H_*\in\mathcal Q_{\rho,\sigma}^{2,d}$ is $(n,\zeta,r)$\textit{-good} if there exists a residual subset $\mathcal U(n,r,H_*)\subset \mathcal B_{\rho,\sigma}^2(r;H_*)$ such that any $H\in \mathcal U(n,r,H_*)$ satisfies the following:
    \begin{enumerate}
        \item For any $I_3\in [-\rho_0,\rho_0]$ the Hamiltonian $H(\cdot;I_3)$ displays a hyperbolic periodic orbit $\gamma_n(I_3)$ which is $\zeta$-close in the $\C^1$ topology to the submanifold $\{I=a^{(n)}(I_3)\}$;
        \item There exists $m\in\mathbb N$ (depending a priori on $n$) and a covering $
[-\rho_0,\rho_0]\subset \bigcup_{0\leq i\leq m} [I_i^-,I_i^+]$
satisfying $I_{i+1}^-<I_{i}^+$ for all $i=0,\dots,m-1$  such that  for any $i=0,\dots,m-1$ and $I_3\in[I_i^-,I_i^+]$ there exists a transverse homoclinic orbit 
 $z_{n,i}(I_3)$ to $\gamma_n(I_3)$ which depends on $I_3$ on a real-analytic fashion (for $I_3\in[I_i^-,I_i^+]$).
    \end{enumerate}
\end{defn}

\noindent\textit{Initialization.} Observe that for $n=1$, by Propositions \ref{prop:periodicorbits} and \ref{prop:splitting}, there exists $0<\varepsilon_1,\zeta_1<1$ such that  the Hamiltonian $H_{\varepsilon_1}^{(1)}$ of the form \eqref{eq:truncatedHamiltonianper} is $(1,\zeta_1,r_1^{(1)})$-good with $r_1^{(1)}=\varepsilon_1^2 \delta$.
\medskip

\noindent\textit{The inductive hypothesis.} Let  $n\in\mathbb N$ and suppose that we have found 
$0<\varepsilon_j,\zeta_j<1/j$ and $ \delta^{(n)}>\delta/2$ 
such that for all $j\leq n$ the Hamiltonian $H^{(n)}_{\varepsilon_j}$ of the form \eqref{eq:truncatedHamiltonianper} is $(j,\zeta_j,r_{j}^{(n)})$-good with $r_j^{(n)}=\varepsilon_j^2 \delta^{(n)}$. 
\medskip

\noindent\textit{The inductive step.} Assume the inductive hypothesis are true up to step $n\in\mathbb N$. Then, choosing $b_{n+1}>0$ small enough, it is clear that there exists some $\delta^{(n+1)}>\delta/2$ such that the Hamiltonian $H^{(n+1)}_{\varepsilon_j}$ is $(j,\zeta_j,r_j^{(n+1)})$-good with $r_j^{(n+1)}=\varepsilon_j^2 \delta^{(n+1)}$. On the other hand, by Propositions \ref{prop:periodicorbits} and \ref{prop:splitting}, there exists $
0<\varepsilon_{n+1},\zeta_{n+1}<1/(n+1)$ such that the Hamiltonian $H_{\varepsilon_{n+1}}^{(n+1)}$ is $(n+1,\zeta_{n+1},r_{n+1}^{(n+1)})$-good with $r_{n+1}^{(n+1)}=\varepsilon_{n+1}^2 \delta^{(n+1)}$.
\medskip

\noindent\textit{The limit Hamiltonian.}  The Cauchy sequence $\{H_{\varepsilon}^{(n)}\}_n\subset \mathcal Q_{\rho,\sigma}^{2,d}$ converges uniformly to a Hamiltonian $H_\varepsilon\in \mathcal Q_{\rho,\sigma}^{2,d}$ which, by construction, satisfies that for all $n\in\mathbb N$  the Hamiltonian $H_{\varepsilon_n}$ is $(n,\zeta_n,r_n)$-good with $r_n=\varepsilon_n^2\delta/2$.
\medskip

The proof of Theorem \ref{thm:mainparametric} is complete.

\medskip

\subsection{Proof of Theorem \ref{thm:mainparametricoutline}: normally hyperbolic cylinders with twist}\label{sec:Hamtomaps}
%\noindent {\it Proof of Theorem \ref{thm:mainparametricoutline}. }
We finally complete the proof of Theorem \ref{thm:mainparametricoutline} using Theorem \ref{thm:mainparametric}. 
To do so, given $h\in\mathcal U_{\rho,\sigma}(\delta) $ (this is the set constructed in Theorem  \ref{thm:mainparametric}), we define the 3 degrees-of-freedom Hamiltonian
\begin{equation}\label{eq:3dofHam}
\mathcal G_\varepsilon(h)(\theta_1,\theta_2,I_1,I_2,I_3)=\mathcal H_{\varepsilon}(h)(\theta_1,\theta_2,I_1,I_2;I_3)
\end{equation}
with $\mathcal H_\varepsilon(h)$ as in \eqref{eq:firstperturbH}.  Note that $\mathcal G_\varepsilon(h)$ thus defined has exactly the form (5.8)-(5.9).

\medskip

\noindent Item (0) of Theorem \ref{thm:mainparametricoutline} holds by construction, since $\mathcal G_\eps (h)=G_0+\eps g_0 +\eps^2 h$ where $h\in \cU_{\rho, \sigma}(\delta)$ and the norm of $g_0$ can be made smaller than $\de$ by an appropriate choice of the sequence $\{b_j\}$. 

\medskip

\noindent (1)  
The Hamiltonian $\mathcal G_\varepsilon$  has the form 
$$
\mathcal G_{\varepsilon}(\theta_1,\theta_2,I_1,I_2,I_3) = G_0(I_1, I_2,I_3) + O_{\C^2}(\varepsilon) .
$$
(In fact, the remainder can be assumed to be $O_{\C^r}(\varepsilon)$ for any $r\geq 2$, but $r=2$ will suffice for our arguments.)
The corresponding vector field  keeps $I_3$ constant and, on the resonant surface $\mathcal  R^{(n)}$ (see \eqref{eq:resonancelocus}), the restricted vector field associated to the integrable Hamiltonian $G_0$ verifies (see \eqref{eq:sk} and \eqref{eq:lowerboundfastfreq}})
\[
k_1^{(n)}\dot\theta_1+k_2^{(n)}\dot \theta_2= k_1^{(n)}\partial_{I_1} G_0+k_2^{(n)}\partial_{I_2} G_0%=(k_1^{(n)}s_2^{(n)}-k_2^{(n)}s_1^{(n)}) \frac{1}{s_2^{(n)}}\partial_{I_1}G_0
=-\frac{1}{s_2^{(n)}}\partial_{I_1}G_0\neq 0.
\]
Hence, the Poincar\'e map  in $\Sigma_n$ is locally well-defined around $ \mathcal R^{(n)}$.

\medskip

\noindent (2) 
 Let $\gamma_n(I_3)$ be the family of hyperbolic periodic orbits (for $\mathcal H_\varepsilon$) in  Theorem \ref{thm:mainparametric}. Then it follows from our construction that the union of $\gamma_n(I_3)$ over $(\theta_3,I_3)\in \mathbb T\times [-\rho_0,\rho_0]$ forms an invariant cylinder
 $\tilde\Lambda_n$ of the form 
 \[
\tilde\Lambda_n= \{  (\tilde\theta_1^{(n)}(\theta_2,I_3),\theta_2,\theta_3, \tilde I_1^{(n)}(\theta_2,I_3), \tilde I_2^{(n)}(\theta_2,I_3),I_3)\colon \ (\theta_2,\theta_3)\in\mathbb T^2,\  I_3\in[-\rho_0,\rho_0]\}\subset \{\mathcal G_{\varepsilon_n}(h)=e_n\}
 \]
 for certain real-analytic functions $\tilde \theta_1^{n)},\tilde I_1^{(n)},\tilde I_2^{(n)}$ satisfying
 \[
|\tilde\theta_1^{(n)}+\frac{s_2^{(n)}}{s_1^{(n)}}\theta_2|_{\C^1}\leq \frac1n,\quad |\tilde I_1^{(n)}-a_{1}^{(n)}|_{C^1}\leq \frac 1n,\quad  |\tilde I_2^{(n)}-a_{2}^{(n)}|_{C^1}\leq \frac 1n
 \]
   with $a_1^{(n)},a_2^{(n)}$ as in  \eqref{eq:anpaths0}-\eqref{eq:anpaths}. 
We define the cylinder $\Lambda_n$ in \eqref{eq:2dcylinderoutline} as
 \[
 \Lambda_n=\widetilde\Lambda_n\cap \{k_1^{(n)}\theta_1+k_2^{(n)}\theta_2=0\}.
 \]

\noindent(3)  
In the following, we denote  
\[
(\OLDtau_1(I),\OLDtau_2(I),\OLDtau_3(I)):=\nabla G_0(I),
\]
and write $  I^{(n)}(I_3)=(  I_1^{(n)}(I_3),  I_2^{(n)}(I_3),I_3)$.  By the calculation in item (1) above, combined with the fact that the vector field keeps $I_3$ constant, the return time (under the integrable dynamics) to the Poincar\'e section $\Sigma_n$ on the cylinder is  (see \eqref{eq:nun})
\[
\nu_n(I_3):=\frac{ s_2^{(n)}}{\OLDtau_1 (I^{(n)}(I_3))}=-\frac{ s_1^{(n)}}{\OLDtau_2 (I^{(n)}(I_3))}.
\]
Therefore, we observe that the  restriction  of  the Poincar\'e map $\Phi_{n}$ to the  two-dimensional cylinder $\Lambda_n$  in \eqref{eq:2dcylinderoutline} is given by 
\begin{equation}%\label{eq:twist}
\Phi_{n}|_{\Lambda_n}:(\th_3, I_3)\mapsto (\th_3  - \nu_n (I_3) \OLDtau_3  (  I^{(n)}(I_3))+ O_{\C^1}(\varepsilon),\ I_3):= (\th_3  + \tilde \nu_n (I_3, \varepsilon),\ I_3). 
\end{equation}
  Hence, if we denote
  \[
  L\OLDtau_1= \nabla \OLDtau_1 \cdot \frac{\mathrm{d}I^{(n)}}{\mathrm dI_3},\qquad\qquad L\OLDtau_3= \nabla \OLDtau_3 \cdot \frac{\mathrm{d}I^{(n)}}{\mathrm dI_3},
  \]
  and let $\pi_{\theta_3}\Phi_{n}$ stand for the $\theta_3$-component of $\Phi_n$,  the twist of the map $\Phi_{n}$ is given by the following expression:
\[
\partial_{I_3}( \pi_{\theta_3}\Phi_{n})(\th_3, I_3)=\frac{2\pi s_2^{(n)} }{\OLDtau_1^2}(\underbrace{\OLDtau_1  L\OLDtau_3-\OLDtau_3 L\OLDtau_1\black}_{T_n(I_3)})+O(\varepsilon), 
\]
where the right hand side is evaluated at $I=I^{(n)}(I_3)$.   
Denote
$$
T_n(I_3):= \OLDtau_1  L\OLDtau_3-\OLDtau_3 L\OLDtau_1 .
$$
In the following, we show that $T_n(I_3)\neq 0$ for all $I_3$ sufficiently close to the origin. 
To do so, the first step is to compute the vector $\mathrm d I^{(n)}/\mathrm d I_3$.
From the definition of  the functions $  I^{(n)}_1$ and   $I^{(n)}_2$,  a straightforward computation shows that at $I=  I^{(n)}(I_3)$ we have:
\[
\binom{-\OLDtau_3}{-s_1^{(n)}\partial_{3}\OLDtau_1-s_2^{(n)}\partial_3\OLDtau_2}=\underbrace{\begin{pmatrix} \OLDtau_1&\OLDtau_2\\ s_1^{(n)} \partial_1 \OLDtau_1+s_2^{(n)}\partial_1\OLDtau_2& s_1^{(n)} \partial_2 \OLDtau_1+s_2^{(n)}\partial_2\OLDtau_2\end{pmatrix}}_{B_n}\binom{\frac{\mathrm d   I^{(n)}_1}{\mathrm d I_3}}{\frac{\mathrm d   I^{(n)}_2}{\mathrm d I_3}},
\]
We have already checked in Lemma \ref{lem:resnonatpaths} that, for an open and dense subset of $G_0\in\mathcal N_\rho^{3,d}(\omega)$ we have $\mathrm{det}(B_n)\neq 0$. Moreover,   
\begin{align*}
\binom{\frac{\mathrm d   I^{(n)}_1}{\mathrm d I_3}}{\frac{\mathrm d   I^{(n)}_2}{\mathrm d I_3}}=& B_n^{-1}\binom{-\OLDtau_3}{-s_1^{(n)}\partial_{3}\OLDtau_1-s_2^{(n)}\partial_3\OLDtau_2}=\frac{1}{\mathrm{det}(B_n)}\binom{
-\OLDtau_3(s_1^{(n)}\partial_2\OLDtau_1+s_2^{(n)}\partial_2\OLDtau_2)+\OLDtau_2(s_1^{(n)}\partial_3\OLDtau_1+s_2^{(n)}\partial_3\OLDtau_2)}{\OLDtau_3(s_1^{(n)}\partial_1\OLDtau_1+s_2^{(n)}\partial_1\OLDtau_2)-\OLDtau_1(s_1^{(n)}\partial_3\OLDtau_1+s_2^{(n)}\partial_3\OLDtau_2)},
\end{align*}
where all the functions are evaluated at $I=  I^{(n)}(I_3)$. Using the equalities $\partial_i \OLDtau_j=\partial_j \OLDtau_i$ for $i=1,2,3$, one can deduce by a calculation that the following holds:
\[
\frac{\mathrm{d} I^{(n)}}{\mathrm{d }I_3}=-\frac{1}{\mathrm{det}(B_n)}\  D^2 G_0\  \bs s^{(n)}\wedge \nabla G_0.
\]
Hence, for $i=1,3$ \black we have:
\[
\mathrm{det} (B_n) L\OLDtau_i=  D^2 G_0 e_i\cdot (D^2G_0 \bs s^{(n)}\wedge \nabla G_0 )=\lVert \bs s^{(n)}\rVert  D^2 G_0 e_i\cdot (D^2G_0 \bs v^{(n)}\wedge \nabla G_0 ),
\]
where $e_1=(1,0,0)$, $e_3=(0,0,1)$ and we have defined $\bs v^{(n)}=\bs s^{(n)}/\lVert \bs s^{(n)}\rVert$. 
Then, if we assume that 
\[
T: = D^2G_0(0)\begin{pmatrix} -\omega_3\\0\\ \omega_1 \end{pmatrix}\  \cdot\  \left(D^2 G_0(0) \bs v \wedge \omega \right)  \neq 0, 
\]
where $\bs v= (-\omega_2,\omega_1,0)$,
which is, obviously, an open and dense property, we get $T_n(I)\neq 0$ on a sufficiently small ball around the origin.
\medskip

(4) This fact follows from  Proposition  \ref{prop:splitting} and the definition of the Poincar\'e section $\Sigma_n$.\qedhere
%\end{proof}

We have completed the proof of Theorem \ref{thm:mainparametricoutline}.

\medskip

 \section{Diffusion along  homoclinic cylinders: proof of Theorem \ref{thm:_apply_GToutline}}\label{sec:diffusion}

In this section we complete the proof of Theorem \ref{thm:_apply_GToutline}, which follows plainly from Lemma \ref{lem:openballcylinders} and Proposition \ref{prop_apply_GT} below.

\subsection{Arnold-Moeckel's \black diffusion mechanism}  
The proof of Theorem \ref{thm:_apply_GToutline} relies on ideas developed by Moeckel \cite{MR1884898} in the context of Arnold diffusion  and later extended by  Gelfreich and Turaev \cite{GelfreichTuraevChaotic}.

\subsection*{A family of 3 degrees-of-freedom Hamiltonians with diffusion}
Fix $\rho,\sigma>0$. Given $\delta>0$, let 
 $\mathcal U_{\rho,\sigma}(\delta)$ be the set in Theorem \ref{thm:mainparametricoutline} and, for $g_1\in \mathcal U_{\rho,\sigma}(\delta)$, let   $\mathcal G_{\varepsilon}(g_1)\in \mathcal Q^{3,d}_{\rho,\sigma}$ be as in Theorem \ref{thm:mainparametricoutline}.  Then, given  $g_2\in \mathcal B^3_{\rho,\sigma}(\delta)$ and $\mu\in\mathbb R$, we let 
\begin{equation}\label{eq:muhamiltonian}
\mathcal G_{\varepsilon,\mu}(g_1,g_2)=\mathcal G_{\varepsilon}(g_1)+\mu g_2.
\end{equation}

Our first observation is that for $\mu>0$ small enough and $g_2\in \mathcal B^3_{\rho,\sigma}(\delta)$, the cylinders $\Lambda_n$ in %Proposition \ref{prop:Poincare_map} 
Theorem \ref{thm:mainparametricoutline} 
admit a {\it continuation} as normally hyperbolic cylinders for the flow of the Hamiltonian \eqref{eq:muhamiltonian}.

%\blue Our first observation is that for $\mu>0$ small %enough and $g_2\in \mathcal B^3_{\rho,\sigma}(\delta)$, %the Hamiltonian \eqref{eq:muhamiltonian} admits normally %hyperbolic cylinders that are $\mu$-close to $\Lambda_n$ %obtained  in Proposition \ref{prop:Poincare_map}. \black

\begin{rem}
   By {\it continuation} we mean that 
   these invariant objects are deformed smoothly when the parameter $\mu$ varies close to zero. In the case where the normally hyperbolic manifold under consideration is a compact manifold without boundary, the smoothness of the continuation  is a consequence of the classical results \cite{MR501173}. Observe, however, that $\Lambda_n$ are manifolds with boundary so, a priori, this continuation is not unique.  In order to have uniqueness of the continuation we will proceed as in \cite{GelfreichTuraevChaotic} and assume from now on that the  boundaries  of $\Lambda_n$  correspond to curves at which the rotation number is Diophantine (since in this case one can guarantee that the continuation is also  unique).
\end{rem}
\black

\begin{lem}\label{lem:openballcylinders}
    Fix any $k\in\mathbb N$. There exists  a positive sequence $\mu_n\to 0$ such that, for any $n\in\mathbb N$, any $g_2\in \mathcal B^3_{\rho,\sigma}(\delta)$ and all $0<\mu\leq\mu_n$, the Hamiltonian \eqref{eq:muhamiltonian} has a  normally hyperbolic invariant cylinder $\Lambda_n(g_1,g_2;\mu)$ with a family of 
$\C^k$ homoclinic cylinders $\{\Gamma_{n,i}(g_1, g_2;\mu)\}_{i=0}^m$ \black corresponding to the unique continuation of $\Lambda_n (g_1)$ and $\{\Gamma_{n,i} (g_1)\}_{i=0}^m$ (with $\Lambda_n (g_1),\Gamma_{n,i}(g_1)$ as in Theorem \ref{thm:mainparametricoutline}).
\end{lem}

\begin{proof}
We start by observing  that, for the normally hyperbolic compact invariant cylinders $\Lambda_n (g_1)$ in \eqref{eq:2dcylinderoutline},
%{eq:2dcylinder}, 
the restricted dynamics is such that the rate of expansion/contraction is, at most, linear. Indeed,$\Phi_n|_{\Lambda_n}$  is a symplectic twist map preserving the foliation $I_3=\mathrm{const}$ (see Theorem \ref{thm:mainparametricoutline}); hence,  %Proposition\ref{prop:Poincare_map})
in coordinates $(\theta_3,I_3)\in \mathbb T\times[-\rho_0,\rho_0]$, it must be of the form (see \eqref{eq:twist}
\[
\Phi_n|_{\Lambda_n}(\theta_3,I_3)=(\theta_3+\tilde\nu_n(I_3),I_3)
\]
for some real-analytic function $\tilde \nu_n:[-\rho_0,\rho_0]\to \mathbb R$, see \eqref{eq:twist}. Hence, its differential is a unipotent matrix, and for any iterate $k\in \mathbb Z$ we have 
\[
D\Phi_n^k|_{\Lambda_n}(\theta_3,I_3)=\begin{pmatrix}
    1&k \tilde \nu_n'(I_3)\\
    0&1
\end{pmatrix} .
\]
\black
Controlling the ratio between normal expansion/contraction and tangential expansion (i.e., for the restricted dynamics)  is crucial to control the regularity of the continuation of the normally hyperbolic invariant manifold $\Lambda_n$ (see \cite{MR501173}). \black

Fix any $k\in\mathbb N$. 
Then,  by virtue of the above observation, given $n\in\mathbb N$ for $\mu$ sufficiently small (depending on $k$ and $n$), for any $g_2\in\mathcal B_{\rho,\sigma}^3(\delta)$, the Hamiltonian \eqref{eq:muhamiltonian} has a $\C^k$ normally hyperbolic  invariant cylinder $\Lambda_n(g_1, g_2;\mu)$ which is $O_{\C^k}(\mu)$ close to $\Lambda_n(g_1)$ (see, for example, \cite{MR501173}). Since, moreover, the stable and unstable invariant manifolds of $\Lambda_n$ behave regularly under smooth perturbations, the proof is now a straightforward consequence of the compactness of the family of cylinders $\{\Gamma_{n,i}\}_{i=0}^m$.
\end{proof}

We now construct a large set of perturbations $g_2$ such that the  Hamiltonians $\mathcal G_{\varepsilon,\mu}(g_1,g_2)$ as in \eqref{eq:muhamiltonian} exhibit orbits drifting along the cylinder $\Lambda_n(g_2; \mu)$ obtained in Lemma \ref{lem:openballcylinders}. 
Given a cylinder $\Lambda_n$ as in \eqref{eq:2dcylinderoutline} we denote by $\partial^{\pm}\Lambda_n$  the circles  $\Lambda_n\cap \{I_3= \pm\rho_0  \}$.

\begin{prop}\label{prop_apply_GT}
Let $\mu_n$ be the sequence in Lemma \ref{lem:openballcylinders}. For any $n\in\mathbb N$  there exists an open and dense set $\mathcal V^{(n)}_{\rho,\sigma} (\delta) \subset \mathcal B^3_{\rho,\sigma}(\delta)$ such that for any $g_2\in\mathcal V^{(n)}_{\rho,\sigma}$   
the Hamiltonian in \eqref{eq:muhamiltonian} satisfies the following.

If $U^\pm_n$ is an open neighborhood of $\partial^{\pm} \Lambda_n(g_1)$, then there exists an orbit of the Hamiltonian $\mathcal G_{\varepsilon_n,\mu_n}(g_1, g_2)$ \black  connecting $U^-_n$ to $U^+_n$.
\end{prop}
The result of Proposition \ref{prop_apply_GT} follows directly from Theorem 2 and Remark 3 in the paper \cite{GelfreichTuraevChaotic}.  We now verify that all the hypotheses needed to apply  Theorem 2 of  \cite{GelfreichTuraevChaotic} are satisfied  as a consequence of Theorem \ref{thm:mainparametricoutline}.
%Proposition \ref{prop:Poincare_map}. \black
\medskip

\subsection{Proof of Proposition \ref{prop_apply_GT}}\label{sec:appendixGT}

 Here we verify that all the assumptions needed to apply Theorem 2 in \cite{GelfreichTuraevChaotic} are satisfied in the setting of Proposition \ref{prop_apply_GT}. These are assumptions on  the Poincar\'e map $\Phi_n:\Sigma_n\to \Sigma_n$ induced on the four-dimensional section
 \begin{equation}\label{eq:4dsec}
 \Sigma_n=\{k_1^{(n)}\theta_1+k_2^{(n)}\theta_2=0,\ \mathcal{G}_{\varepsilon_n}(\theta,I)=e_n\},
 \end{equation}
 introduced in Theorem \ref{thm:mainparametricoutline}.
 %Proposition \ref{prop:Poincare_map}.
 Recall that in Theorem \ref{thm:mainparametricoutline}
 %Proposition \ref{prop:Poincare_map}
 we have shown that the map $\Phi_n$ exhibits a two-dimensional invariant cylinder $\Lambda_n$ on which the restricted dynamics $\Phi_n|_{\Lambda_n}$ is given by an integrable twist map. 
 
 \subsection*{Notation and definitions}
 We start by recalling some of the definitions used in \cite{GelfreichTuraevChaotic}. Below $N=\mathbb T^2\times V$  where $V\in\mathbb R^2$ is an open set containing the origin.

\begin{defn}
%symmetrically normally-hyperbolic
A smooth invariant  two-dimensional cylinder $\Lambda$ of a map $\Phi:N\to \mathbb T^2\times\mathbb R^2$ is called {\it symmetrically normally-hyperbolic} if 
at each point $v\in \Lambda$ the tangent space is decomposed into a direct sum of three non-zero subspaces:
$T_vN= N_v^{c}\oplus N_v^{u}\oplus N_v^{s}$, where $N_v^{c}$ is the two-dimensional plane tangent to $\Lambda$ at the point $v$. The subspaces $N_v^{u,s}$ depend continuously on $v$ and are invariant with respect to the derivative $D \Phi$ of the map. We assume that for some choice of norms in $N_v^{c,u,s}$ there exist $\alpha>1$ and $\lambda \in (0,1)$ such that for every $v\in \Lambda$ we have:
\begin{equation}\label{eq:hyperbolicityestimates}
\begin{split}
\|D\Phi (v)_{\mid{N_v^c}}\|<\alpha, \quad &\|(D \Phi (v))^{-1}_{\mid {N_v^c}}\|<\alpha, \\
\|D\Phi(v)_{\mid_{N_v^{s}}}\|<\lambda, \quad &\|(D\Phi(v))^{-1}_{\mid_{N_v^{u}}}\|<\lambda,
\end{split}
\end{equation} 
where $\alpha^2 \lambda<1$.
The word {\it symmetrically} in this definition refers to the fact that in the estimates above {\it the same} pair of exponents, $\alpha$ and $\lambda$, bound both $D\Phi $ and
$(D\Phi )^{-1}$.
    
\end{defn}

\begin{defn}% KAM-curve
A smooth
invariant essential (non-contractible to a point) simple curve $\gamma\subset \Lambda$ is called  {\it KAM-curve} of the map $\Phi$ if  the map $\Phi$ restricted to $\gamma$ is
smoothly conjugated to the rigid rotation by a Diophantine frequency and the map $\Phi _{\mid_\Lambda}$
near $\gamma$ satisfies the twist condition. 
\end{defn}

\begin{defn} % homoclinic cylinder, simple relative to the cylinders $\bA$ and $A$
Let $\bA\subset \text{int} (\Lambda)$ be a compact invariant sub-cylinder in $\Lambda$, i.e., it is a closed region
in $\text{int} (\Lambda)$ bounded by two non-intersecting invariant essential simple curves $\gamma^+$ and $\gamma^-$. 
Let the set of points homoclinic to $\Lambda$ contain a smooth two-dimensional manifold
$\Gamma \subset W^u (\text{int} (\Lambda)) \cap W^s \text{int} (\Lambda) ) \setminus \Lambda$.  Then, $\Gamma$ is called  {\it homoclinic cylinder, simple relative to the cylinders $\bA$ and $\Lambda$}, if the following assumptions hold:
\begin{enumerate}[label={[S\arabic*]}]
\item  The strong transversality condition 
\[
T_y E^{ss}_y\oplus T_y E^{uu}_y\oplus T_y(W^s(\Lambda)\cap W^u(\Lambda))=\mathbb R^{4}
\]
holds for all $y\in \Gamma$. Here, for $y\in \Gamma$ we have denoted by $E^{ss}_y$ (resp. $E^{uu}_y$) the unique leave of the strong stable (resp. unstable) invariant foliation passing through $y$;

\item For every point $x\in \bar\Lambda$, the corresponding leaf of the foliation $E^{uu}$ intersects
the homoclinic cylinder $\Gamma$ at exactly one point each, and no two points in $\Gamma$
belong to the same leaf of the stable foliation, that is, if $y_1\in W^s_{x} \cap \Gamma$, $y_2\in W^s_{x} \cap \Gamma $, then $y_1=y_2$. In other words, the scattering map
\begin{equation}
F_\Gamma = \Omega^{+}_\Gamma \circ (\Omega _\Gamma^-)^{-1}: \bA \to \text{int}(\Lambda) 
\end{equation}
is well-defined; 

\item The set $F_\Gamma (\bA)$, image of $\bA$  by the scattering map $F_\Gamma $ contains an essential curve.

\end{enumerate}
\end{defn}

Given a real-analytic Hamiltonian $G:\mathbb T^3\times B_\rho\to \mathbb R$ we denote by $\Phi$  its induced Poincar\'e map on a transverse four-dimensional section $\Sigma_{G}$ contained in a constant energy level $\{G=e\}$. 
Given $\rho,\sigma>0$, Gelfreich and Turaev consider the set $\cV_{\rho,\sigma}$ of real-analytic Hamiltonians $G: \mathbb T^3_\sigma\times \mathbb B_\rho\to \CC$ for which the associated Poincar\'e maps
\begin{equation}\label{eq:poincmapfinal}
\Phi:= \Phi(G):\Sigma_G\to \Sigma_G
\end{equation}
on $\Sigma_G$
%\simeq N$\footnote{$\Sigma_H$ can be any real-analytic deformation of $\Sigma_{G_0}$ contained in a constant energy level of $G$.} 
satisfy the following:
\begin{enumerate}
\item 
Each map $\Phi$ with $G \in \cV_{\rho,\sigma}$  has an invariant, bounded by KAM-curves, symmetrically normally-hyperbolic, two-dimensional closed cylinder $\Lambda$;

\item  In $\Lambda$ there exist two invariant sub-cylinders $\bA$ and $\hA$ such that
$\bA\subset \text{int} (\hA)\subset \text{int} (\Lambda)$ each of the cylinders being bounded by KAM-curves;

\item  $\Phi$ has a homoclinic cylinder $\Gamma$ simple relative to $\hA$ and $\Lambda$, i.e., $F_\Gamma(\hA)\subset \mathrm{int}(\Lambda)$;
 
\item  The cylinder $\Gamma$ is simple relative to $\bA$ and $\hA$, i.e., 
$F_\Gamma(\bA)\subset \mathrm{int}(\hA)$;

\item  The map $\Phi \mid_\Lambda$ 
has the twist property.
\end{enumerate}

Note that the set $\mathcal V_{\rho,\sigma}$ is open with respect to the natural topology  in the Banach space $\mathcal Q^3_{\rho,\sigma}$ which was introduced in \eqref{spaceQdsigmarho} %{eq:banachspace}. 
Theorem 2  from \cite{GelfreichTuraevChaotic} states the following (see also Remark 3 in that paper). Let $\gamma^\pm$ be the boundary curves of $\bar \Lambda$.

\begin{thm}[Theorem 2 in \cite{GelfreichTuraevChaotic}]\label{thm:gelfreichturaevthm} 
In $\cV_{\rho,\sigma}$ there is an open and dense subset  $\tilde {\mathcal V}_{\rho,\sigma}$ such that for each Poincaré map
 $\Phi(G)$ as in \eqref{eq:poincmapfinal} with  $G\in \tilde  \cV_{\rho,\sigma}$
and for every two open neighborhoods $U^-$ of $\gamma^-$ and $U^+$ of $\gamma^+$, the image of $U^-$  by
some forward iteration of $\Phi(G)$ intersects $U^+$.
\end{thm}

For our purposes, we need a slightly more general result that allows for multiple (small) homoclinic cylinders instead of a large one. More precisely, in the very same setting above, given $m\in\mathbb N$, we let $\mathcal V_{\rho,\sigma}(m)$ be the set of real-analytic Hamiltonians for which the associated Poincar\'e maps satisfy items $(1)$ and $(5)$ above and, in addition, satisfy the following:
\begin{enumerate}[label=(\arabic*$'$), start=2]
    \item In $\Lambda$ there exist $m$ triples of invariant sub-cylinders $\bA_i$, $\hA_i$ and $\Lambda_i$, $i=0,\dots,m-1$ such that
$\bA_i\subset \text{int} (\hA_i)\subset \text{int} (\Lambda_i)$, each of the cylinders being bounded by KAM-curves, which we denote by $\bar \gamma_i^{\pm}$, $\hat\gamma_i^\pm$ and $\gamma_i^{\pm}$. Moreover, for all $i=0,\dots,m-1$ we have\footnote{Given two KAM curves $\gamma, \gamma'$ we say that $\gamma<\gamma'$ if the graph of $\gamma$ is strictly below that of $\gamma'$.} 
\[
\gamma_{i+1}^-<\bar\gamma_i^+<\bar\gamma_{i+1}^+ ;
\]
%\color{blue} Jaime, here we must explain the notation %$\gamma_{i+1}^-<\bar\gamma_i^+$. \orange I have added a %footnote.  \color{black}
\item  $\Phi$ has $m$ homoclinic cylinders $\Gamma_i$, $i=0,\dots,m-1$ such that $\Gamma_i$ is simple relative to $\hA_i$ and $\Lambda_i$, i.e., $F_{\Gamma_i}(\hA_i)\subset \mathrm{int}(\Lambda_i)$;
 
\item  For all $i=0,\dots,m-1$ the cylinder $\Gamma_i$ is simple relative to $\bA_i$ and $\hA_i$, i.e., $F_{\Gamma_i}(\bA_i)\subset \mathrm{int}(\hA_i)$.
\end{enumerate}

A straightforward adaptation  
%\color{blue} Jaime, here we must explain the modification %\orange Done, let me know if you think we should add more. %\color{black}  
of the argument used to prove Theorem 2 in \cite{GelfreichTuraevChaotic} provides the following statement.

\begin{thm}[After Theorem 2 in \cite{GelfreichTuraevChaotic}]\label{thm:gelfreichturaevthm2} 
Fix any $m\in\mathbb N$. In $\cV_{\rho,\sigma}(m)$ there is an open and dense subset  $\tilde {\mathcal V}_{\rho,\sigma}(m)$ such that for each Poincaré map
 $\Phi(G)$ as in \eqref{eq:poincmapfinal} with  $G\in \tilde  \cV_{\rho,\sigma}(m)$
and for every two open neighborhoods $U^-$ of $\bar\gamma^-_0$ and $U^+$ of $\bar\gamma^+_{m}$, the image of $U^-$  by
some forward iteration of $\Phi(G)$ intersects $U^+$.
\end{thm}

\begin{proof}
Here we give the indications which allow us to establish Theorem \ref{thm:gelfreichturaevthm2} using the techniques developed in \cite{GelfreichTuraevChaotic}. The proof has  four parts:
    \medskip
    
    \noindent\textit{Step 1.} In view of assumptions $(2')-(4')$, for each $G\in \mathcal V_{\rho,\sigma}(m)$ and  for each $i=0,\dots,m-1$, one can define a family of eight  disjoint\footnote{By disjoint  we mean that their orbits are disjoint.}  homoclinic cylinders $\{\Gamma_i^{(n)}\}_{n=1}^8$  which are simple relative to $\hat\Lambda_i$ and $\Lambda_i$, i.e., $F_{\Gamma_i^{(n)}}(\hat\Lambda_i)\subset \mathrm{int}(\Lambda_i)$ and are also simple relative to $\bar\Lambda_i$ and $\hat\Lambda_i$, i.e., $F_{\Gamma_i^{(n)}}(\bar\Lambda_i)\subset \mathrm{int}(\hat\Lambda_i)$.  

    This can be proved exactly the same way as in Section 3.3 of \cite{GelfreichTuraevChaotic}. Indeed, for each $i=0,\dots,m-1$ the argument of Gelfrich and Turaev yields a countable family of disjoint homoclinic cylinders, but only eight  are needed for their construction. 
\medskip

\noindent\textit{Step 2.} Let $\{F_i^{(n)}\}_{i,n}$ with $i=0,\dots,m-1$ and $n=1,\dots, 8$ be the Iterated Function System (IFS) in $\Lambda$ generated by the scattering maps $F_i^{(n)}=F_{\Gamma_i^{(n)}}:\hat\Lambda_i\to \Lambda_i$.  Then Theorem 4 in \cite{GelfreichTuraevChaotic} shows that  for an open and dense subset  $\widetilde {\mathcal V}_{\rho,\sigma}(m)\subset \mathcal V_{\rho,\sigma}(m)$, for any $i=0,\dots,m-1$ the maps $\{F_i^{(n)}\}_{n=1}^8:\hat \Lambda_i\to \Lambda_i$ of the  corresponding family do not share any common invariant curve contained in  $\bar\Lambda_i$.
\medskip

\noindent\textit{Step 3.}
Observe that, in view of assumption $(2')$, for any $i=0,\dots,m-1$, both $\gamma_{i}^-$ and $\gamma_{i+1}^-$ are KAM invariant curves (for the inner dynamics, i.e., the restriction of the Poincaré map $\Phi(G)$ to $\Lambda$ with $G\in \mathcal V_{\rho,\sigma}(m)$) contained in $\bar\Lambda_i$. Hence, Theorem 3 in \cite{GelfreichTuraevChaotic} says that for $G\in \widetilde {\mathcal V}_{\rho,\sigma}(m)$ and for each $i=0,\dots,m-1$ there exists an $M_i\in\mathbb N$ and  a finite  orbit segment $\{z_k^{(i)}\}_{k=0}^{M_i}$ %(for some $M\in\mathbb N$) 
of the IFS generated by  $\{F_i^{(n)}\}_{n=1}^8:\hat \Lambda_i\to \Lambda_i$ such that $z_0^{(i)}\in \bar\gamma_i^-$ and $z_{M_i}\in \bar\gamma_{i+1}^-$. 
On the other hand, for $i=m$, the same argument implies that that for $G\in \widetilde {\mathcal V}_{\rho,\sigma}(m)$  there exists a finite  orbit segment $\{z_k^{(m)}\}_{k=0}^{M_{m}}$ %(for some $M=M(m)\in\mathbb N$)
of the IFS generated by  $\{F_m^{(n)}\}_{n=1}^8:\hat \Lambda_m\to \Lambda_m$ such that $z_0^{(m)}\in \bar\gamma_m^-$ and $z_{M_m}^{(m)}\in \bar\gamma_{m}^+$.

Thus, we have a collection of $m+2$ KAM invariant curves for the inner dynamics -- formed by $\{\bar\gamma_i^-\}_{i=0}^m$ and $\bar\gamma_{m}^+$ -- such that 
$\bar\gamma_i^-$ is connected to $\bar\gamma_{i+1}^-$ 
by an orbit of the IFS generated by $\{F_i^{(n)}\}_{n=1}^8:\hat \Lambda_i\to \Lambda_i$ and $\bar\gamma_m^-$ is connected to $\bar\gamma_m^+$ by an orbit of the IFS generated by $\{F_m^{(n)}\}_{n=1}^8:\hat \Lambda_m\to \Lambda_m$.
      \medskip

       \noindent\textit{Step 4.} The existence of a true orbit of the Poincaré map $\Phi(G)$ with $G\in\widetilde{\mathcal V}_{\rho,\sigma}(m)$ which connects a neighborhood of $\bar\gamma_0^-$ to a neighborhood of $\bar\gamma_m^+$ is now obtained by means of a standard shadowing argument (see for instance Section 4.2 in \cite{GelfreichTuraevChaotic} or \cite{MR4033892}).\qedhere
\end{proof}

\medskip

We now show how to complete the Proof of Proposition \ref{prop_apply_GT}.
\begin{proof}[Proof of Proposition \ref{prop_apply_GT}]

  Let 
$\mathcal U_{\rho,\sigma}(\delta)$  be the set in Theorem \ref{thm:mainparametricoutline} and, for $g_1\in \mathcal U_{\rho,\sigma}(\delta)$, let   $\mathcal G_{\varepsilon}(g_1)\in \mathcal Q^{3,d}_{\rho,\sigma}$ (see \eqref{eq:banachspace}) be as in Theorem \ref{thm:mainparametricoutline}. Fix $n\in\mathbb N$, and let $m\in\mathbb N$ be as in item $(4)$ of Theorem \ref{thm:mainparametricoutline}. 
 
 Let $\Sigma_n\subset M$ be the four-dimensional section introduced in \eqref{eq:4dsec}.  We just need to show that $\mathcal G_{\varepsilon_n}(g_1)\in \mathcal V_{\rho,\sigma}(m)$. Denote by $\Phi_n$ the Poincar\'e map that $\mathcal G_{\varepsilon_n}(g_1)$ induces on $\Sigma_n$. We have shown in Theorem \ref{thm:mainparametricoutline}
 %Proposition \ref{prop:Poincare_map} 
  that the map $\Phi_n$ is locally well-defined and admits an invariant cylinder $\Lambda_n$.  This invariant cylinder is symmetrically normally hyperbolic since the differential of the restriction $\Phi_n|_{\Lambda_n}$ is given by a shear parabolic matrix. In particular, the rates of expansion/contraction displayed by the inner dynamics depend, at most, polynomially on the number of iterates.  \black Hence, since $\Lambda_n$ is compact, the expansion/contraction along $N_u/N_s$ is bounded from below by some $\lambda_n>1$, while, by a suitable rescaling of the norm on $N_c$, we can achieve that the rate $\alpha_n$ in \eqref{eq:hyperbolicityestimates} is arbitrarily close to $1$ and, consequently, verify the inequality $\alpha_n^2\lambda_n<1$.  Moreover, since by construction $\mathcal G_{\varepsilon_n}(g_1)$ does not depend on $\theta_3$, the map $\Phi_n|_{\Lambda_n}$ leaves invariant each of the circles
\begin{equation}\label{eq:invariantcircles}
\gamma_{n}(I_3^*)=\Lambda_n\cap\{I_3=I_3^*\},\qquad\qquad I_3^*\in[-\rho_0,\rho_0],
\end{equation}
and satisfies the  twist condition (see Theorem \ref{thm:mainparametricoutline}).
%Proposition \ref{prop:Poincare_map}).
In particular,  $\Lambda_n$ contains a full measure set of invariant KAM-curves.

Consider now the covering $[-\rho_0,\rho_0]\subset \bigcup_{i=0}^{m-1} [I_i^-,I_i^+]$ in item $(4)$ of Theorem \ref{thm:mainparametricoutline} and choose $m$ triples of invariant sub-cylinders as in $(2')$ above such that $\Lambda_i\subset \mathbb T\times [I_i^-,I_i^+]$. Then, for each $i=0,\dots, m-1$, the homoclinic cylinder $\Gamma_{n,i}\subset W^s(\Lambda_n)\pitchfork W^u(\Lambda_n)$ is simple relative to the pair  $\hat \Lambda_i,\Lambda_i$ and also to the pair $\bar\Lambda_i,\hat\Lambda_i$. Indeed, each of the curves in \eqref{eq:invariantcircles} are also invariant under the scattering map dynamics $F_{\Gamma_i}:\mathrm{int}(\Lambda_i)\to\mathrm{int}(\Lambda_i)$. We thus conclude that $\mathcal G_{\varepsilon_n} (g_1)\in\mathcal V_{\rho,\sigma}(m)$. The proof of Proposition \ref{prop_apply_GT} is completed.
\end{proof}
\black

\section{Generic transversality of invariant manifolds: proof of Proposition \ref{prop:splitting}}\label{sec:mainsptlittingsection}

We now complete the proof of Proposition \ref{prop:splitting}. 
Throughout this section $\delta>0$ will be fixed so we simply write $\mathcal B_{\rho,\sigma}=B_{\rho,\sigma}^2(\delta)\subset \mathcal Q^2_\sigma$ to refer to the  ball of radius $\delta>0$ centered around the origin.  
Recall that in Lemma \ref{lem:auxiliaryperiodicorbits2} we have shown that there exists $\varepsilon_n>0$ such that, for any $0<\varepsilon\leq \varepsilon_n$ and any $h\in\mathcal B_{\rho,\sigma}$, the Hamiltonian $\mathcal H^{(n)}_\varepsilon(h)(\cdot;I_3)$ in \eqref{eq:epssquareHamiltonian} displays a hyperbolic periodic orbit $\gamma_n(h;I_3)$ admitting a parametrization
\begin{equation}\label{eq:parametrizationperiodic}
\gamma_n(h;I_3)=\{(\varphi^{(n)}(h)(\tt;I_3),t,J^{(n)}(h)(t;I_3),E^{(n)}(h)(\tt;I_3))\colon \tt\in\mathbb T\},
\end{equation}
which is of the form \eqref{eq:parametrizationdistorted}-\eqref{eq:parametrizationdistortedestimates} and depends smoothly (in fact, real-analytically) on $I_3$ for $I_3\in[-\rho_0,\rho_0]$.

\subsection{Lagrangian parametrizations and the splitting potential}

We now obtain suitable parametrizations of the stable and unstable manifolds of the hyperbolic periodic orbit $\gamma_n(h;I_3)$ above. 
More concretely we will exploit the fact that, since the stable and unstable invariant manifolds $W^s(\gamma_n(h;I_3))$ and $W^u(\gamma_n(h;I_3))$ of the Hamiltonian $\mathcal{H}_\varepsilon^{(n)} $
are Lagrangian submanifolds, they can be (locally) parametrized in terms of  a generating function. 

\begin{rem}
    Both the stable and unstable manifolds of this periodic orbit have two branches. 
    For our purposes, it will be enough to consider the left branch of the stable manifold and the right branch of the unstable manifold, which we denote by $W^{s}(\gamma_n(h;I_3))$ and $W^u(\gamma_n(h;I_3))$ (see Figure \ref{fig:fig2}).

\end{rem}

\begin{comment}
 Recall that, for the averaged system  $H_{\mathrm{\mathrm{av}},\varepsilon}^{(n)}$    in $ \eqref{eq:integrablen} $, the periodic orbit  $\gamma^{\mathrm{av}}_{n}(I_3)$ has a ``pendulum like''   homoclinic connection given by \eqref{eq:homoclinicpendulum} that can also be written in terms of a generating function $S^h (\varphi)\sim  \pm
\sqrt{\varepsilon}
\sqrt{\frac{2b_n}{C_n} } \sin \frac{\varphi}{2}$ as:
\[
E = \tilde a_E^{(n)} +\partial_\tau S^h (\varphi), \quad J= \tilde a_J^{(n)}+\partial_\varphi S^h (\varphi), 
%\quad C_n (J-\tilde a_1^{(n)})^2+ \varepsilon b_n (\cos \varphi -1) \simeq 0 
\quad \varphi  \in [0, 2\pi].
\]
Therefore, it is natural to expect that the generating functions of the (un)stable manifolds of the perturbed periodic orbit can be continued for $\varphi$ near $\pi$ and will be close to these unperturbed homoclinic connections.
Although each of these manifolds has two branches, it will suffice for our purposes to just consider the right branch of $W^u(\gamma_n(h))$ and the left branch of $W^s(\gamma_n(h))$.
This is the content of the next proposition, whose proof is standard using classical perturbation theory applied to system $\mathcal{H}_\varepsilon^{(n)}$ as a perturbation of the averaged system
$H_{\mathrm{\mathrm{av}},\varepsilon}^{(n)}$.
\end{comment}
%
\begin{prop}\label{prop:LagrangianW^su}
For any $h\in\mathcal B_{\rho,\sigma}$ the stable and the unstable invariant manifolds of the hyperbolic periodic orbit $\gamma_n(h;I_3)$ associated to the Hamiltonian \eqref{eq:epssquareHamiltonian} admit a graph parametrization of the form

\begin{equation}\label{eq:Laggraphparametrizations}
    \begin{split}
	        W^u(\gamma_n(h;I_3))=&\{(\varphi,\tt,\beta_\varphi+\partial_\varphi S^u(h)(\varphi,\tt;I_3),\beta_\tt+\partial_\tt S^u(h)(\varphi,\tt;I_3))\colon \varphi^{(n)}(\tt)<\varphi<5\pi/4,\ \tt\in \mathbb T \},\\
         W^s(\gamma_n(h;I_3))=&\{(\varphi,\tt,\beta_\varphi+\partial_\varphi S^s(h)(\varphi,\tt;I_3),\beta_\tt+\partial_\tt S^s(h)(\varphi,\tt;I_3))\colon  \pi/4<\varphi<2\pi+\varphi^{(n)}(\tt),\ \tt\in \mathbb T\}\\
    \end{split}
\end{equation}
in terms of some $\beta=(\beta_\varphi,\beta_\tt)\in\mathbb R^2$ and real-analytic functions $S^u(h)$ and $S^s(h)$ that depend on $I_3$ in a real-analytic fashion for $I_3\in[-\rho_0,\rho_0]$.
\end{prop}

\begin{rem}
    For simplicity of notation, we omit the index $n$ in the notation of $\beta$, $S^u$, $S^s$ and $\Delta$ in \eqref{eq:splittingpotential}. 
\end{rem}

The proof of Proposition \ref{prop:LagrangianW^su} is rather standard and is deferred to Appendix \ref{sec:appendixhyperbolic}. Proposition \ref{prop:LagrangianW^su} reduces the problem of the existence of intersections between $W^{s}(\gamma_n(h;I_3))$ and $W^{u}(\gamma_n(h;I_3))$ to that of the existence of critical points of the \textit{splitting potential}, defined by
\begin{equation}\label{eq:splittingpotential}
\tt\mapsto \Delta(h)(\tt;I_3)=S^u(h)(\pi,\tt;I_3)-S^s(h)(\pi,\tt;I_3).
\end{equation}
Notice that since $\tt\mapsto\Delta(\tt;I_3)$ is periodic and differentiable (in fact, even real-analytic), for any $I_3\in[-\rho_0,\rho_0]$ there exist at least two critical points of this function corresponding to two global extreme values.  However, we do not know a priori if these are non-degenerate (which would imply that the corresponding intersection is transverse). Nevertheless, following an idea of Delshams and Zhang (see \cite{MR4913967}), we will be able to show that for a generic $h\in\mathcal B_{\rho,\sigma}$ the corresponding splitting potential does not admit degenerate global extreme values. 
More precisely, we will show the following.
\begin{thm}\label{thm:genericityofnondegmax}
     There exists $\varepsilon_n>0$ such that for any $0<\varepsilon\leq \varepsilon_n$  there exists a residual subset  $\mathcal U^{(n)}_{\rho,\sigma}\subset\mathcal B_{\rho,\sigma}$ such that the splitting potential \eqref{eq:splittingpotential} associated to the Hamiltonian $\mathcal H_{\varepsilon}^{(n)} (h)$ in \eqref{eq:epssquareHamiltonian} with $h\in \mathcal U_{\rho,\sigma}^{(n)}$ does not admit degenerate global extrema for any $I_3\in[-\rho_0,\rho_0]$.
\end{thm}

Now,  Proposition \ref{prop:splitting} follows easily.
\begin{proof}[Proof of Proposition \ref{prop:splitting}]
By compactness of the interval $[-\rho_0, \rho_0]$, Theorem \ref{thm:genericityofnondegmax} implies that there exists $m\in\mathbb N$ and a covering 
\[
[-\rho_0,\rho_0]\subset \bigcup_{0\leq i\leq m} [\cI_i^-,\cI_i^+]
\]
satisfying $\cI_{i+1}^-<\cI_{i}^+$ for all $i=0,\dots,m-1$  such that  for any $i=0,\dots,m-1$ there exists a transverse homoclinic orbit $z_{n,i}(h)(I_3)$ to $\gamma_n(h)(I_3)$ which depends real-analytically on $I_3$ for $I_3\in[\cI_i^-,\cI_i^+]$.  Hence,  the proof is complete.
\end{proof}

\subsection{Parametric transversality: proof of Theorem \ref{thm:genericityofnondegmax}}

The main step in the proof of Theorem \ref{thm:genericityofnondegmax} is to establish Proposition \ref{prop:surjective} below. We denote 
\[
\mathbb A_{\rho_0}=\mathbb T\times[-\rho_0,\rho_0],
\]
and let $\Delta',\Delta''$ and $\Delta'''$ stand for the first, second and third derivatives of the function $\tt\mapsto \Delta(h)(\tt;I_3)$.

\begin{prop}\label{prop:surjective}
    At every $(\tt,I_3)\in\mathbb A_{\rho_0}$ the map
    \begin{equation}\label{eq:surjmap}
h\in \mathcal B_{\rho,\sigma} \mapsto F(\tt,I_3,h):=\begin{pmatrix}
    \Delta'(h)(\tt;I_3)\\
    \Delta''(h)(\tt;I_3)\\
    \Delta'''(h)(\tt;I_3)\\
\end{pmatrix}\in\mathbb R^3
\end{equation}
is Fréchet differentiable, and its differential is surjective. 

\end{prop}
The proof of this result is given below. Now we show how to complete the proof of Theorem \ref{thm:genericityofnondegmax} assuming Proposition \ref{prop:surjective}. We do so using the following version of the parametric transversality theorem, which can be easily read from Theorem 19.1 in \cite{MR240836}.
\begin{thm}[Parametric transversality theorem]\label{thm:parametrictransvthm}
    Given $r>0$, let $X,Y$ and $\mathcal Z$ be $\C^r$ Banach manifolds, and let $\mathcal F:X\times \mathcal Z \to Y$ be a $\C^r$  map. Let $y\in Y$, and suppose that 
    \begin{enumerate}
        \item $\mathrm{dim}(X)=p<\infty$,
        \item $\mathrm{dim}(Y)=q<\infty$,
        \item $r>\max\{0,p-q\}$,
        \item $y$ is a {\it regular value} of the map $\mathcal F$. 
        %\blue (better $\mathcal F(X,S)$ intersects transversally the set $\{y\}$. \black
    \end{enumerate}
    Then 
    \[
    \mathcal Z_y=\{z\in \mathcal Z\colon y\ \text{ is a regular value of the map }x\mapsto \mathcal F(x,z)\} 
    \]
    is a residual subset of $\mathcal Z$.
\end{thm}

We recall that, given smooth Banach manifolds $M$ and $N$ and a smooth enough function $G: M\to N$, we say that $n \in N$ is a {\it regular value} of  $G$ if and only if $G(M)$ intersects transversally the set 
$\{n\}$ (denoted $G(M) \pitchfork \{n\}$).  This means  that one of the following holds:
\begin{itemize}
    \item either $n\notin G(M)$
    \item or there exists $m\in M$ such that $G(m)=n$ and the differential at $m$ has full rank. 
\end{itemize}

A straightforward corollary of Theorem \ref{thm:parametrictransvthm} is the following.
\begin{cor}\label{cor:corollarytransv}
    Under the hypothesis of Theorem \ref{thm:parametrictransvthm}, if $p-q<0$, then for any $z\in\mathcal Z_y$ we have $\mathcal F(x,z)\neq y$ for all $x\in X$.
\end{cor}
\begin{proof}
    If $p<q$ then, at any $z \in \mathcal Z$, the differential of the map $x\mapsto \mathcal F(x,z)$ fails to be surjective. 
    Hence, we deduce  that for all $z\in\mathcal Z_y$ we must have $\mathcal F(x,z)\neq y$ for all $x\in X$.
\end{proof}

\begin{proof}[Proof of Theorem \ref{thm:genericityofnondegmax}]
In view of Proposition \ref{prop:surjective}, $0\in \mathbb R^3$ is a regular value of the map $\mathcal F:\mathbb A_{\rho_0}\times\mathcal B_\sigma\to \mathbb R^3$ given by $
\mathcal F(\tt,I_3,h)\mapsto F(\tt,I_3,h)
$ with $F$ as in \eqref{eq:surjmap}. Hence, since $\mathrm{dim}(\mathbb A_{\rho_0})=2<3=\mathrm{dim}(\mathbb R^3)$, a direct application of Theorem \ref{thm:parametrictransvthm} and Corollary \ref{cor:corollarytransv} shows that there exists a residual subset  $\mathcal U_{\rho,\sigma}\subset \mathcal B_{\rho,\sigma}$ such that for any $h\in \mathcal U_{\rho,\sigma}$ we have $F(\tt,I_3,h)\neq 0$ for all $(\tt,I_3)\in \mathbb A_\rho$.
\medskip

The desired conclusion is obtained after noticing that for $(\tt,I_3,h)\in \mathbb A_{\rho_0}\times{\mathcal U}_{\sigma}$, if $\Delta'(h)(\tt;I_3)=0=\Delta''(h)(\tt;I_3)$, then we must have $\Delta'''(h)(\tt;I_3)\neq 0$ so that this critical point cannot correspond to a global extremizer.
\end{proof}

\subsection{Proof of Proposition \ref{prop:surjective}}

The argument in  \cite{MR4913967} is based on approximating compactly supported perturbations by real-analytic ones. Namely, we will first establish a version of Proposition \ref{prop:surjective} for compactly supported functions, and then conclude by appealing to the density of real-analytic functions within the space of smooth compactly supported ones.

To do so, we start by recalling some well-known facts about the smooth dependence of invariant manifolds with respect to  perturbations. Let $B_{\C^n}\subset \C^n(\mathbb T^2,\mathbb R)$ stand for the unit ball centered at the origin in $\C^n(\mathbb T^2,\mathbb R)$. 

\begin{prop}\label{prop:frechet}
    There exists $\varepsilon_n>0$ such that for any $0<\varepsilon \leq \varepsilon_n$, any $I_3\in [-\rho_0,\rho_0]$ and any $h\in B_{\C^3}$ the Hamiltonian $\mathcal H^{(n)}_\varepsilon(h)$ in \eqref{eq:epssquareHamiltonian} has a hyperbolic periodic orbit $\gamma_n(h;I_3)$ whose stable and unstable manifolds admit a graph parametrization of the form  \eqref{eq:Laggraphparametrizations} 
in terms of some $\beta(h)\in\mathbb R^2$ and $\C^3$ functions $S^u(h)$ and $S^s(h)$ which depend on $I_3$ in a real-analytic fashion. Moreover, the maps
\begin{equation}\label{eq:frechetmaps}
\begin{split}
\Omega^\star:B_{\C^3}&\to \C^3(\mathbb T^2,\mathbb R)\\
h&\mapsto S^\star(h), \qquad \star=s, \ u,
    \end{split}
\end{equation}
are Fréchet differentiable.
\end{prop}

\begin{proof}
  %The existence of a hyperbolic periodic orbit $\gamma_n(h)$ follows by the very same argument in the proof of Theorem \ref{thm:mainparametricoutline} since we already assume that the $C^3$ norm of the perturbation is uniformly bounded by $\varepsilon^2$. By the same argument, it admits a parametrization of the form \eqref{eq:parametrizationperiodic}. 

To avoid a cumbersome notation, during the proof we omit the dependence on $I_3$.
We consider Hamiltonians of the form 
\[
\mathcal H^{(n)}_\varepsilon (h)=H_\varepsilon^{(n)}+\varepsilon^2 h
\]
with $H_\varepsilon^{(n)}$ as in \eqref{eq:truncatedHamiltonianper} and $h\in B_{ \C^3}$. 
In Proposition \ref{prop:periodicorbits} we have shown that the real-analytic Hamiltonian $\mathcal H_\varepsilon^{(n)}$ admits a hyperbolic periodic orbit $\gamma_n$, and in Proposition \ref{prop:LagrangianW^su} we have shown that its stable/unstable manifolds can be parametrized as real-analytic Lagrangian graphs over the $(\varphi,\tau)$ coordinates (see the parametrization \eqref{eq:Laggraphparametrizations}). 
%By regular dependence of unstable/stable manifolds on parameters, since the vector field associated to the Hamiltonian $\mathcal H_\varepsilon^{(n)}(h)$ is given by a $O_{C^2}(\varepsilon^2)$ perturbation of that of $H_\varepsilon^{(n)}$, 
Therefore, we deduce the existence of $\C^2$-parametrizations
  \begin{equation}\label{eq:paramproof}
    \begin{split}
  W^u(\gamma_n(h))=&\{(\varphi,\tt,J^u(h)(\varphi,\tt),E^u(h)(\varphi,\tt))\colon \varphi^{(n)}(h)(\tt)<\varphi<5\pi/4,\ \tt\in \mathbb T\},\\
         W^s(\gamma_n(h))=&\{(\varphi,\tt,J^s(h)(\varphi,\tt),E^s(h)(\varphi,\tt))\colon  \pi/4<\varphi<2\pi+\varphi^{(n)}(h)(\tt),\ \tt\in \mathbb T\}
         \end{split}
  \end{equation}
for the unstable and stable manifold of the hyperbolic periodic orbit $\gamma_n(h)$ of the Hamiltonian $\mathcal H_\varepsilon^{(n)}(h)$. Moreover, the maps ($\star=u,s$) \begin{align*}
      \widetilde{\Omega}^\star :B_{\C^3}&\to \C^2(\mathbb T^2,\mathbb R^2)\\
      h&\mapsto (J^\star(h)(\varphi,\tt),E^\star(h)(\varphi,\tt))
  \end{align*}
  are Fréchet differentiable (see Appendix 1 in \cite{MR1237641}; in fact, these maps are $\C^\infty$ --- as maps between Banach spaces --- since the dependence of $\mathcal H_\varepsilon(h)$ on $h$ is $\C^\infty$). The conclusion follows since, as $\gamma_n(h)$ are Lagrangian, the functions $(J^\star(h)(\varphi,\tt),E^\star(h)(\varphi,\tt))$ in the parametrizations \eqref{eq:paramproof} correspond to the differential of the functions of the form $\mathcal S^\star(h)(\varphi,\tt)=\beta_\varphi\varphi+\beta_\tt \tt+S^\star(h)(\varphi,\tt)$ with $S^\star(h):\mathbb T^2\to \mathbb R$ and some $(\beta_\varphi,\beta_\tt)\in\mathbb R^2$.
\end{proof}

\begin{figure}
    \centering
    \includegraphics[scale=0.60]{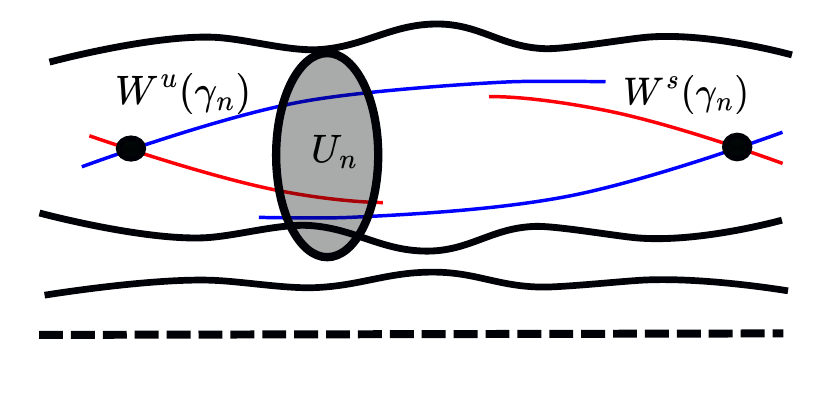}
    \caption{The stable and unstable manifolds of the hyperbolic periodic orbit $\gamma_n$ and the domain $U_n$ where we localize the perturbation. In dashed line we represent the invariant torus at $\{I=0\}$.}
    \label{fig:fig2}
\end{figure}
\medskip

Proposition \ref{prop:frechet} is the key to deduce Proposition \ref{prop:surjective} from the corresponding statement about smooth compactly supported perturbations, which we now describe. 
We will consider perturbations which are localized in (see Figure \ref{fig:fig2})  
\begin{equation}\label{eq:Un}
U_n=\phi_n([\pi/2,\pi]\times\mathbb T\times R_n)\subset\mathbb T^2\times \mathbb R^2,
\end{equation}
where $\phi_n$ is the coordinate change introduced in \eqref{eq:symplecticlatticechange} and $R_n\subset \mathbb R^2$ is a sufficiently small open subset (depending possibly on $n$ but not on $\varepsilon$) which contains the point $(\tilde a_J^{(n)}, \tilde a_E^{(n)})\subset\mathbb R^2$ defined in \eqref{eq:jn*}. 

Notice that in the variables $(\varphi,\tt)\in\mathbb T^2$ we only consider perturbations which are compactly supported on $\{\varphi\in[\pi/2,\pi]\}$. In particular, for any $f\in \C^\infty_c(U_n,\mathbb R)$ the Hamiltonian $\mathcal H_\varepsilon^{(n)}(h+f)$ still admits $\gamma_n(h)$ as a hyperbolic periodic orbit. We will now show that, using these type of perturbations, it is possible to displace the (global) invariant manifolds of the corresponding hyperbolic orbit by choosing a suitable $f\in \C^\infty_c(U_n,\mathbb R)$.

To be more precise, given $f\in \C^\infty_c(U_n,\mathbb R)$ and  $\mu\in\mathbb R$ let $g_\mu=h+\mu f$.  Denote  
\[
T^{\star}_f(h)=\frac{\mathrm d}{\mathrm{d}\mu} S^\star(g_\mu)|_{\mu=0}, \quad \star=s, u
\]
the Gateaux derivative of $S^\star$ at $h$ in the direction $f$. Then, since $S^u(g_\mu)$ is a solution of the partial differential equation
\begin{align*}
\begin{cases}
 \mathcal H_\varepsilon^{(n)}(g_\mu)(\varphi,\tt,\beta_\varphi+\partial_\varphi S^u(g_\mu)(\varphi,\tt),\beta_\tt+\partial_\tt S^u(g_\mu)(\varphi,\tt))=e_n\\
\lim_{\varphi\to \varphi(h)(\tt)} \beta_\varphi+\partial_\varphi S^u(g_\mu)(\varphi,\tt)=J^{(n)}(h)(\tt)\\
\lim_{\varphi\to \varphi(h)(\tt)} \beta_\tt+\partial_\tt S^u(g_\mu)(\varphi,\tt)=E^{(n)}(h)(\tt)
\end{cases} \qquad (\varphi,\tt)\in D^u
\end{align*}
with $\varphi^{(n)}(h)(\tt)$ as in \eqref{eq:paramproof},
\begin{equation}\label{eq:unstabledomain}
D^u=\{ \varphi^{(n)}(h)(\tt)<\varphi<5\pi/4,\ \tt\in \mathbb T\}
\end{equation} 
and 
$(J^{(n)}(h),E^{(n)}(h))$ as in \eqref{eq:parametrizationperiodic}, it is a straightforward computation to see that $T_f^u(h)$ satisfies the linear partial differential equation 
\begin{equation}\label{eq:pdeinfdispl}
\begin{cases}
{\mathcal L}^{u} [T^{u}_f(h)]=-f\\
\lim_{\varphi\to\varphi(h)(\tt)}\partial_\varphi T^{u}_f(h)(\varphi,\tt)=0
\\ \lim_{\varphi\to\varphi(h)(\tt)}\partial_\tt T^{u}_f(h)(\varphi,\tt)=0
\end{cases}
\end{equation}
with 
\begin{equation}\label{eq:diffoperatordistorted}
{\mathcal L}^{u}[\cdot]=\partial_J \mathcal H_\varepsilon^{(n)}(h)\partial_\varphi [\cdot]+\partial_E\mathcal H^{(n)}_\varepsilon(h)\partial_\tt[\cdot].
\end{equation}
On the other hand, it is clear from \eqref{eq:pdeinfdispl} that for any $f\in \C^\infty_c(U_n,\mathbb T)$ and  any $j\in\mathbb N$ we have 
\[
\partial_{\tt^j}^jT^s_f(h)(\varphi,\tt)=0
\]
at any $(\varphi,\tt)\in [\pi,2\pi ]\times \mathbb T$. 
Hence, it is enough to consider perturbations $f\in \C^\infty_c(U_n,\mathbb R)$ and observe the displacement $T^u_f(h)$ on the generating function $S^u(h)$. 

\begin{lem}\label{lem:surjectivetechnical}
    Let $U_n$ be of the form \eqref{eq:Un} where $R_n$ is a sufficiently small open neighborhood. 
    For any $\tt\in\mathbb T$ and $h\in\mathcal B_{\rho,\sigma}$,  there exists $f_1,f_2,f_3\in \C^\infty_c(U_n,\mathbb R)$ such that
    the vectors 
    \[
a_i=(\partial_\tt  T^{u}_{f_i}(\pi,\tt),\partial^2_{\tt^2}  T^{u}_{f_i}(\pi,\tt),\partial^3_{\tt^3} T^{u}_{f_i}(\pi,\tt)),\qquad\qquad i=1,2,3,
    \]
    form a basis of $\mathbb R^3$ and satisfy
    \[
    0=(\partial_\tt  T^{s}_{f_i}(\pi,\tt),\partial^2_{\tt^2}  T^{s}_{f_i}(\pi,\tt),\partial^3_{\tt^3} T^{s}_{f_i}(\pi,\tt)),\qquad\qquad i=1,2,3,
    \]
\end{lem}

In order to give a proof of this result, we first straighten the partial differential operator $\mathcal L^u$. 

\begin{lem}\label{lem:straightening}
Let $D^u$ be as in \eqref{eq:unstabledomain} and let $\mathcal L^u$ be the differential operator \eqref{eq:diffoperatordistorted}.  There exist $\C^1$ functions $\varphi_{\mathrm{av}}:\mathbb R\to (0,2\pi)$ and $\psi_{v},\psi_\xi:(-\infty,1)\times\mathbb T\to \mathbb R$ satisfying 
\[
\varphi_{\mathrm{av}}(0)=\pi, \qquad\qquad |\psi_v|_{\C^1},|\psi_\xi|_{\C^1}\lesssim \varepsilon ,
\]
such that the change of variables  
\begin{equation}\label{eq:straightening}
\begin{split}
\psi:(-\infty,1)\times\mathbb T &\to   D^u\\
(v,\xi)&\mapsto (\varphi_{\mathrm{av}}(v)+\psi_v(v,\xi),\xi+\psi_\xi(v,\xi))
\end{split}
\end{equation}
satisfies the following.  
For any differentiable function $T$ we have  
\[
({\mathcal L}^{u} [T])\circ\psi =  \mathcal L [T\circ\psi],
\]
 where  $\mathcal L$ is a constant coefficients partial differential operator of the form
\[
\mathcal L[\cdot]=\partial_v[\cdot]+\tilde\omega_n\partial_\xi [\cdot]
\]
for some $\tilde\omega_n\neq 0$ (independent of $\varepsilon$).
\end{lem}
The proof of this result boils down to conjugate the flow on $W^u(\gamma_n(h))$ to a  linear translation on $\mathbb R\times\mathbb T$ and can be obtained by standard methods.  For the sake of self completeness we provide a proof in Appendix \ref{sec:straightening}. 

\begin{rem}
    The function $\varphi_{\mathrm{av}}$ comes from the time parametrization of the homoclinic orbit of the averaged Hamiltonian $H_{\mathrm{av},\varepsilon}^{(n)}$ (see Appendix \ref{sec:integrable}).
    
    The constant $\tilde\omega_n$  can be related to the $\tt$-component of the vector field, i.e. $\partial_E \mathcal H^{(n)}_\varepsilon (h)$ when evaluated at the periodic orbit $\gamma_n(h)$. This connection is established in Appendix \ref{sec:straighteningproof}, where we also show that $\tilde\omega_n$ does not vanish. 
\end{rem}

\medskip

In the new coordinate system $(v,\xi)$ it is easier to understand the infinitesimal displacement $T_f^u(h)$ on the unstable generating function $S^u(h)$ induced by $f$. In particular, it follows from Lemma \ref{lem:straightening} that 
\[
\mathcal T^{u}_f(h)=T^{u}_f(h)\circ\psi
\]
 is given by 
 \[
\begin{cases}
\mathcal L \mathcal T^{u}_f(h)=-\tilde  f\\
\lim_{v\to -\infty} \partial_v \mathcal T^{u}_f(h)(v,\xi)=0\\
\lim_{v\to -\infty} \partial_\xi \mathcal T^{u}_f(h)(v,\xi)=0
\end{cases}\qquad\qquad\tilde f=  f\circ\psi.
\]
Hence,
\[
\mathcal T^{u}_f(h)(v,\xi)=-\int_{-\infty}^0 \tilde f(v+s,\xi+\omega_ns)\mathrm ds .
\]

\begin{proof}[Proof of Lemma \ref{lem:surjectivetechnical}]
We claim that   at any $\xi\in\mathbb T$, given any $i=1,2,3$ and  $\kappa>0$   we are able to find a function $\tilde f_i(v,\xi)$ such that 
\begin{equation}\label{eq:densitycompact}
|\partial^i_{\xi^i} \mathcal T^u_{f_i}(0,\xi)-1|\leq \kappa, \qquad\qquad |\partial^j_{\xi^j} \mathcal T^u_{f_i}(0,\xi) |\leq \kappa\qquad\qquad j\neq i.
\end{equation}
Then, since the change of variables $\psi$ in  \eqref{eq:straightening} is $O_{\C^1}(\varepsilon)$ close to diagonal, there exists a close to diagonal $9\times 9$ matrix $A$ such that (the $O(\varepsilon)$ error comes from the fact that, a priori, we only know that $\psi(0,\xi)=(\pi+O(\varepsilon),\xi+O(\varepsilon))$)
\[
\left(\partial_{\varphi,\tau} T_f^u(h)(\pi,\xi),\ \partial_{\varphi,\tau}^2 T_f^u(h)(\pi,\xi),\ \partial_{\varphi,\tau}^3 T_f^u(h)(\pi,\xi)\right)=A\begin{pmatrix}
    \partial_{v,\xi} \mathcal T_f^u(h)(0,\xi)\\ \partial_{v,\xi}^2 \mathcal T_f^u(h)(0,\xi)\\ \partial_{v,\xi}^3 \mathcal T_f^u(h)(0,\xi)
\end{pmatrix}+O(\varepsilon).
\]
Hence, if we let $c_i\in\mathbb R^9$ be given by 
\[
c_i(\xi)= (\partial_{v,\xi} \mathcal T_{f_i}^u(h)(0,\xi),\  \partial^2_{v,\xi}\mathcal T_{f_i}^u(h)(0,\xi),\  \partial^3_{v,\xi}\mathcal T_{f_i}^u(h)(0,\xi)),
\]
Observe that, considering $A c_i$ we obtain functions $T^u_{f_i} (h) (\varphi,\tau)$ satisfying:
\[
\partial^j_{\tau^j} \ T^u_{f_i}(h) \simeq \delta_{ij}.
\]
Indeed,  let $V\subset \mathbb R^9$ be the three-dimensional subspace spanned by 
\[
b_1=(0,\underbrace{1}_{\partial_\tt},\dots),\qquad\qquad b_2=(0,\dots,\underbrace{1}_{\partial^2_{\tt^2}},\dots,0),\qquad\qquad b_3=(0,\dots,\underbrace{1}_{\partial^3_{\tt^3}}),
\]
it follows from the fact that $A$ is close to diagonal that 
\[
\mathrm{span}\{\pi_V A c_i\colon i=1,2,3\}=\mathbb R^3.
\]
Therefore, Lemma \ref{lem:surjectivetechnical} follows by taking $\varepsilon$ small enough.
\medskip 

We now proceed to verify the claim by finding $\tilde f_i$ such that \eqref{eq:densitycompact} holds. Given $i=1,2,3$ and   $\xi_0\in\mathbb T$ we let $\eta_i(\xi)$ be any function satisfying 
\[
\partial^{j}_{\xi^{j}}\eta_i(\xi_0) =\delta_{ij}.
\]
Let $\chi\in \C^\infty_c((0,1),\mathbb R)$ satisfy $\int_\mathbb R\chi=1 $ and, for $\rho>0$ (which will later be chosen in terms of $\kappa$) define the function 
\[
\chi_\rho(v)=\frac{1}{\rho}\chi\left(\frac{v+\rho}{\rho}\right).
\]
Trivially, the support of $\chi_\rho$ is contained in $(-\rho,0)$. Then, we chose $\tilde f_i=\eta_i \chi_\rho$ so 
\begin{align*}
\partial_{\xi^j}^j \mathcal T^u_{f_i}(0,\xi_0)=\int_{-\infty}^0 \partial_{\xi^j}^j\tilde f_i(s,\xi_0+\omega_n s)\mathrm{d}s=&\int_{-\infty}^0 \chi_\rho(s)\partial_{\xi^j}^j\eta_i(\xi_0+\omega_n s)\mathrm{d}s\\
=&\int_{\mathbb R}\chi_\rho(s)\partial_{\xi^j}^j\eta_i(\xi_0+\omega_n s)\mathrm{d}s\\
=&\int_{\mathbb R}\chi(\tau)\partial_{\xi^j}^j\eta_i(\xi_0+\omega_n  \rho(\tau-1))\mathrm{d}\tau
\end{align*}
so, after writing 
\[
\partial_{\xi^j}^j\eta_i(\xi_0+\omega_n  \rho(\tau-1))=\partial_{\xi^j}^j\eta_i(\xi_0)+O( \rho\lVert \eta_i\rVert_{C^{j+1}}),
\]
we deduce that 
\[
\partial_{\xi^j}^j \mathcal T^u_{f_i}(0,\xi_0)=\delta_{ij}+O( \rho\lVert \eta_i\rVert_{\C^{j+1}}).
\]
In particular, by choosing $\rho$ small enough 
\[
|\partial_{\xi^j}^j \mathcal T^u_{f_i}(0,\xi_0)-\delta_{ij}|\leq \kappa.\qedhere
\]
\end{proof}

\begin{proof}[Proof of Proposition \ref{prop:surjective}] 
The smooth compactly supported perturbations constructed above yield a surjective differential. 

They can be approximated arbitrarily well in the required $\C^r$-topology on  $\overline{U_n}$ by analytic perturbations lying in the Banach ball  $\mathcal B_{\rho,\sigma}$. 

The differential is determined by the restriction of the perturbation to $U_n$ and depends continuously on it in the $\C^r$-topology. Since surjectivity of a linear map onto a finite-dimensional space is an open property, the corresponding analytic perturbations still yield a surjective differential. This completes the proof.\end{proof}

\section{Proof of Theorem \ref{thm:approxbydiffusionellipticpoints} %and \ref{thm:main2}
}\label{sec:proofElliptic}

In this section we show how to adapt the construction in Section \ref{sec:outlineparametric}, where we provided a proof of Theorem \ref{thm:approxbydiffusion},  to obtain a proof of Theorem \ref{thm:approxbydiffusionellipticpoints}. There are two steps:
\begin{enumerate}
    \item Creation of a sequence of large normally hyperbolic invariant cylinders $\{\Lambda_n\}$ displaying a rich net of homoclinic channels;
    \item Existence of diffusing orbits along each $\Lambda_n$.
\end{enumerate}

\begin{rem}
    For the rest of this section we omit the dependence on the parameter $\eta$ as it plays no role in our discussion.
\end{rem}

\subsection{The sign condition and creation of large normally hyperbolic invariant cylinders}

This is the only step where a small modification of the argument in Section \ref{sec:outlineparametric} is needed.

 Below, given $G_0\in \mathcal M^{3,d}_\rho$ we denote by $\mathtt G_0$ the function for which \[
G_0(\tilde q,\tilde p)=\mathtt G_0(\tilde I),\qquad\qquad
 \tilde I_i=(\tilde q_i^2+ \tilde p_i^2)/2, \quad i=1,2,3.
 \]

\begin{rem}\label{rem:translation_H1-H2} Note that if, for some fixed $\omega$, $G_0\in \mathcal M^{3,d}_\rho(\omega)$, then $\mathtt G_0\in \mathcal N^{3,d}_\rho(\omega)$. 
We will say that the function $G_0\in \mathcal M^{3,d}_\rho(\omega)$ satisfies hypotheses \ref{it:assumption1}-\ref{it:assumption2} if the corresponding function $\mathtt G_0\in \mathcal N^{3,d}_\rho(\omega)$ does. 
\end{rem}
The first observation is that an analog of Lemma \ref{lem:resnonatpaths} holds for $G_0$.
Moreover, due to the sign condition on the entries of the frequency vector, we can guarantee that the paths constructed in Lemma \ref{lem:resnonatpaths} are contained in the positive quadrant.

\begin{lem}\label{lem:resonatpathsellipt} 
    Fix any $\rho>0$, and let  $\omega\in \mathbb R^3$ satisfy: $\omega_i\neq 0$ for all $i=1,2,3$, and  
\begin{equation}\label{eq:omegasing}
\operatorname{sign}\omega_1=\operatorname{sign}\omega_2\neq \operatorname{sign}\omega_3.
\end{equation}

Then, provided that $G_0$ belongs to an open and dense subset of $\mathcal M^{3,d}_\rho(\omega)$, 
there exists $
    \rho_0(G_0)>0$ 
    such that for any $n\in\mathbb N$
    the system of equations
\begin{equation}\label{eq:ellipticpathseqsdefn}
          \mathtt G_0(a^{(n)}(u))=e_n,\qquad\qquad \bs s^{(n)}\cdot  \nabla \mathtt G_0(a^{(n)}(u))=0
\end{equation}
    defines a real-analytic embedded curve $a^{(n)}(u):(0,\rho_0) \to  (0,\rho)^3$ of the form
\[
    a^{(n)}(u)=(a^{(n)}_1(u),a^{(n)}_2(u),u),
\]
    satisfying 
    \[
    \lim_{n\to \infty} \lVert a^{(n)}(0)\rVert \to 0. 
    \]
\end{lem}
Lemma \ref{lem:resonatpathsellipt} guarantees that there exists a sequence of resonant curves in the $(q,p) \in \mathbb{R}^6$ plane  of the form $\tilde I= a^n(u)$,
approaching the origin. We omit its proof, which is analogous to Lemma \ref{lem:resnonatpaths}.

\medskip

Our goal now is to build a perturbation of $G_0$ for which normally hyperbolic cylinders emerge along the resonant paths in Lemma \ref{lem:resonatpathsellipt} and, in addition, exhibit a rich net of homoclinic cylinders. The following is the natural generalization of Theorem \ref{thm:mainparametricoutline}.

\begin{thm}\label{thm:mainparametricoutlineelliptic}
Fix any  $\rho>0$, let  
 $\omega$ be as in \eqref{eq:omegasing} and fix a  Dirichlet sequence $\{s^{(n)}\}_n\subset\mathbb Z^2\setminus\{0\}$ associated to $\frac{\omega_2}{\omega_1}$.   Let  $G_0\in \mathcal M^{3,d}_\rho(\omega)$ satisfy \ref{it:assumption1}-\ref{it:assumption2} and   let $\rho_0=\rho_0(G_0)$ be given by Lemma \ref{lem:resonatpathsellipt}.    
Fix any $\delta>0$. Then, there exists a sequence of homogeneous polynomials $\{f_j\}$ and a parametric Hamiltonian of the form 
\begin{equation}\label{eq:limitHamiltonianoutlineelliptic} 
    G_{\varepsilon}=
G_0+\varepsilon g_0,
\qquad\qquad
{where}\qquad\qquad g_0(\tilde q,\tilde p)=\sum_{j\in\mathbb N}f_j(\tilde q,\tilde p)
\end{equation} for which one can find a sequence $\ \{\varepsilon_n\}_n$
with $0<\varepsilon_n\leq 1/n$ such that the following holds. 
For any 
$\eta\in B_{\rho_0}^{d-3}$  there exists a residual subset $\mathcal U_{\rho}(\delta)\subset \mathcal B^2_{\rho} (\delta)\subset \mathcal{P}_{\rho}^2$ (see \eqref{eq:banachspaceelliptintro}) such that for any $g\in \mathcal{U}_{\rho}(\delta)$  the Hamiltonian 
\begin{equation*}%\label{eq:firstperturb-ell}
\mathcal G_\varepsilon (g)=G_\varepsilon+\varepsilon^2 g
\end{equation*}
satisfies the following conditions:
\begin{enumerate}
    \item  
    For any $\varepsilon\in[0,1]$ we have $\mathcal G_{\varepsilon}(g)\in \mathcal {P}_{\rho}^{3,d}$ (see \eqref{eq:banachspaceelliptic}) and 
    \[
    \sup_{(\tilde q,\tilde p)\in  \mathbb B^3_{\rho}} (|\mathcal G_\varepsilon(g)-G_0|)<2\delta.
    \]
    \item 
    For any $n\in\mathbb N$ there exists a (locally) transverse section $\Sigma_n \subset \{\mathcal G_{\varepsilon_n} (g)=e_n\}$ such that   the Poincar\'e map $\Phi_n:\Sigma_n\to \Sigma_n$ induced by the flow of the Hamiltonian $\mathcal G_{\varepsilon_n}(g)$   possesses a  normally hyperbolic invariant manifold (cylinder) which admits a parametrization of the form
\begin{equation}\label{eq:2dcylinderoutlineelliptic}
{\Lambda}_n(g)= \{  (q_1^{(n)}(I_3),q_2^{(n)}(I_3), q_3,p_1^{(n)}(I_3),p_2^{(n)}(I_3),p_3)\colon \  \frac12 (q_3^2+p_3^2)=:I_3\in[0,\rho_0]\}
\end{equation}
for some real-analytic functions $q_i^{(n)}, p_i^{(n)}$, $i=1,2$,  satisfying 
\[
|\frac 12((q_i^{(n)})^2+(p_i^{(n)})^2)-a_{i}^{(n)}|_{\C^1}\leq \frac 1n
\]
for $a^{(n)}=(a_1^{(n)},a_2^{(n)})$ as in \eqref{eq:ellipticpathseqsdefn}.
\item 
The restriction  ${\Phi_n}|_{\Lambda_n}$ is an integrable twist map which leaves invariant the foliation by  circles $\{I_3=\mathrm{const}\}$.
\item  
There exists $m\in\mathbb N$ (depending a priori on $n$) and a covering 
\[
[0,\rho_0]\subset \bigcup_{0\leq i\leq m} [\mathcal{I}_i^-,\mathcal{I}_i^+]
\]
satisfying $\mathcal{I}_{i+1}^-<\mathcal{I}_{i}^+$ for all $i=0,\dots,m-1$  such that  for any $i=0,\dots,m-1$ there exists a homoclinic cylinder $\Gamma_{n,i}$ admitting a parametrization of the form
\[
\Gamma_{n,i}(g)=\{  (\hat q_{1,i}^{(n)}(I_3), \hat q_{2,i}^{(n)}(I_3), q_3,\hat p_{1,i}^{(n)}(I_3),\hat p_{2,i}^{(n)}(I_3),p_3)\colon \ \frac 12 (q_3^2+p_3^2)=:  I_3\in[\mathcal{I}_i^-,\mathcal{I}_i^+]\}
\]
for certain real-analytic functions $\hat q_{j,i}^{(n)}, \hat p_{j,i}^{(n)}$ with $j=1,2$ and $i=0,\dots,m-1$.
\end{enumerate}

\end{thm}

\subsubsection{Proof of Theorem \ref{thm:mainparametricoutlineelliptic}}

As we did in Section \ref{sec:2dof}, we will work with Hamiltonian systems  of $2$ degrees-of-freedom, thinking of $ I_3= \frac{p_3^2+q_3^2}{2}$ as a parameter.  That is, we will consider Hamiltonians of the form 
\[
K_\varepsilon (q,p;  I_3)=K_0(q,p;I_3)+\varepsilon k(q,p;I_3)
\]
where $q=(q_1,q_2)$, $p=(p_1,p_2)$ and $K _0 \in \mathcal M^{2,d}_\rho(\omega)$. The  following technical result will prove useful to translate our construction in Section \ref{sec:2dof} to the elliptic scenario.

%We now show how to build the function $g_0(\tilde q, \tilde p)$  (see \eqref{eq:limitHamiltonianoutline}) in Theorem \ref{thm:mainparametricoutline} (expressed in action-angle coordinates) so that it becomes analytic in the neighborhood of the elliptic equilibrium. As we did in the proof of Theorem \ref{thm:mainparametricoutline} it is enough to consider perturbations depending only on the first two coordinate pairs (see the discussion at the beginning of Section \ref{sec:cylinders} and, in particular, Theorem \ref{thm:mainparametric}). 
\begin{lem}\label{lem:hompolynomial}
Let $(q,p)\in \mathbb R^4$ and $(\theta, I)\in \mathbb T^2\times \mathbb R^2_+$ be related by a symplectic coordinate change 
\begin{equation}\label{eq:symplecticpolar}
\psi:(\theta,I)\to (q,p),\qquad\qquad \begin{cases}
q_i=\sqrt{2I_i} \cos \theta_i \\
p_i=\sqrt{2I_i} \sin \theta_i ,
\end{cases}
\end{equation}
$i=1,2$. Then, for any $s_1,s_2\in\mathbb N$ the function $f$ defined by
\begin{equation}\label{def:f_j}
f = g \circ \psi^{-1}, \quad g (I, \theta)= I_1^{s_1} I_2^{s_2} \cos (2 (s_1 \theta_1 + s_2 \theta_2 )), 
\end{equation}
is a homogeneous polynomial of degree $2(s_1+s_2)$ in the variables $q_1,q_2,p_1,p_2 $. 
\end{lem}
\begin{proof} Note that for each $i$ we have:
$$
2I_i \cos 2\theta_i = q_i^2-p_i^2, \quad I_i\sin 2\theta_i =  q_i p_i;
$$
$$
 \cos (2 (s_1 \theta_1 + s_2 \theta_2 )) =  \cos (2 s_1 \theta_1 ) \,  \cos (2  s_2 \theta_2 )-  \sin (2 s_1 \theta_1 ) \,  \sin (2 s_2 \theta_2 ),
$$
and by De Moivre's formula we have for each $s\in \mathbb N$:
$$
 \cos (2 s \theta ) =   P_{cos} ( \cos  (2  \theta ), \,  \sin (2  \theta ) ) , \quad   \sin (2 s \theta  ) =  P_{sin}( \cos (2  \theta ), \,  \sin (2  \theta  )),
$$
where $P_{cos}$ and $P_{sin}$  are homogeneous polynomials of degree $s$.
\end{proof}

Given $\omega\in\mathbb R^2$ with $\omega_1\omega_2<0$ and
$\omega_2/\omega_1\in\mathbb R\setminus\mathbb Q$, let
$\{s^{(n)}\}_{n}\subset\mathbb N^2\setminus\{0\}$ be a Dirichlet sequence
associated with the positive irrational number $-\omega_2/\omega_1$.

 Then for any $n\in\mathbb N$ we denote by 
$f_n(q,p):\mathbb R^4\to \mathbb R$ the homogeneous polynomial constructed in Lemma \ref{lem:hompolynomial} for the natural numbers $s_1^{(n)},s_2^{(n)}\in\mathbb N$, that is,
\begin{equation}\label{eq:fn}
f_n = g_n \circ \psi^{-1},\quad g_n (I, \theta)= I_1^{s^{(n)}_1} I_2^{s^{(n)}_2} \cos (2 (s_1^{(n)} \theta_1 + s_2 ^{(n)}\theta_2 )).
\end{equation}
We now divide the proof of Theorem \ref{thm:mainparametricoutlineelliptic} into two steps.
\medskip

\noindent\textbf{Step 1: Creation of a sequence of hyperbolic periodic orbits.}
The following is the analogue of Lemma \ref{lem:auxiliaryperiodicorbits2} for Hamiltonians displaying an elliptic fixed point.

\begin{lem}\label{lem:periodicorbitselliptic}
    Fix any $\delta, \rho>0$, and suppose that  $K_0\in \mathcal M^{2,d}_\rho(\omega)$ 
    satisfies 
  %  \footnote{With an abuse of notation, we %identify $H_0$ with its expression in %symplectic coordinates %\eqref{eq:symplecticpolar}, which, by %construction, is a function of $I_1,I_2$ for %$(I_1,I_2)\in \mathbb R^2_+$.}  
  \eqref{eq:nondegh} in the sense of Remark \ref{rem:translation_H1-H2}.
    Fix an arbitrary $n\in\mathbb N$, and consider a parametric Hamiltonian of the form 
    \begin{equation}\label{eq:truncatedHamiltonianperelliptic} 
    K_\varepsilon^{(n)} =K_0+\varepsilon k _0^{(n)},\qquad\qquad \text{with}\qquad\qquad k _0^{(n)}(q,p)=\sum_{j\leq n} b_j f_j(q,p),
    \end{equation}
    where $f_j$ are of the form \eqref{eq:fn}, and
     $\{b_j\}_{j\leq n}\subset \mathbb R_+$ are chosen so that  \begin{equation}\label{eq:smallnesstruncatedellipt}
     \sup_{(q,p)\in \mathbb B^2_\rho} |k_0^{(n)}(q,p)|<\delta.
   \end{equation}
   Then there exists $\varepsilon_n>0$
such that for all $0<\varepsilon\leq\varepsilon_n$ and all $\frac{q_3^2+p_3^2}{2}=:I_3\in [0,\rho_0]$ the Hamiltonian $K_{\varepsilon}^{(n)}(\cdot;I_3)$ in \eqref{eq:truncatedHamiltonianperelliptic} possesses a hyperbolic periodic orbit $ \gamma_n(I_3;\varepsilon)=\gamma_n(I_3)$ that in symplectic polar coordinates \eqref{eq:symplecticpolar} admits a parametrization of the form \eqref{eq:parametrizationperiodicproposition}. 
Moreover, $\gamma_n(I_3)\subset\{K^{(n)}_{\varepsilon_n}=e_n\}$,  and the dependence of $ \gamma_{n}(I_3)$ on  $I_3$ is real-analytic. The same statement holds for any Hamiltonian of the form 
\begin{equation}\label{eq:calkepsilon}
\mathcal{K}^{(n)}_\varepsilon(k):=K^{(n)}_\varepsilon+\varepsilon^2 k= K_0+\varepsilon k_0^{(n)}+\varepsilon^2 k,\qquad\qquad k \in \mathcal B^2_\rho(\delta),
\end{equation}
where $\mathcal B_\rho^{2}(\delta)\subset \mathcal P^2_\rho$ is the ball of radius $\delta$ centered around the origin.
\end{lem}

\begin {rem} \label{rem.kn} 
As in Remark \ref{rem.hn}, the coefficients $b_n$ can be chosen arbitrarily small, so that, despite the fact that the degree of $k_0^{(n)}$ tends to infinity, their analytic norms on the fixed ball can be made uniformly as small as desired. This will only affect the size of the corresponding hyperbolic eigenvalues.
\end{rem}

The proof of this result, which is entirely similar to that of Lemma \ref{lem:auxiliaryperiodicorbits2}, is presented in Appendix \ref{sec:appendixhyperbolic}. The next proposition is completely analogous to Proposition \ref{prop:LagrangianW^su}
\begin{prop}\label{prop:LagrangianW^su_elliptic}
For any $k\in\mathcal B_{\rho}^2$ the stable and the unstable invariant manifolds of the hyperbolic periodic orbit $\gamma_n(h)$ associated to the Hamiltonian \eqref{eq:calkepsilon}, when expressed in action-angle variables, admit a graph parametrization of the form \eqref{eq:Laggraphparametrizations}.
\end{prop}
The proof of this proposition is presented in Appendix \ref{sec:appendixhyperbolic}.

\medskip

\noindent \textbf{Step 2: Existence of transverse homoclinics.} The next step is to establish a result analogous to Proposition \ref{prop:splitting}, that is,  to find a residual set $\mathcal U_\rho^{(n)}\subset  \mathcal P^2_\rho$ of suitable perturbations $k \in \mathcal U_\rho^{(n)}$, such that the periodic orbits of the Hamiltonian  $\mathcal{K}^{(n)}_\varepsilon$ (see \eqref{eq:calkepsilon}) obtained in Lemma \ref{lem:periodicorbitselliptic}, have transverse homoclinic intersections.

This part of the argument holds unchanged mutatis mutandis. 
To see why, let us recall that  the core of the proof of Proposition \ref{prop:splitting} (given in Section \ref{sec:mainsptlittingsection}) is to establish Proposition \ref{prop:surjective}. 
In the proof of that proposition we deduce that the  differential of the operator in \eqref{eq:frechetmaps} is surjective in two steps. First, we make use of compactly supported functions in $\C^\infty_c(U_n)$, where $U_n\subset \mathbb T^2\times B_\rho^2$ is a suitable small open set that intersects a compact piece of $W^u_{\mathrm{loc}}(\gamma_n)$. Second, we conclude, exactly as in the proof of Proposition  \ref{prop:surjective}, by approximating the smooth compactly supported perturbations in the required $\C^r$-topology on $\overline{U_n}$ by analytic ones.

For the case of elliptic equilibria let $\gamma_n'$ be any of the hyperbolic periodic orbits in Lemma \ref{lem:periodicorbitselliptic}. We will use the analog of functional $F$ in Proposition \ref{prop:surjective} but defined in $\mathcal{P}^2_\rho$.  
Let $U_n'\in B_\rho^4$ be an open set  intersecting a compact piece of $W^u_{\mathrm{loc}}(\gamma_n')$ (see Figure \ref{fig:fig3}).  
Since $\mathcal P^2_\rho$ is dense within $\C^\infty_c(U_n')$ the corresponding analog of Proposition \ref{prop:surjective} also holds in this case and we deduce a statement analogous to Proposition \ref{prop:splitting}.

\begin{prop}\label{prop:splittingellipt}
 Fix any  $\delta, \rho>0$, let $K_0\in\mathcal N_\rho^{2,d}(\omega)$ satisfy \eqref{eq:nondegh} with $\rho_0=\rho_0(K_0)$ given by Lemma \ref{lem:resonatpathsellipt}, and choose any sequence  $\{b_j\}_{j\leq n}\subset\mathbb R_+$ such that $k_0^{(n)}$, defined by \eqref{eq:truncatedHamiltonianperelliptic},  satisfies \eqref{eq:smallnesstruncatedellipt}.   
 There exists $\varepsilon_n>0$  such that for any $0<\varepsilon\leq \varepsilon_n$ there exists a residual subset $\mathcal U^{(n)}_\rho(\delta)\subset \mathcal B^2_\rho(\delta)\subset \mathcal P_\rho^2$ such that the following holds.  
 
For any  $h\in \mathcal U^{(n)}_\rho(\delta)$ there exists $m\in\mathbb N$ (depending a priori on $n$ and $h$) and a covering 
 \[
 [0,\rho_0]\subset \bigcup_{0\leq i\leq m} [\cI_i^-,\cI_i^+]
, \]
 satisfying $\cI_{i+1}^-<\cI_{i}^+$ for all $i=0,\dots,m-1$, such that  for any $i=0,\dots,m-1$ and $I_3\in[\cI_i^-,\cI_i^+]$ the Hamiltonian $\mathcal K^{(n)}_{\varepsilon}(k)(\cdot;I_3)$ in \eqref{eq:calkepsilon} admits a  transverse homoclinic orbit 
 $z_{n,i}(h)(I_3)$ to $\gamma_n(h)(I_3)$ which depends on $I_3$ in a real-analytic way for $I_3\in[\cI_i^-,\cI_i^+]$. 
\end{prop}
 
\medskip

\noindent\textbf{Step 3: Completion of the proof of Theorem \ref{thm:mainparametricoutlineelliptic}.} The proof of Theorem \ref{thm:mainparametricoutlineelliptic} is obtained from Lemma \ref{lem:periodicorbitselliptic} and Proposition \ref {prop:splittingellipt} in the very same way we deduced Theorem \ref{thm:mainparametricoutline} from Lemma \ref{lem:auxiliaryperiodicorbits2} and Proposition  \ref{prop:splitting}.

\begin{figure}
    \centering
    \includegraphics[scale=0.65]{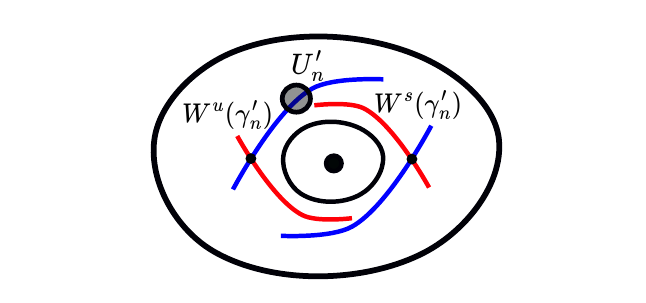}
    \caption{The invariant manifolds of the  hyperbolic periodic orbit $\gamma_n'$ and the domain $U_n'$ where we localize the perturbation.}
    \label{fig:fig3}
\end{figure}

\subsection{Drifting orbits and the proof of Theorem \ref{thm:approxbydiffusionellipticpoints}} \label{sec.82}
Again, no additional arguments are required in the construction of drifting orbits, and we obtain a result similar to Theorem \ref{thm:_apply_GToutline} with obvious modifications.

\begin{thm}\label{thm:_apply_GToutlineellipt} 
 Fix any  $\rho>0$; let $\omega$, $G_0$ and $\rho_0=\rho_0(G_0)$ be as in the assumption of Theorem \ref{thm:mainparametricoutlineelliptic}.
There exists a sequence $\{ \mu_n\}_n$,   $ 0<\mu_n\leq 1/n$  such that the following holds. Fix any $\delta>0$, any $\eta\in B^{d-3}_{\rho_0}$, let  
$\{\varepsilon_n\}_n$ and $\mathcal{U}_{\rho}(\delta)\subset \mathcal B^2_{\rho}(\delta)\subset \mathcal P^2_\rho$ be as in the conclusion of Theorem \ref{thm:mainparametricoutlineelliptic}  and denote 
\[
\mathcal G_{\varepsilon,\mu}(g_1,g_2)=\mathcal G_\varepsilon(g_1)+\mu g_2, \qquad\qquad g_1\in \mathcal U_{\rho}(\delta),
\quad g_2\in \mathcal B^3_{\rho}(\delta)\subset \mathcal P^3_\rho.
\]  
 Then for any $g_1\in \mathcal U_{\rho}(\delta)$ and any  $n\in\mathbb N$ there exists an open and dense set $\mathcal V^{(n)}_{\rho}(\delta) \subset \mathcal B^3_{\rho}(\delta)$ such that for any $g_2\in\mathcal V^{(n)}_{\rho}(\delta)$   
the Hamiltonian $\mathcal G_{\varepsilon_n,\mu_n}(g_1,g_2)$  satisfies the following:
\begin{enumerate}
\item it displays a normally hyperbolic invariant manifold $\Lambda_n(g_1,g_2)$ which is $\frac 1n$-close to $\Lambda_n(g_1)$ in Theorem \ref{thm:mainparametricoutlineelliptic} (see the parametrization  \eqref{eq:2dcylinderoutlineelliptic}).
    \item If $U^\pm_n$ is an open neighborhood of $\partial^{\pm} \Lambda_n(g_1,g_2)$ then there exists an orbit of this Hamiltonian that connects $U^-_n$ to $U^+_n$.
\end{enumerate} 
\end{thm}

We have completed the proof of Theorem \ref{thm:approxbydiffusionellipticpoints}.
\appendix

\section{Averaging and hyperbolic periodic orbits}\label{sec:appendixhyperbolic}

This section is devoted to the proof of:
\begin{enumerate}
    \item  Lemma \ref{lem:auxiliaryperiodicorbits2} about the existence of the hyperbolic periodic orbit $\gamma_n(h)$ for the Hamiltonians $\mathcal H_\varepsilon^{(n)}$ in \eqref{eq:epssquareHamiltonian} and its analogue for the case of elliptic equilibria, Lemma \ref{lem:periodicorbitselliptic}.
    \item Proposition \ref{prop:LagrangianW^su}.  about the existence of Lagrangian graph parametrizations of the un/stable invariant manifolds of $\gamma_n(h)$.
\end{enumerate}

\subsection{Proof of Lemma \ref{lem:auxiliaryperiodicorbits2} }\label{sec:proofaveraging}

Consider a Hamiltonian $\mathcal H_\varepsilon^{(n)}$ as in \eqref{eq:epssquareHamiltonian}. 
In this section, with an abuse of notation, we write $H$ instead of $\mathcal H^{(n)}_{\varepsilon}(h)\circ\phi_n$, drop the sub- and superscript $n$ from the functions $N_\omega,V,P, \phi$ and write $h$ instead of $ h_0\circ\phi_n$, where $\phi$ is the coordinate system \eqref{eq:symplecticlatticechange}. 
Recall that, in this coordinate system, the Hamiltonian $H$  reads: 
\begin{equation}\label{eq:simplifiedH}
H= \underbrace{N_\omega+ \varepsilon V}_{H_{\mathrm{av},\varepsilon}}+ \varepsilon  P +\varepsilon^2 h, \quad N_\omega (\bs J) =H_0\circ \phi_n =H_0(A_n \bs J).
\end{equation}
The proof is now  divided into a number of steps. 

\begin{rem}
 In the following, we omit the dependence on $I_3$. Note, however, that the dependence of all the objects on $I_3$ is real-analytic and all the estimates are uniform for $I_3\in [-\rho_0,\rho_0]$.
\end{rem}

\medskip

\noindent\textit{ Notation for the norms.}
%{Functional setting.} 
Given any open set $D\subset\mathbb C^2$ and a constant ${\tilde\sigma}>0$, we let 
\[
\lVert f\rVert_{D,{\tilde\sigma}}=\sum_{l\in\mathbb Z^2} \sup_{J\in D} |f^{[l]}(\bs J)| e^{|l|{\tilde\sigma}}.
\]
We simply write $\lVert f\rVert_{\tilde\sigma}$ if $f$ does not depend on $\bs J$.
\medskip

\noindent\textit{Estimates for $P$.} For $\sigma_n>0$  sufficiently small (depending only on $\{s^{(n)}\}_{j\leq n}$) we may assume that $P$ satisfies 
    \[
    \lVert  P\rVert_{\mathbb{B}_\rho ^2 , \sigma_n}\leq 1.
    \]
   Denote by $\mathcal Z$ the Fourier support of $P$, and let $L=\max\{|l|_1\colon l\in\mathcal Z\}$.
    \medskip
    
\noindent\textit{Small divisor estimates.}
Recall that 
\[
\nabla_{ \bs J} N_\omega ( \bs J)=\left(\nabla_I H_0(A_n \bs J)\cdot s^{(n)}, \ \nabla_I H_0(A_n \bs J)\cdot k^{(n)} \right).
 \]
Let $\rho_n>0$ be sufficiently small (depending only on $\{s^{(n)}\}_{j\leq n}$)  and denote:
\[
D^{(n)}_{\rho_n,\rho_0}=\{ \bs J\in  \mathbb C ^2 \colon |\nabla_I H_0(A_n \bs J)\cdot s^{(n)}|\leq \rho_n /L,\ \lVert A_n \bs J\rVert\leq \rho_0\}\subset\mathbb C^2.
\]
The second inequality in the brackets above only means that we look at the values of $\bs J$ for which $\lVert I\rVert\leq \rho_0$.
Observe that $D^{(n)}_{\rho_n,\rho_0}$ is a neighborhood of the point $\tilde a^{(n)}$ given in \eqref{eq:jn*}.
 It is now easy to check that for any $ \bs J\in D_{\rho_n,\rho_0}^{(n)}$ and any 
 %\color{blue} (I changed $(l_\varphi,l_t)$ by %$(l_1,l_2)$, as it seems more natural.)  \color{black}
 $l=(l_1,l_2)\in\mathcal  Z$ with $l_{2}\neq 0$ we have
\begin{align*}
|\nabla_{ \bs J} N_\omega( \bs J)\cdot l|\geq &|\nabla_I H_0(A_n \bs J)\cdot k^{(n)}|\ |l_2| -\rho_n|l_1|/L\\
\geq &\frac{1}{|s_1^{(n)}|} \partial_{I_2} H_0(A_n\bs J)\  |l_2|-O(\rho_n)\\
\geq &\frac{\omega_2}{2|s_1^{(n)}|}\  |l_2|,
\end{align*}
where in the first inequality we have used the fact that $ 
\nabla_I H_0(A_n \bs J)\cdot s^{(n)}=O(\rho_n/L)$, in the second inequality we have used \eqref{eq:sk}  and, in the third inequality we have shrinked (if necessary) the value of $\rho_n$ and used the fact that $\partial_{I_2} H_0(0)=\omega_2$ so that, possibly after diminishing $\rho_0$, we can assume that for any $\bs J\in D^{(n)}_{\rho_n,\rho_0}$ we have $|\partial_{I_2} H_0(A_n\bs J)|> \omega_2/2$.
\medskip

\noindent\textit{Notation.} Below we write $a\lesssim_n b$ if there exists a constant $C$, depending on $n$ but not on $\varepsilon$, such that $a\leq C b$.

\medskip

\noindent\textit{The averaging  conjugacy.}  
Denoting $e(x)=\operatorname{exp}(2\pi i x)$,  we define 
\[
G(\varphi,\tt,J,E)=  \frac{i}{2\pi }\sum_{l\in \mathcal Z} \frac{P^{[l]} }{\ \nabla_{\bs J}N_\omega(\bs J)\cdot l} \  e(l\cdot (\varphi, \tt  )),
\]
which solves 
\[
\{N_\omega,G\}+ P=0
\]
and satisfies: 
\[
\lVert G\rVert_{D^{(n)}_{\rho_n,\rho_0},\sigma_n}\lesssim_n 1.
\]
\medskip

\noindent\textit{Estimation of the error term}:  Let $\phi^t_{\varepsilon G}$ denote the time-$t$ flow of the Hamiltonian $\varepsilon G$, and let $\phi_{\varepsilon G}=\phi^1_{\varepsilon G}$.  Denote
\[
\widehat P=h\circ\phi_{\varepsilon G}+\int_0^1 \{ V+t P,G\}\circ\phi^t_{\varepsilon G}\mathrm{d}t.
\]     
From the definition of $G$, we have:
\begin{equation}\label{eq:onestepaveraging}
H\circ \phi_{\varepsilon G}=H_{\mathrm{av},\varepsilon}+\varepsilon^2 \widehat P .
\end{equation}
 Hence, since for $0\leq t\leq 1$ we have
\[
\lVert \{V+t P,G\}\rVert_{D^{(n)}_{\rho_n/2,\rho_0/2},\sigma_n/2}\lesssim_n 1,
\]
is a standard exercise to check that 
\[
\lVert \widehat P\rVert_{D^{(n)}_{\rho_n/4,\rho_0/4},\sigma_n/4}\lesssim_n 1.
\]
\medskip

\noindent\textit{The averaged system.} The averaged system $H_{\mathrm{av},\varepsilon}(\varphi,J)=N_\omega(\bs J)+\varepsilon V(\varphi)$ 
is integrable, and the equations of motion read
\begin{equation}\label{eq:averagedproofvfield}
    \begin{aligned}
\dot\varphi=&\partial_{J}H_{\mathrm{av},\varepsilon}=\partial_J N_\omega (\bs J)
\qquad\qquad 
&\dot J=&-\partial_\varphi H_{\mathrm{av},\varepsilon}=
-\varepsilon \partial_\varphi V=
\varepsilon b_n \sin \varphi\\
\dot \tt=&\partial_{E}H_{\mathrm{av},\varepsilon}=
\partial_E N_\omega (\bs J)\qquad\qquad &\dot E=&0.       
    \end{aligned}
\end{equation}
Take  $\rho_0=\rho_0(H_0)$ be given by Lemma \ref{lem:resnonatpaths}.
Then, the vector field \eqref{eq:averagedproofvfield} displays a hyperbolic periodic orbit which, in coordinates $(\varphi,\tt,J,E)$ can be parametrized as
\begin{equation}\label{eq:unperturbedperorbits,app}
\gamma^{\mathrm{av}}_{n}=\{(0,\tt, \tilde a^{(n)}_J,\tilde a^{(n)}_E) \colon \tt\in\mathbb T\}
\end{equation}
with $
\tilde a^{(n)} =A_n^{-1} a^{(n)}$ and $a^{(n)}$ is the resonant path given in \eqref{eq:anpaths0}.
The dynamics on $\gamma^{\mathrm{av}}_{n}$ is given by \eqref{eq:averagedproofvfield},
therefore it is a periodic orbit of period (see \eqref{eq:lowerboundfastfreq})
\[
\nu_n(I_3) =  \frac{s_1^{(n)}) } {\partial_{I_2}H_0(a^{(n)}(I_3))} \sim \frac {s_1^{(n)}}{\omega_2}=-\frac {s_2^{(n)}}{\omega_1}.
\]
\medskip

\noindent\textit{Application of the implicit function theorem}: The discussion in the previous step can be reinterpreted as follows. Let 
\[
\Phi_{\mathrm{av}}:(\varphi,J,E)\in \mathbb T_{\sigma_n/8}\times D_{\rho_n/8,\rho_0/8}^{(n)}\mapsto (\Phi_{\mathrm{av},\varphi}\ \ \mathrm{mod}\ 1,\Phi_{\mathrm{av},J},\Phi_{\mathrm{av},E})\in (\mathbb C/2\pi\mathbb Z)\times \mathbb C^2
\]
denote the Poincar\'e map induced by the flow of $H_{\mathrm{av},\varepsilon}$ on the section $\{\tt=0\}$ and define the map 
\begin{align*}
\Psi_{\mathrm{av}}: \mathbb T_{\sigma_n/8}\times D_{\rho_n/8,\rho_0/8}^{(n)} &\to \mathbb C\times (\mathbb C/2\pi\mathbb Z)\times \mathbb C\\
(\varphi,J,E)&\mapsto (H_{\mathrm{av}}-e_n,\Phi_{\mathrm{av},\varphi}-\varphi\ \ \mathrm{mod}\ 1,\Phi_{\mathrm{av},J}-J).
\end{align*}
By energy conservation, a zero of the map $\Psi_{\mathrm{av}}$ corresponds to a periodic orbit of the Hamiltonian $H_{\mathrm{av}}$ and vice versa. In particular, the point $(0,\bs{\widetilde J})$ with 
$
\bs{\widetilde J}=\tilde a^{(n)}$ 
is a solution of the equation $\Psi_{\mathrm{av}}(\varphi,J,E)=0$. We claim (and verify below) that the differential of $\Psi_{\mathrm{av}}$ at $(0,\bs{\widetilde J})$ is of the form
\begin{equation}\label{eq:claim} 
D\Psi_{\mathrm{av}}(0,\bs{\widetilde J})=\begin{pmatrix}
    0&0&\frac{1}{\nu}\\
    O(\varepsilon)& \nu \partial^2_{I^2} H_{0}(A_n\bs{\widetilde J}) s^{(n)}\cdot s^{(n)}+O(\varepsilon)&O(1)\\
   -\nu \varepsilon  \partial_{\varphi^2}^2V(0,\bs{\widetilde J})+O(\varepsilon^2)&O(\varepsilon)&O(1)
\end{pmatrix},
\end{equation}
where 
\begin{equation}\label{eq:returntime}
\nu=\frac{1}{\partial_I H_{\mathrm{av},\varepsilon}(0,\bs {\widetilde J})\cdot k^{(n)}}  =\nu _n(I_3). 
\end{equation}
\medskip

For the full Hamiltonian $H$, in the new coordinate system given by the averaging transformation induced by  the Hamiltonian $G$ constructed above, the vector field reads
\begin{align*}
\dot\varphi=&\partial_{J}H_{\mathrm{av},\varepsilon}+\varepsilon^2\partial_{J}\widehat P \qquad\qquad &\dot J=&-\partial_\varphi H_{\mathrm{av},\varepsilon}-\varepsilon^2\partial_\varphi \widehat P\\
\dot t=&\partial_{E}H_{\mathrm{av},\varepsilon}+\varepsilon^2\partial_{E}\widehat P \qquad\qquad &\dot E=&-\varepsilon^2\partial_t \widehat P .
\end{align*}
%
%\blue Observe that if we add a term $\eps^2 h$ to the original Hamiltonian, after the averaging procedure we will get  similar Hamiltonian with a different $\hat P$ \orange now the proof is done directly for the term with $O(\varepsilon^2)$, we can remove this. \black
Hence, the corresponding Poincaré map $\Phi:(\varphi,J,E)\in \mathbb T_{\sigma_n/8}\times D_{\rho_n/8,\rho_0/8}^{(n)}\mapsto (\Phi_\varphi,\Phi_{J},\Phi_{E})$  on the section $\{t=0\}$ is given by a $O_{\C^1}(\varepsilon^2)$ perturbation of the map $\Phi_{\mathrm{av}}$. In particular, the map 
\begin{align*}
\Psi: \mathbb T_{\sigma_n/8}\times D_{\rho_n/8,\rho_0/8}^{(n)} &\to \mathbb C\times (\mathbb C/2\pi\mathbb Z)\times \mathbb C\\
(\varphi,J,E)&\mapsto (H\circ\phi_{\varepsilon G}-e_n,\Phi_{\varphi}-\varphi\ \ \mathrm{mod}\ 1,\Phi_{J}-J),
\end{align*}
is also a $O_{C^1}(\varepsilon^2)$ perturbation of the map $\Psi_{\mathrm{av},\varepsilon}$. Again, by energy conservation, a zero of the map $\Psi$ corresponds to a periodic orbit of the Hamiltonian $H$ and vice versa. Finally, we easily deduce the existence of $(\hat \varphi,\bs{\widehat J})$ satisfying $\Psi(\hat \varphi,\bs{\widehat J})=0$ by means of \eqref{eq:claim} and the implicit function theorem. Moreover, from the construction it is not difficult to check that 
\[
|\hat\varphi|\lesssim_n \varepsilon,\qquad\qquad |\hat J-\tilde a_J^{(n)}|\lesssim_n \varepsilon,\qquad\qquad |\hat E-\tilde a_E^{(n)}|\lesssim_n\varepsilon .
\]

\medskip

\noindent\textit{Conclusion.}
We have proved\footnote{Throughout the proof we have omitted the dependence on $I_3$ in order to alleviate the notation. Notice, however, that the unique non-degeneracy conditions used in this proof,  are uniformly satisfied provided that $I_3$ varies in a compact small interval (independent of $n$) so all the objects depend on $I_3$ in a real-analytic fashion.} the existence of the hyperbolic periodic orbit $\gamma_{n}(I_3)$ admitting a parametrization  of the form \eqref{eq:parametrizationdistorted}.  Moreover, the $\C^1$ estimates in \eqref{eq:parametrizationdistortedestimates} are easily obtained by making use of Cauchy estimates (note that $\rho_n,\sigma_n$ do not depend on $\varepsilon$).
\medskip

\noindent\textit{Verification of claim \eqref{eq:claim}.} 
By definition, we have
\[
\Phi_{\mathrm{av}}(\varphi,J,E)=\phi^{\tilde \nu(\varphi,J,E)}_{H_{\mathrm{av},\varepsilon}}(\varphi,0,J,E),
\]
  where $\tilde \nu(\varphi,J,E)$ is the return time  to the section $\{\tt=0\}$. 
It is not difficult to check that
\[
 D\Phi_{\mathrm{av}}(0,\bs{\widetilde J})= \pi D\phi^{\tilde \nu(0,\bs{\widetilde J})}_{H_{\mathrm{av},\varepsilon}}(0,0,\bs{\widetilde J}),
\]
where $\pi$ is the projection to the $(\varphi,J,E)$ components. 
In particular, we only need to compute the differential of the time-$t$ map and evaluate it at $t=\tilde \nu(0,\bs{\widetilde J})$. 
Moreover, $\tilde \nu(0,\bs{\widetilde J})=\nu$ where $\nu$ is the return time in  \eqref{eq:returntime}. 
The differential of the time-$t$ map is the solution to the linear initial value problem
\[
\begin{cases}
    \frac{\mathrm d}{\mathrm d t} D\phi^t_{H_{\mathrm{av},\varepsilon}} = DX_{H_{\mathrm{av},\varepsilon}}(\phi^t_{H_{\mathrm{av},\varepsilon}}) D\phi^t_{H_{\mathrm{av},\varepsilon}}\\
    D\phi^0_{H_{\mathrm{av},\varepsilon}}=\mathrm{Id}
\end{cases},
\]
where $X_{H_{\mathrm{av},\varepsilon}}$ is the vector field \eqref{eq:averagedproofvfield}. 
Evaluated at the periodic orbit 
\eqref{eq:unperturbedperorbits,app}   the differential of the vector field reads:
\[
DX_{H_{\mathrm{av},\varepsilon}}(\phi^t_{H_{\mathrm{av},\varepsilon}}(0,\tt,\bs{\widetilde J}))     = \begin{pmatrix}
    0 & 0 &\partial_{J^2}^2 H_{\mathrm{av},\varepsilon} (0,A\bs{\widetilde J})& \partial_{J\,E}^2 H_{\mathrm{av},\varepsilon} (0,A\bs{\widetilde J}) \\
    0 & 0 &\partial_{E\,J}^2 H_{\mathrm{av},\varepsilon} (0,A\bs{\widetilde J})& \partial_{E^2}^2 H_{\mathrm{av},\varepsilon} (0,A\bs{\widetilde J}) \\
    \partial_{\varphi^2}^2 H_{\mathrm{av},\varepsilon} (0,A\bs{\widetilde J}) & 0&\partial_{\varphi\,J}^2 H_{\mathrm{av},\varepsilon} (0,A\bs{\widetilde J})& 0 \\
    0 & 0 & 0 & 0
\end{pmatrix}.
\]
Hence, if we denote  (note that $\alpha>0$ in view of assumption \eqref{eq:nondegh} and that  $\beta>0$)
\[
\alpha= \partial_{J^2}^2 H_{\mathrm{av},\varepsilon} (0,A\bs{\widetilde J})=  \partial_{I^2}^2 H_{0} (0,A\bs{\widetilde J}) s^{(n)}\cdot s^{(n)},
\quad\qquad \beta=\partial_{\varphi^2}^2 H_{\mathrm{av},\varepsilon} (0,A\bs{\widetilde J})=-\varepsilon \partial_{\varphi^2}^2V(0)
\]
and
\[
\lambda=\sqrt{\alpha\beta},
\]
it is not difficult to check (use that the $\varphi-J$ block in $DX_{H_{\mathrm{av},\varepsilon}}$ is independent)  that the matrix \eqref{eq:claim} is of the form 
\[
D\Psi_{\mathrm{av}}(0,\bs{\widetilde J})=\begin{pmatrix}
    0&0&\frac{1}{\nu}\\
  \cosh(\lambda\nu)-1&\frac{\alpha}{\lambda}\sinh(\lambda \nu)&O(1)\\
    \frac{\beta}{\lambda}\sinh(\lambda \nu)& \cosh(\lambda\nu)-1&O(1)
\end{pmatrix}.
\]
The claim follows making use of the Taylor expansion of the functions $\cosh x$ and $\sinh x$ for $|x|\ll 1$.
\medskip

\subsection{Proof of Lemma \ref{lem:periodicorbitselliptic}}

We express the Hamiltonian $\mathcal K_\varepsilon ^{(n)}$   in \eqref{eq:calkepsilon} in action-angle coordinates making use of the change of coordinates $\psi$ in Lemma \ref{lem:hompolynomial}, obtaining
\begin{equation}\label{eq:hamiltonianellipticaa}
\mathcal H_{\varepsilon,\mathrm{ell}} ^{(n)} = \mathcal K_\varepsilon ^{(n)}\circ \psi^{-1}, 
\end{equation}

The proof now follows from the very same argument in Section \ref{sec:proofaveraging}. We just indicate the minor modifications to the above discussion for the sake of completeness. 
\medskip

First we observe that, in the local coordinate system introduced in \eqref{eq:symplecticlatticechange}, the Hamiltonian $\mathcal H_{\varepsilon,\mathrm{ell}} ^{(n)}$ reads
\[
\mathcal H_{\varepsilon,\mathrm{ell}}^{(n)}\circ\phi_n=N_\omega^{(n)}+\varepsilon \underbrace{b_nf_n\circ\psi^{-1}\circ\phi_n}_{V_{\mathrm{ell}}^{(n)}}+ \varepsilon \underbrace{(k_0^{(n)}-b_nf_n)\circ\psi^{-1}\circ\phi_n+}_{P^{(n)}_{\mathrm{ell}}}\varepsilon ^2 k\circ\psi^{-1}\circ \phi_n.
\]
With an abuse of notation, we write $H$ instead of $\mathcal H^{(n)}_{\varepsilon,\mathrm{ell}}(h)\circ\phi_n$, drop the  superscript $n$ and subscript $\mathrm{ell}$ from the functions $N_\omega,V,P$ and write $k$ instead of $ h\circ\psi^{-1} \circ \phi_n$ so 
\begin{equation}\label{eq:averagingelliptic}
H=\underbrace{N_\omega+\varepsilon V}_{H_{\mathrm{av},\varepsilon}}+\varepsilon P+\varepsilon^2 h.
\end{equation}
The main difference with respect to the situation in Section \ref{sec:proofaveraging} is that now 
\[
V(\varphi,\bs J)=b_n  v_n(\bs J) \cos (2\varphi),
\]
where  $v_n(\bs J)$ is a polynomial in $\bs J$  (see  \eqref{eq:symplecticlatticechange} and \eqref{def:f_j}) 
which satisfies $v_n(\bs J)= 0$  if and only if $\bs J=0$. Hence, the averaged system $H_{\mathrm{av},\varepsilon}(\varphi,J)=N_\omega(\bs J)+\varepsilon V(\varphi,\bs J)$ 
is integrable, and the equations of motion read 
\begin{equation}\label{eq:averagedproofvfield-ell}
    \begin{aligned}
\dot\varphi=&\partial_{J}H_{\mathrm{av},\varepsilon}=\partial_J (N_\omega (\bs J)+ \varepsilon b_n v_n(\bs J) \cos (2\varphi))
\qquad\qquad 
&\dot J=&-\partial_\varphi H_{\mathrm{av},\varepsilon}=
-\varepsilon \partial_\varphi V=2
\varepsilon b_n v_n(\mathbf{J})\sin (2\varphi)\\
\dot \tt=&\partial_{E}H_{\mathrm{av},\varepsilon}=
\partial_E (N_\omega (\bs J)+\varepsilon b_n  v_n (\bs J) \cos (2\varphi) )\qquad\qquad &\dot E=&0.       
    \end{aligned}
\end{equation}
Observe that for any $\bs J$ we have $\partial_\varphi V(0,\bs J)=0$. 
On the other hand, it follows from   Lemma \ref{lem:resonatpathsellipt} 
that there exists a vector  $\bs{\widetilde J}=(\widetilde{ J},\widetilde{ E})\in D^{(n)}_{\rho_n,\rho_0}$ such that 
\[
H_{\mathrm{av},\varepsilon}(0,\bs{\widetilde J})=e_n \qquad\text{and}\qquad \partial_J H_{\mathrm{av},\varepsilon}(0,\bs{\widetilde J})=
\partial_J (N_\omega+ \varepsilon b_n v_n) =0.
\]
Moreover, it is easy to check that 
\[
|\tilde J-\tilde a_J^{(n)}|\lesssim_n \varepsilon \qquad\text{and}\qquad |\tilde E-\tilde a_E^{(n)}|\lesssim_n \varepsilon,
\]
with $\tilde a^{(n)}_J,\tilde a^{(n)}_E$ as in \eqref{eq:jn*}. Hence, we deduce that the averaged vector field \eqref{eq:averagedproofvfield-ell} displays a hyperbolic periodic orbit which can be parametrized as 
\begin{equation}\label{eq:averagedperiodicproof}
\gamma_n^{\mathrm{av}}=\{(0,\tt,\tilde J,\tilde E) \colon \tt\in \mathbb T\}.
\end{equation}
and  has period:
\[
\nu_n =\frac{1}{\partial_E (N_\omega (\bs{\widetilde J}) +\varepsilon b_n  v_n (\bs{\widetilde J}) )}.
\]
The rest of the proof of Lemma \ref{lem:periodicorbitselliptic} follows the same lines as in Section \ref{sec:proofaveraging}. Indeed, the perturbing function $P$ in \eqref{eq:averagingelliptic} is given by
\[
P= \sum _{j<n} b_j g_j\circ\phi (\varphi,\tau,J,E)    
\]
where the function $g_j$ are defined in \eqref{eq:fn}.
Consequently the function $P(\varphi,\tt,\bs J)$ has a (finite) Fourier series composed by terms which have zero mean with respect to $\tau$. Therefore, it is still possible to find an averaging conjugacy which brings \eqref{eq:averagingelliptic} into the form \eqref{eq:onestepaveraging} and conclude the proof by direct application of the implicit function theorem.

\medskip

\subsection{Proof of Propositions \ref{prop:LagrangianW^su} and \ref{prop:LagrangianW^su_elliptic}}
    We borrow the notation from Section \ref{sec:proofaveraging} and recall that we have constructed therein a local coordinate system in which $\mathcal H_\varepsilon^{(n)}(h)$ is given by a $O(\varepsilon^2)$ perturbation of the averaged Hamiltonian $H_{\mathrm{av},\varepsilon}$ in \eqref{eq:integrablen}. The latter  admits the hyperbolic periodic orbit $\gamma^{(n)}_{\mathrm{av}}$ in \eqref{eq:unperturbedperorbits}. Moreover, its stable and unstable manifolds coincide along a homoclinic manifold defined implicitly by the equations
\[
E = E_n:=  \tilde  a _E^{(n)} (I_3), \quad H_{\mathrm{av},\varepsilon}^{(n)} 
(J,E,\varphi) = H_{\mathrm{av},\varepsilon}^{(n) } (\tilde a^{(n)} (I_3),0)
\]
with $\tilde a^{(n)}$ as in \eqref{eq:jn*}. Taylor expanding the second equation, and using the fact that 
$ \partial_{J} N^{(n)}_{\omega_n}(\tilde a^{(n)} (I_3))=0$,
we deduce that the upper branch of this connection admits a parametrization of the form
\[
W_{\mathrm{h}}(\gamma_{\mathrm{av}}^{(n)})=\{(\varphi,\tau,J_h(\varphi),E_n)\colon (\varphi,\tau)\in(0,2\pi)\times\mathbb T\}
\]
with 
\begin{equation} \label{eq:homoclinicpendulum}
J_h(\varphi)=\tilde a_J^{(n)}+\sqrt{\frac{\varepsilon b_n}{C_n} (\cos \varphi -1) (1+ O(\varepsilon))}, \quad \varphi  \in [0, 2\pi],
\end{equation} 
and $C_n= \frac 1 2 \partial_{J^2} N^{(n)}_{\omega_n}(\tilde a^{(n)} (I_3))$. Therefore,  standard  perturbation theory (see for instance Chapter 2 in \cite{PalisMelo} or Appendix A in \cite{MR1237641}) shows that there exist parametrizations
\begin{equation}
  \begin{split}
	        W^u(\gamma_n(h))=&\{(\varphi,\tt,J^u(\varphi,\tau),E^u(\varphi,\tau))\colon \varphi^{(n)}(\tt)<\varphi<5\pi/4,\ \tt\in \mathbb T \},\\
         W^s(\gamma_n(h))=&\{(\varphi,\tt,J^s(\varphi,\tau),E^s(\varphi,\tau))\colon  \pi/4<\varphi<2\pi+\varphi^{(n)}(\tt),\ \tt\in \mathbb T\}
    \end{split}
\end{equation}
in terms of real-analytic functions $J^{u,s},E^{u,s}$ which are $O(\varepsilon)$ close to $J_h$ and the constant function with value $E_n$ respectively. Finally, the fact that $W^{u,s}$ is Lagrangian implies that $(J\mathrm d\varphi+E\mathrm d\tau)|_{W^{u,s}}$ is closed and the desired conclusion follows. The proof of Proposition \ref{prop:LagrangianW^su} is complete.

To prove Proposition \ref{prop:LagrangianW^su_elliptic} we transform the Hamiltonian $\mathcal{K} _\varepsilon^{(n)}$ to the action angle variables, obtaining \eqref{eq:hamiltonianellipticaa}. This Hamiltonian, when expressed in the variables $(J,E,\varphi,\tt)$  has the form \eqref{eq:averagingelliptic}. Then, applying Proposition \ref{prop:LagrangianW^su}, we obtain the result.

\section{Straightening the unstable foliation}\label{sec:straightening}
In this section we provide the proof of Lemma \ref{lem:straightening}. Namely, we show how to straighten the differential operator 
\begin{equation}\label{eq:distortedop}
{\mathcal L}^{u}[\cdot]=\partial_J \mathcal H_\varepsilon^{(n)}(h)\partial_\varphi [\cdot]+\partial_E\mathcal H^{(n)}_\varepsilon(h)\partial_\tt[\cdot],
\end{equation}
where the functions $\partial_J \mathcal H_\varepsilon^{(n)}(h)$ and $\partial_E \mathcal H_\varepsilon^{(n)}(h)$ are evaluated at
\begin{equation}\label{eq:zu}
z^u(\varphi,\tt)=(\varphi,\tt,\beta_\varphi+\partial_\varphi S^u(h)(\varphi,\tt),\beta_\tt+\partial_\tt S^u(h)(\varphi,\tt)).
\end{equation}
To do so, the first step is to obtain quantitative information on the coefficients of the differential operator $\mathcal L^u$. We do so in  Sections \ref{sec:integrable} and \ref{sec:goodcoordinates}, where we exploit the nearly-integrable structure to construct  a  coordinate system, tailored for the asymptotic analysis of these coefficients. Then,  in Section \ref{sec:straighteningproof} we complete the proof of Lemma \ref{lem:straightening}.

\subsection{The integrable Hamiltonian}\label{sec:integrable}
As in Section \ref{sec:cylinders}, we will take advantage of the fact that the Hamiltonian $\mathcal H_\varepsilon^{(n)}(h)$ can be seen as a fast time-periodic perturbation of an integrable Hamiltonian. Namely (see Section \ref{sec:cylinders}), in the coordinate system defined using the transformation $\phi_n$ in  \eqref{eq:symplecticlatticechange}, we have: 
\[
\mathcal H_\varepsilon^{(n)}(h)\circ\phi_n=H_{\mathrm{av},\varepsilon}^{(n)}+\varepsilon P^{(n)}+\varepsilon^2h\circ\phi_n.
\]
The Hamiltonian $H_{\mathrm{av},\varepsilon}^{(n)}$ possesses a hyperbolic periodic orbit $\gamma^{(n)}_{\mathrm{av}}$ in \eqref{eq:unperturbedperorbits} %$\gamma_{\mathrm{av}}=\{\varphi=0, (J,E)= \tilde a^{(n)}(I_3)\}$. 
The stable and the unstable invariant manifolds of $\gamma^{(n)}_{\mathrm{av}}$ coincide along two homoclinic manifolds (an upper and a lower branch). The upper branch admits a parametrization of the form 
\[
W(\gamma_{\mathrm{av}})=\{(\varphi,\tt,\partial_\varphi S_{\mathrm{av}}(\varphi),E_n)\colon (\varphi,\tt)\in(0,2\pi)\times\mathbb T\}
\]
in terms of some real-analytic function $S_{\mathrm{av}}:(0,2\pi)\to \mathbb R$ satisfying $\partial_\varphi S_{\mathrm{av}}(\varphi)\to 0$ as $\varphi\to 0$ or $\varphi\to 2\pi$ and $E_n=\tilde a^{(n)}_E(I_3)$. In fact $\partial_\varphi S_{\mathrm{av}}(\varphi)= J_h(\varphi)$, given in \eqref{eq:homoclinicpendulum}.
\medskip

As a preparation for the discussion that follows, suppose first that we want to straighten the averaged differential operator
\[
\mathcal L_{\mathrm{av}}[\cdot]=\partial_J H_{\mathrm{av},\varepsilon}^{(n)} \partial_\varphi[\cdot]+\partial_E H_{\mathrm{av},\varepsilon}^{(n)} \partial_\tt[\cdot],
\]
where the coefficients $\partial_J H_{\mathrm{av},\varepsilon}^{(n)}$, $\partial_J H_{\mathrm{av},\varepsilon}^{(n)}$ are evaluated at 
\[
z_{\mathrm{av}}(\varphi,\tt)=(\varphi,\tt, \partial_\varphi S_{\mathrm{av}}(\varphi), E_n).
\]
To do so, we just appeal to the time parametrization
\[
W(\gamma_{\mathrm{av}})=\{(\varphi_{\mathrm{av}}(s),\tt, \partial_\varphi S_{\mathrm{av}}(\varphi_{\mathrm{av}}(s)),E_n)\colon (s,\tt)\in\mathbb R\times\mathbb T\}
\]
in terms of the real-analytic function $\varphi_{\mathrm{av}}$ defined by  
\[
\varphi_{\mathrm{av}}'(s)=\partial_J H_{\mathrm{av},\varepsilon}^{(n)}(\varphi_{\mathrm{av}}(s), \partial_\varphi S_{\mathrm{av}}(\varphi_{\mathrm{av}}(s)),E_n)\qquad\qquad\text{and}\qquad\qquad
\lim_{s\to \pm\infty}\varphi_{\mathrm{av}}(s)=0.
\]
It is not difficult to check  that
\begin{enumerate}
\item $\varphi_{\mathrm{av}}'(s)\neq 0$ for all $s\in\mathbb R$,
    \item 
there exists $\lambda>0$ (depending possibly on $n\in\mathbb N$ but independent of $\varepsilon$) such that 
\[
\varphi_{\mathrm{av}}(s)\asymp \exp( \lambda \sqrt \varepsilon s)\qquad \text{as }s\to -\infty\qquad\qquad\text{and}\qquad\qquad  \varphi_{\mathrm{av}}(s)\asymp 2\pi-\exp( -\lambda \sqrt \varepsilon s)\qquad\text{as }s\to +\infty.
\]
\end{enumerate}
Moreover, as we have seen in Section \ref{sec:cylinders},
\begin{enumerate}[ start=3]
\item the constant
\begin{equation}\label{eq:averagedfreq}
\omega_n:=\partial_{E} H_{\mathrm{av},\varepsilon}^{(n)} (z_{\mathrm{av}}(0,\tt)),
\end{equation}
which gives the frequency of the periodic orbit $\gamma_{\mathrm{av}}$, does not vanish as $\varepsilon\to 0$ (see \eqref{eq:lowerboundfastfreq}).
\end{enumerate}
Define now the coordinate transformation
\begin{align*}
\psi_{\mathrm{av}}:\mathbb R\times\mathbb T&\to (0,2\pi)\times\mathbb R\\
(s,\tt)&\mapsto (\varphi_{\mathrm{av}}(s),\tt).
\end{align*}
Then, the chain rule shows that for any differentiable function $f$ 
\begin{equation}\label{eq:straightenaveraged}
\mathcal L_{\mathrm{av}}[f]\circ\psi_{\mathrm{av}}=\partial_s[f\circ\psi_{\mathrm{av}}]+\omega_n \partial_\tt [f\circ\psi_{\mathrm{av}}].
\end{equation}
That is, the change $\psi_{\mathrm{av}}$ sends the differential operator $\mathcal L_{\mathrm{av}}$ to constant coefficients.
\subsection{A suitable coordinate system}\label{sec:goodcoordinates}
     We now show how to adapt the above discussion to straighten operator \eqref{eq:distortedop}. We  define the coordinate transformation 
\begin{equation}\label{eq:integrablestraightening}
\begin{split}
 \psi_0:\mathbb R\times\mathbb T&\to D^u\\
 (s,\tt)&\mapsto (\varphi_{\mathrm{av}}(s)+\varphi^{(n)}(h)(\tt), \tt).
 \end{split}
\end{equation}
A straightforward application of the chain rule shows that for any differentiable function $f$ (compare with \eqref{eq:straightenaveraged} above)
\begin{equation}\label{eq:pdopauxiliar}
\mathcal L^u[f]\circ\psi_0= \frac{1}{\varphi_{\mathrm{av}}'}\ (A-(\varphi^{(n)}(h))'B)\circ\psi_0 \   \partial_s [f\circ\psi_0]+C\circ\psi_0\  \partial_\tt[f\circ\psi_0]:=\widetilde{\mathcal L}^u[f\circ\psi_0]
\end{equation}
where we have introduced the functions (the notation $z^u$ was introduced in \eqref{eq:zu})
     \[
     A(\varphi,\tt)=\partial_J \mathcal H_\varepsilon^{(n)}(h) (z^u(\varphi,\tt))\qquad\qquad   C(\varphi,\tt)=\partial_E \mathcal H_\varepsilon^{(n)}(h) (z^u(\varphi,\tt))
     \]
     In the following lemma we describe the behavior of the coefficients of the differential operator in the right hand side of \eqref{eq:pdopauxiliar}. We let 
     \[
     \varpi(\tt)=\partial_E \mathcal H_\varepsilon^{(n)}(h) ((\varphi^{(n)}(h)(\tt),\tt,J^{(n)}(h)(\tt),E^{(n)}(h)(\tt))
     \]
     describe the inner dynamics at the periodic orbit $\gamma_n(h)$.
 \begin{lem}\label{lem:prestraightening}
   Define the functions
     \[
    \mathcal A= \frac{1}{\varphi_{\mathrm{av}}'}\ (A-(\varphi^{(n)})'B)\circ\psi_0-1, \qquad\qquad \mathcal C=(C-\varpi)\circ\psi_0
     \]
     Then, for any  $|r|_1=0,1$ and any $(s,\tt)\in(-\infty,1)\times\mathbb T$, we have:
     \[
    |\exp(\lambda \sqrt\varepsilon s) \partial^r \mathcal A(s,\tt)|\lesssim \varepsilon , \qquad\qquad   |\exp(\lambda \sqrt\varepsilon s) \partial^r \mathcal C(s,\tt)|\lesssim \varepsilon.
     \]
 \end{lem}
\begin{proof}
    The proof of this Lemma is obtained by standard arguments and is left to the reader. On the one hand, $O_{\C^1}(\varepsilon)$ smallness comes from the fact that the perturbed invariant manifolds are $O_{\C^1}(\varepsilon)$ close to the homoclinic manifold $W(\gamma_{\mathrm{av}})$ above. On the other hand, exponential decay as $s\to -\infty$ is a consequence of the construction of the coordinate system $\psi_0$ since, as $s\to -\infty$, the points $z^u\circ\psi_0(s,\tt)$ approach the hyperbolic periodic orbit $\gamma_n(h)$.
\end{proof}

 In view of Lemma \ref{lem:prestraightening} we rewrite the right hand side of \eqref{eq:pdopauxiliar} as 
\[
\widetilde{\mathcal L}^u[\cdot]=(1+\mathcal A)\partial_s[\cdot]+(\varpi+\mathcal C)\partial_\tt[\cdot] .
\]
Observe that, moreover, for all $t\in\mathbb T$
\begin{equation}\label{eq:omeganapprox}
\varpi(t)=\omega_n+O_{\C^1}(\varepsilon)
\end{equation}
for $\omega_n\neq 0$ as in \eqref{eq:averagedfreq}. Hence, both coefficients of the operator $\widetilde {\mathcal L}^u$ are close to constant.
\subsection{Proof of Lemma \ref{lem:straightening}}\label{sec:straighteningproof}
We divide the proof into two steps.
\medskip

\noindent\textit{Dynamics at the periodic orbit.} We first observe that as $u\to \infty$
\[
\widetilde {\mathcal L}^u[\cdot]\to \mathcal L_\infty[\cdot]=\partial_s[\cdot]+\varpi \partial_\tt[\cdot]
\]
so we look for a change of variables  of the form 
\[
\psi_1:(s,\tt)\mapsto(s,\tt+\psi_\tt (\tt))
\]
sending $\mathcal L_\infty$ into 
\[
\mathcal L[\cdot]=\partial_s[\cdot]+\tilde\omega_n \partial_\tt[\cdot]
\]
for some $\tilde\omega_n$ to be chosen below. A straightforward computation shows that for any differentiable function $f$ we have $\mathcal L_\infty[f]\circ\psi_1=\mathcal L[f\circ\psi_1]$ if and only if
\[
\psi_t'=\frac{1}{\tilde\omega_n}(\varpi-\tilde\omega_n)\circ\psi_1.
\]
To analyze the solvability of this (nonlinear) equation, it is enough to observe that if we write $t=\tt+\tilde\psi_\tt(\tt)$, the equation above is equivalent to 
\begin{equation}\label{eq:cohomologicalequationcircle}
\tilde\psi_\tt'=\frac{1}{\varpi}(\tilde\omega_n-\varpi).
\end{equation}
Hence, we must choose $\tilde\omega_n$ so the average of the right hand side of \eqref{eq:cohomologicalequationcircle} vanishes, i.e.
\begin{equation}\label{eq:omegainfty}
    \tilde\omega_n=\frac{2\pi}{\int \frac{1}{\varpi(\tt)}\mathrm d\tt}.
\end{equation}
In view of \eqref{eq:omeganapprox} we have that 
\[
\tilde\omega_n=\omega_n+O(\varepsilon)
\]
with $\omega_n$ as in \eqref{eq:averagedfreq}. Summing up, there exists a $O_{\C^1}(\varepsilon)$ close to identity change of variables $\psi_1$ as above which sends $\widetilde {\mathcal L}^u$ into the operator
\[
\widehat{\mathcal L}^u[\cdot]=(1+\widehat {\mathcal A}) \partial_s[\cdot]+(\tilde\omega_n+\widehat{\mathcal C})\partial_\tau[\cdot],
\]
where 
\[
\widehat{\mathcal A}=\mathcal A\circ\psi_1, \qquad\qquad \widehat{\mathcal C}= \frac{\tilde\omega_n}{\varpi\circ\psi_1}\mathcal C\circ\psi_1 .
\]
Observe that for the operator $\widehat {\mathcal L}^u$ both coefficients are given by the sum of a constant  and a function which is $O_{\C^1}(\varepsilon)$ and decays exponentially fast as $s\to -\infty$.
\medskip

\noindent\textit{Dynamics along the unstable foliation.} Having straightened the operator asymptotically for $s\to -\infty$ we now look for a change of variables of the form 
\[
\psi_2:(v,\xi)\mapsto(v+\psi_v(v,\xi),\xi+\psi_\xi(v,\xi)),
\]
which straightens $\widehat{\mathcal L}^u$. A straightforward application of the chain rule shows that for any differentiable function $f$ we have 
\[
\widehat {\mathcal L}^u[f]\circ\psi_1=\ \mathcal L[f\circ\psi_2]
\]
if and only if 
\[
\mathcal L\psi_v=\widehat{\mathcal A}\circ\psi_2\qquad\qquad\text{and}\qquad\qquad\mathcal L\psi_\xi=\widehat{\mathcal C}\circ\psi_2.
\]
This system of differential equations can be recast as the fixed point equation
\begin{equation}\label{eq:fixedpointdiffstraighten}
(\psi_v,\psi_\xi)=(\mathcal G(\widehat{\mathcal A}\circ\psi_2),\mathcal G(\widehat{\mathcal C}\circ\psi_2)),
\end{equation}
where, for any $f$ decaying sufficiently fast as $v\to -\infty$, we have defined 
\[
\mathcal G(f)(v,\xi)=\int_{-\infty}^0f(v+s,\xi+\omega_n s)\mathrm ds.
\]
It is straightforward that \eqref{eq:fixedpointdiffstraighten} admits a unique differentiable solution $|\psi_1|_{\C^1}\lesssim \varepsilon$ decaying exponentially fast as $v\to-\infty$ by means of the Banach fixed point theorem. The proof of Lemma \ref{lem:straightening} is complete.
\bibliographystyle{alpha}
\bibliography{Biblio}

\end{document}